\documentclass{amsart}
\usepackage[T1]{fontenc}
\usepackage[utf8]{inputenc}
\usepackage{etoolbox}
\usepackage{xstring}
\usepackage{lmodern}
\usepackage[english]{babel}
\usepackage{csquotes}
\usepackage{dsfont}
\usepackage{amsmath}
\usepackage{amssymb}
\usepackage{mathtools}
\usepackage{mathrsfs}
\usepackage[normalem]{ulem}
\usepackage[makeroom]{cancel}
\usepackage[dvipsnames]{xcolor}
\usepackage{graphicx}
\usepackage[hypertexnames=false,colorlinks=true,allcolors=black]{hyperref} 
\usepackage{enumitem}
\usepackage{listofitems}
\usepackage{datetime2} 
\usepackage{derivative}
\usepackage{pict2e}
\usepackage{bm}

\numberwithin{equation}{section}
\usepackage[noabbrev,capitalize,english]{cleveref}
\theoremstyle{plain}
  \newtheorem{theoremINTERNAL}{Theorem}[section]
  \newtheorem{lemmaINTERNAL}[theoremINTERNAL]{Lemma}
  \newtheorem{propositionINTERNAL}[theoremINTERNAL]{Proposition}
  \newtheorem{corollaryINTERNAL}[theoremINTERNAL]{Corollary}
\theoremstyle{definition}
  \newtheorem{definitionINTERNAL}[theoremINTERNAL]{Definition}
  \newtheorem{exampleINTERNAL}[theoremINTERNAL]{Example}
  \newtheorem{remarkINTERNAL}[theoremINTERNAL]{Remark}
\crefname{theorem}{Theorem}{Theorems}
\crefname{lemma}{Lemma}{Lemmas}
\crefname{proposition}{Proposition}{Propositions}
\crefname{corollary}{Corollary}{Corollaries}
\crefname{definition}{Definition}{Definitions}
\crefname{example}{Example}{Examples}
\crefname{remark}{Remark}{Remarks}
\crefformat{theorem}{#2Theorem~#1#3}
\crefformat{lemma}{#2Lemma~#1#3}
\crefformat{proposition}{#2Proposition~#1#3}
\crefformat{corollary}{#2Corollary~#1#3}
\crefformat{definition}{#2Definition~#1#3}
\crefformat{example}{#2Example~#1#3}
\crefformat{remark}{#2Remark~#1#3}
\NewDocumentEnvironment{theorem}{O{} o}
{
  \newcommand*{\TEMPtheorem}{#2}
  \ifblank{#1}
  {\theoremINTERNAL}
  {\theoremINTERNAL[#1]}
  \IfNoValueTF{#2}
  {}
  {\label[theorem]{\TEMPtheorem}}
}
{
  \endtheoremINTERNAL
}

\NewDocumentEnvironment{lemma}{O{} o}
{
  \newcommand*{\TEMPlemma}{#2}
  \ifblank{#1}
  {\lemmaINTERNAL}
  {\lemmaINTERNAL[#1]}
  \IfNoValueTF{#2}
  {}
  {\label[lemma]{\TEMPlemma}}
}
{
  \endlemmaINTERNAL
}

\NewDocumentEnvironment{proposition}{O{} o}
{
  \newcommand*{\TEMPproposition}{#2}
  \ifblank{#1}
  {\propositionINTERNAL}
  {\propositionINTERNAL[#1]}
  \IfNoValueTF{#2}
  {}
  {\label[proposition]{\TEMPproposition}}
}
{
  \endpropositionINTERNAL
}

\NewDocumentEnvironment{corollary}{O{} o}
{
  \newcommand*{\TEMPcorollary}{#2}
  \ifblank{#1}
  {\corollaryINTERNAL}
  {\corollaryINTERNAL[#1]}
  \IfNoValueTF{#2}
  {}
  {\label[corollary]{\TEMPcorollary}}
}
{
  \endcorollaryINTERNAL
}

\NewDocumentEnvironment{definition}{O{} o}
{
  \newcommand*{\TEMPdefinition}{#2}
  \ifblank{#1}
  {\definitionINTERNAL}
  {\definitionINTERNAL[#1]}
  \IfNoValueTF{#2}
  {}
  {\label[definition]{\TEMPdefinition}}
}
{
  \enddefinitionINTERNAL
}

\NewDocumentEnvironment{example}{O{} o}
{
  \newcommand*{\TEMPexample}{#2}
  \ifblank{#1}
  {\exampleINTERNAL}
  {\exampleINTERNAL[#1]}
  \IfNoValueTF{#2}
  {}
  {\label[example]{\TEMPexample}}
}
{
  \endexampleINTERNAL
}

\NewDocumentEnvironment{remark}{O{} o}
{
  \newcommand*{\TEMPremark}{#2}
  \ifblank{#1}
  {\remarkINTERNAL}
  {\remarkINTERNAL[#1]}
  \IfNoValueTF{#2}
  {}
  {\label[remark]{\TEMPremark}}
}
{
  \endremarkINTERNAL
}

\let\OLDdelta\delta
\let\OLDgamma\gamma

\let\OLDvarphi\varphi
\let\OLDlambda\lambda

\let\OLDphi\phi
\let\OLDPhi\Phi

\let\OLDrho\rho

\let\OLDnu\nu
\let\OLDsqcup\sqcup
\let\OLDbigsqcup\bigsqcup
\let\OLDDelta\Delta

\undef{\bigsqcup}
\undef{\czero}
\undef{\crM}
\undef{\d}
\undef{\delta}
\undef{\epsr}
\undef{\epszero}
\undef{\f}
\undef{\fr}
\undef{\gamma}
\undef{\h}
\undef{\hzero}
\undef{\k}
\undef{\lambda}
\undef{\m}
\undef{\mr}
\undef{\n}
\undef{\nuBr}
\undef{\p}
\undef{\phi}
\undef{\pK}
\undef{\r}
\undef{\rp}
\undef{\rhoFAMILY}
\undef{\rhoeps}
\undef{\s}
\undef{\sqcup}
\undef{\t}
\undef{\u}
\undef{\uSEQUENCE}
\undef{\uk}
\undef{\un}
\undef{\uFAMILY}
\undef{\ueps}
\undef{\uOm}
\undef{\varphi}
\undef{\varphir}
\undef{\varphiR}
\undef{\x}
\undef{\xp}
\undef{\xzero}
\undef{\y}
\undef{\z}
\undef{\DFAMILY}
\undef{\Deps}
\undef{\Deltah}
\undef{\Deltahu}
\undef{\Deltahueps}
\undef{\Om}
\undef{\OmUNO}
\undef{\OmDUE}
\undef{\E}
\undef{\F}
\undef{\G}
\undef{\H}
\undef{\Hmeno}
\undef{\Hpiu}
\undef{\HKbdryBr}
\undef{\HKbdryE}
\undef{\HKepsbdryBr}
\undef{\HKepsbdryE}
\undef{\IG}
\undef{\IK}
\undef{\IKsym}
\undef{\K}
\undef{\KSEQUENCE}
\undef{\Kk}
\undef{\Kn}
\undef{\KFAMILY}
\undef{\Keps}
\undef{\Kdelta}
\undef{\KpiuKsymMezzi}
\undef{\Ksym}
\undef{\Ks}
\undef{\M}
\undef{\U}
\undef{\UFAMILY}
\undef{\Ueps}
\undef{\PhiFAMILY}
\undef{\Phit}
\undef{\PhiinvFAMILY}
\undef{\Phiinvt}
\undef{\PK}
\undef{\PKpiuKsymMezzi}
\undef{\PKk}
\undef{\PKs}
\undef{\Rpos}
\undef{\VK}
\undef{\VKdelta}
\undef{\X}
\undef{\ForwardDifferenceOperatorSymbol}
\undef{\NonlocalPerimeterSymbol}
\undef{\NonlocalFunctionalSymbol}
\undef{\NonlocalCurvatureSymbol}
\undef{\HausdorffMeasureSymbol}
\undef{\ComplexNumbersSymbol}
\undef{\NaturalNumbersSymbol}
\undef{\IntegerNumbersSymbol}
\undef{\RealNumbersSymbol}
\undef{\LebesgueSpaceSymbol}
\undef{\SobolevSpaceSymbol}
\undef{\LebesgueSpacelocSymbol}
\undef{\SobolevSpacelocSymbol}
\undef{\LipschitzSymbol}
\undef{\BoundedVariationFunctionsSpaceSymbol}
\undef{\BoundedVariationFunctionsSpacelocSymbol}
\undef{\SubdifferentialSymbol}
\undef{\PartialDerivativeSymbol}
\undef{\DistributionalDerivativeSymbol}
\undef{\JacobianMatrixSymbol}
\undef{\ModulusOfDeterminantOFJacobianMatrixSymbol}
\undef{\GradientVectorSymbol}
\undef{\SupportSymbol}
\undef{\almostevSymbol}
\undef{\CharacteristicFunctionSymbol}
\undef{\ConjugateExponentSymbol}
\undef{\BallSymbol}
\undef{\TopologicalBoundarySymbol}
\undef{\NonlocalIsoPBSymbol}
\undef{\NonlocalMaxPBSymbol}

\newcommand{\xmathpalette}[2]{%
  \mathchoice
    {#1\displaystyle\textfont{#2}}
    {#1\textstyle\textfont{#2}}
    {#1\scriptstyle\scriptfont{#2}}
    {#1\scriptscriptstyle\scriptscriptfont{#2}}
}

\makeatletter

\newcommand{\mres@thickness}[1]{\dimexpr1.5\fontdimen8 #13\relax}

\newcommand{\mresA}{\mspace{3mu}{\xmathpalette\mresA@\relax}\mspace{3mu}}
\newcommand{\mresA@}[3]{%
  \begingroup
  \setlength\unitlength{%
    \dimexpr\fontcharht#21`A-0.5\mres@thickness{#2}
  }%
  \raisebox{0.5\dimexpr\mres@thickness{#2}}{%
    \begin{picture}(1,1)
      \roundcap\roundjoin
      \linethickness{\mres@thickness{#2}}
      \polyline(0,1)(0,0)(1,0)
    \end{picture}%
   }%
  \endgroup
}

\newcommand{\mresB}{\mspace{5mu}{\xmathpalette\mresB@\relax}\mspace{5mu}}
\newcommand{\mresB@}[3]{%
  \begingroup
  \setlength\unitlength{%
    \dimexpr0.8\fontcharht#21`A-0.5\mres@thickness{#2}
  }%
  \raisebox{0.5\dimexpr\mres@thickness{#2}}{%
    \begin{picture}(0.5,1)
      \roundcap\roundjoin
      \linethickness{\mres@thickness{#2}}
      \polyline(0,1)(0,0)(0.5,0)
    \end{picture}%
   }%
  \endgroup
}

\newcommand{\mresC}{\mspace{3mu}{\xmathpalette\mresC@\relax}\mspace{3mu}}
\newcommand{\mresC@}[3]{%
  \begingroup
  \setlength\unitlength{%
    \dimexpr0.8\fontcharht#21`A-0.5\mres@thickness{#2}
  }%
  \raisebox{0.5\dimexpr\mres@thickness{#2}}{%
    \begin{picture}(1,1)
      \roundcap\roundjoin
      \linethickness{\mres@thickness{#2}}
      \polyline(0,1)(0,0)(1,0)
    \end{picture}%
   }%
  \endgroup
}
\makeatother

\newcommand*{\newlink}[2]{#2}
\newcommand*{\newtarget}[2]{#2}

\def\Xint#1{\mathchoice
{\XXint\displaystyle\textstyle{#1}}%
{\XXint\textstyle\scriptstyle{#1}}%
{\XXint\scriptstyle\scriptscriptstyle{#1}}%
{\XXint\scriptscriptstyle\scriptscriptstyle{#1}}%
\!\int}
\def\XXint#1#2#3{{\setbox0=\hbox{$#1{#2#3}{\int}$ }
\vcenter{\hbox{$#2#3$ }}\kern-.6\wd0}}

\def\dashint{\Xint-}

\ExplSyntaxOn

\NewDocumentCommand \switchcase { m m o}
  {
    \str_case:nnF { #1 } { #2 } { #3 }
  }

\NewDocumentCommand{\countlist}{m}
    {
        \clist_count:n { #1 }
    }

\ExplSyntaxOff

\DeclarePairedDelimiter{\absDelimiter}{\lvert}{\rvert}
\NewDocumentCommand{\abs}{s m}{
  \IfBooleanTF{#1}
    {\newlink{abs_global}{}\absDelimiter*{\newlink{abs_global}{}#2\newlink{abs_global}{}}\newlink{abs_global}{}}
    {\newlink{abs_global}{\lvert}#2\newlink{abs_global}{\rvert}}
}
\NewDocumentCommand{\EucNorm}{s m}{
  \IfBooleanTF{#1}
    {\newlink{EucNorm_global}{}\absDelimiter*{\newlink{EucNorm_global}{}#2\newlink{EucNorm_global}{}}_{2}\newlink{EucNorm_global}{}}
    {\newlink{EucNorm_global}{\lvert}#2\newlink{EucNorm_global}{\rvert_{2}}}
}
\NewDocumentCommand{\lebd}{s m}{
  \IfBooleanTF{#1}
    {\newlink{lebesgue_global}{}\absDelimiter*{\newlink{lebesgue_global}{}#2\newlink{lebesgue_global}{}}\newlink{lebesgue_global}{}}
    {\newlink{lebesgue_global}{\lvert}#2\newlink{lebesgue_global}{\rvert}}
}
\DeclarePairedDelimiter{\norm}{\|}{\|}
\DeclarePairedDelimiter{\seminorm}{[}{]}
\DeclarePairedDelimiter{\tonde}{(}{)}
\DeclarePairedDelimiter{\quadre}{[}{]}
\DeclarePairedDelimiter{\graffe}{\{}{\}}
\DeclarePairedDelimiter{\dualPairingDelimiter}{\langle}{\rangle}
\NewDocumentCommand{\dualPairing}{s m m}{
  \IfBooleanTF{#1}
    {\newlink{dualPairing_global}{}\dualPairingDelimiter*{\newlink{dualPairing_global}{}#2, #3\newlink{dualPairing_global}{}}\newlink{dualPairing_global}{}}
    {\newlink{dualPairing_global}{}\dualPairingDelimiter{\newlink{dualPairing_global}{}#2, #3\newlink{dualPairing_global}{}}\newlink{dualPairing_global}{}}
}
\NewDocumentCommand{\R}{s o}{%
  \IfNoValueTF{#2}{\RealNumbersSymbol}{{\RealNumbersSymbol}^{#2}}%
}

\NewDocumentCommand{\N}{s o}{%
  \IfNoValueTF{#2}{\NaturalNumbersSymbol}{{\NaturalNumbersSymbol}^{#2}}%
}

\NewDocumentCommand{\Z}{s o}{%
  \IfNoValueTF{#2}{\IntegerNumbersSymbol}{{\IntegerNumbersSymbol}^{#2}}%
}

\NewDocumentCommand{\C}{s o}{%
  \IfNoValueTF{#2}{\ComplexNumbersSymbol}{{\ComplexNumbersSymbol}^{#2}}%
}

\NewDocumentCommand{\Matrices}{m m o}{
  \IfNoValueTF{#3}
  {\newlink{Matrices_global}{\mathbb{M}_{#1\times#2}(\R)}}
  {\newlink{Matrices_global}{\mathbb{M}_{#1\times#2}\tonde*{#3}}}
}
\NewDocumentCommand{\SquareMatrices}{m o}{
  \IfNoValueTF{#2}
  {\newlink{SquareMatrices_global}{\mathbb{M}_{#1}(\R)}}
  {\newlink{SquareMatrices_global}{\mathbb{M}_{#1}\tonde*{#2}}}
}
\NewDocumentCommand{\SymmetricMatrices}{m o}{
  \IfNoValueTF{#2}
  {\newlink{SymmetricMatrices_global}{\mathbb{S}_{#1}(\R)}}
  {\newlink{SymmetricMatrices_global}{\mathbb{S}_{#1}\tonde*{#2}}}
}

\newcommand*{\eps}{\varepsilon}
\DeclareMathOperator{\distance}{dist}
\NewDocumentCommand{\TopologicalBoundary}{o}{
  \IfNoValueTF{#1}
  {\TopologicalBoundarySymbol}
  {\TopologicalBoundarySymbol#1}
}
\NewDocumentCommand{\supp}{o}{
  \IfNoValueTF{#1}
  {\SupportSymbol}
  {\SupportSymbol\tonde*{#1}}
}
\NewDocumentCommand{\Lip}{o}{
  \IfNoValueTF{#1}
  {\LipschitzSymbol}
  {\LipschitzSymbol\tonde*{#1}}
}
\DeclareMathOperator{\loc}{loc}

\DeclareMathOperator{\scalar}{\newlink{scalar_global}{\cdot}}

\DeclareMathOperator{\convolution}{\newlink{convolution_global}{\ast}}

\NewDocumentCommand{\almostev}{o}{
\IfNoValueTF{#1}{
  \almostevSymbol
  }
  {
    #1\text{-}\almostevSymbol
  }  
}
\NewDocumentCommand{\Lspace}{s o o o}{%
  \IfNoValueTF{#2}{%
    \LebesgueSpaceSymbol%
  }{%
    \IfNoValueTF{#3}{%
    \LebesgueSpaceSymbol^{#2}%
    }{%
        \readlist\temp{#3}%
        \IfNoValueTF{#4}{
          \ifnumequal{\countlist{#3}}{1}{\LebesgueSpaceSymbol^{#2}\tonde*{\temp[1]}}{}%
          \ifnumequal{\countlist{#3}}{2}{\LebesgueSpaceSymbol^{#2}\tonde*{\temp[1],\temp[2]}}{}%
          \ifnumequal{\countlist{#3}}{3}{\LebesgueSpaceSymbol^{#2}\tonde*{\temp[1],\temp[2],\temp[3]}}{}%
        }{%
          \ifnumequal{\countlist{#3}}{1}{\LebesgueSpaceSymbol^{#2}\tonde*{\temp[1];#4}}{}%
          \ifnumequal{\countlist{#3}}{2}{\LebesgueSpaceSymbol^{#2}\tonde*{\temp[1],\temp[2];#4}}{}%
          \ifnumequal{\countlist{#3}}{3}{\LebesgueSpaceSymbol^{#2}\tonde*{\temp[1],\temp[2],\temp[3];#4}}{}%
        }
    }
  }
}
\NewDocumentCommand{\Llocspace}{s o o o}{%
  \IfNoValueTF{#2}{%
    \LebesgueSpacelocSymbol%
  }{%
    \IfNoValueTF{#3}{%
    \LebesgueSpacelocSymbol^{#2}%
    }{%
        \readlist\temp{#3}%
        \IfNoValueTF{#4}{
          \ifnumequal{\countlist{#3}}{1}{\LebesgueSpacelocSymbol^{#2}\tonde*{\temp[1]}}{}%
          \ifnumequal{\countlist{#3}}{2}{\LebesgueSpacelocSymbol^{#2}\tonde*{\temp[1],\temp[2]}}{}%
          \ifnumequal{\countlist{#3}}{3}{\LebesgueSpacelocSymbol^{#2}\tonde*{\temp[1],\temp[2],\temp[3]}}{}%
        }{%
          \ifnumequal{\countlist{#3}}{1}{\LebesgueSpacelocSymbol^{#2}\tonde*{\temp[1];#4}}{}%
          \ifnumequal{\countlist{#3}}{2}{\LebesgueSpacelocSymbol^{#2}\tonde*{\temp[1],\temp[2];#4}}{}%
          \ifnumequal{\countlist{#3}}{3}{\LebesgueSpacelocSymbol^{#2}\tonde*{\temp[1],\temp[2],\temp[3];#4}}{}%
        }
    }
  }
}
\NewDocumentCommand{\Wspace}{s o o o}{%
  \IfNoValueTF{#2}{%
    \SobolevSpaceSymbol%
  }{%
    \IfNoValueTF{#3}{%
    \SobolevSpaceSymbol^{#2}%
    }{%
        \readlist\temp{#3}%
        \IfNoValueTF{#4}{
          \ifnumequal{\countlist{#3}}{1}{\SobolevSpaceSymbol^{#2}\tonde*{\temp[1]}}{}%
          \ifnumequal{\countlist{#3}}{2}{\SobolevSpaceSymbol^{#2}\tonde*{\temp[1],\temp[2]}}{}%
          \ifnumequal{\countlist{#3}}{3}{\SobolevSpaceSymbol^{#2}\tonde*{\temp[1],\temp[2],\temp[3]}}{}%
        }{%
          \ifnumequal{\countlist{#3}}{1}{\SobolevSpaceSymbol^{#2}\tonde*{\temp[1];#4}}{}%
          \ifnumequal{\countlist{#3}}{2}{\SobolevSpaceSymbol^{#2}\tonde*{\temp[1],\temp[2];#4}}{}%
          \ifnumequal{\countlist{#3}}{3}{\SobolevSpaceSymbol^{#2}\tonde*{\temp[1],\temp[2],\temp[3];#4}}{}%
        }
    }
  }
}
\NewDocumentCommand{\Wlocspace}{s o o o}{%
  \IfNoValueTF{#2}{%
    \SobolevSpacelocSymbol%
  }{%
    \IfNoValueTF{#3}{%
    \SobolevSpacelocSymbol^{#2}%
    }{%
        \readlist\temp{#3}%
        \IfNoValueTF{#4}{
          \ifnumequal{\countlist{#3}}{1}{\SobolevSpacelocSymbol^{#2}\tonde*{\temp[1]}}{}%
          \ifnumequal{\countlist{#3}}{2}{\SobolevSpacelocSymbol^{#2}\tonde*{\temp[1],\temp[2]}}{}%
          \ifnumequal{\countlist{#3}}{3}{\SobolevSpacelocSymbol^{#2}\tonde*{\temp[1],\temp[2],\temp[3]}}{}%
        }{%
          \ifnumequal{\countlist{#3}}{1}{\SobolevSpacelocSymbol^{#2}\tonde*{\temp[1];#4}}{}%
          \ifnumequal{\countlist{#3}}{2}{\SobolevSpacelocSymbol^{#2}\tonde*{\temp[1],\temp[2];#4}}{}%
          \ifnumequal{\countlist{#3}}{3}{\SobolevSpacelocSymbol^{#2}\tonde*{\temp[1],\temp[2],\temp[3];#4}}{}%
        }
    }
  }
}
\NewDocumentCommand{\BVspace}{s o o}{%
  \IfNoValueTF{#2}{%
    \BoundedVariationFunctionsSpaceSymbol%
  }{%
    \IfNoValueTF{#3}{%
      \BoundedVariationFunctionsSpaceSymbol\tonde*{#2}%
    }{%
      \BoundedVariationFunctionsSpaceSymbol\tonde*{#2;#3}
    }
  }
}
\NewDocumentCommand{\BVlocspace}{s o o}{%
  \IfNoValueTF{#2}{%
    \BoundedVariationFunctionsSpacelocSymbol%
  }{%
    \IfNoValueTF{#3}{%
      \BoundedVariationFunctionsSpacelocSymbol\tonde*{#2}%
    }{%
      \BoundedVariationFunctionsSpacelocSymbol\tonde*{#2;#3}
    }
  }
}
\NewDocumentCommand{\nonlocalBVspace}{o o o}{%
  \IfNoValueTF{#2}{%
      \newlink{nonlocalBVspace_global}{\BoundedVariationFunctionsSpaceSymbol*^{#1}}%
  }{%
    \IfNoValueTF{#3}{%
      \newlink{nonlocalBVspace_global}{\BoundedVariationFunctionsSpaceSymbol*^{#1}}\tonde*{#2}%
    }{%
      \newlink{nonlocalBVspace_global}{\BoundedVariationFunctionsSpaceSymbol*^{#1}}\tonde*{#2;#3}
    }
  }
}
\NewDocumentCommand{\nonlocalBVlocspace}{o o o}{%
  \IfNoValueTF{#2}{%
      \newlink{nonlocalBVlocspace_global}{\BoundedVariationFunctionsSpacelocSymbol*^{#1}}%
  }{%
    \IfNoValueTF{#3}{%
      \newlink{nonlocalBVlocspace_global}{\BoundedVariationFunctionsSpacelocSymbol*^{#1}}\tonde*{#2}%
    }{%
      \newlink{nonlocalBVlocspace_global}{\BoundedVariationFunctionsSpacelocSymbol*^{#1}}\tonde*{#2;#3}
    }
  }
}
\NewDocumentCommand{\nonlocalSobolevspace}{s o o o}{%
  \IfNoValueTF{#2}{%
    \SobolevSpaceSymbol%
  }{%
    \IfNoValueTF{#3}{%
    \SobolevSpaceSymbol^{#2}%
    }{%
        \readlist\temp{#3}%
        \IfNoValueTF{#4}{
          \ifnumequal{\countlist{#3}}{1}{\SobolevSpaceSymbol^{#2}\tonde*{\temp[1]}}{}%
          \ifnumequal{\countlist{#3}}{2}{\SobolevSpaceSymbol^{#2}\tonde*{\temp[1],\temp[2]}}{}%
          \ifnumequal{\countlist{#3}}{3}{\SobolevSpaceSymbol^{#2}\tonde*{\temp[1],\temp[2],\temp[3]}}{}%
        }{%
          \ifnumequal{\countlist{#3}}{1}{\SobolevSpaceSymbol^{#2}\tonde*{\temp[1];#4}}{}%
          \ifnumequal{\countlist{#3}}{2}{\SobolevSpaceSymbol^{#2}\tonde*{\temp[1],\temp[2];#4}}{}%
          \ifnumequal{\countlist{#3}}{3}{\SobolevSpaceSymbol^{#2}\tonde*{\temp[1],\temp[2],\temp[3];#4}}{}%
        }
    }
  }
}
\NewDocumentCommand{\nonlocalSobolevlocspace}{s o o o}{%
  \IfNoValueTF{#2}{%
    \SobolevSpacelocSymbol%
  }{%
    \IfNoValueTF{#3}{%
    \SobolevSpacelocSymbol^{#2}%
    }{%
        \readlist\temp{#3}%
        \IfNoValueTF{#4}{
          \ifnumequal{\countlist{#3}}{1}{\SobolevSpacelocSymbol^{#2}\tonde*{\temp[1]}}{}%
          \ifnumequal{\countlist{#3}}{2}{\SobolevSpacelocSymbol^{#2}\tonde*{\temp[1],\temp[2]}}{}%
          \ifnumequal{\countlist{#3}}{3}{\SobolevSpacelocSymbol^{#2}\tonde*{\temp[1],\temp[2],\temp[3]}}{}%
        }{%
          \ifnumequal{\countlist{#3}}{1}{\SobolevSpacelocSymbol^{#2}\tonde*{\temp[1];#4}}{}%
          \ifnumequal{\countlist{#3}}{2}{\SobolevSpacelocSymbol^{#2}\tonde*{\temp[1],\temp[2];#4}}{}%
          \ifnumequal{\countlist{#3}}{3}{\SobolevSpacelocSymbol^{#2}\tonde*{\temp[1],\temp[2],\temp[3];#4}}{}%
        }
    }
  }
}
\NewDocumentCommand{\Cspace}{s o o o}{%
  \providecommand*{\CspaceTag}{\newlink{Cspace_global}{C}}%
  \IfBooleanT{#1}{\renewcommand*{\CspaceTag}{C}}%
  \IfNoValueTF{#2}{%
    \CspaceTag%
  }{%
    \IfNoValueTF{#3}{%
    \CspaceTag^{#2}%
    }{%
        \readlist\temp{#3}%
        \IfNoValueTF{#4}{
          \ifnumequal{\countlist{#3}}{1}{\CspaceTag^{#2}\tonde*{\temp[1]}}{}%
        }{%
          \ifnumequal{\countlist{#3}}{1}{\CspaceTag^{#2}\tonde*{\temp[1];#4}}{}%
        }
    }
  }
}
\NewDocumentCommand{\CcptSpace}{s o o o}{%
  \providecommand*{\CcptSpaceTag}{\newlink{CcptSpace_global}{C}}%
  \IfBooleanT{#1}{\renewcommand*{\CcptSpaceTag}{C}}%
  \IfNoValueTF{#2}{%
    \CcptSpaceTag_c%
  }{%
    \IfNoValueTF{#3}{%
    \CcptSpaceTag^{#2}_c%
    }{%
        \readlist\temp{#3}%
        \IfNoValueTF{#4}{
          \ifnumequal{\countlist{#3}}{1}{\CcptSpaceTag^{#2}_c\tonde*{\temp[1]}}{}%
        }{%
          \ifnumequal{\countlist{#3}}{1}{\CcptSpaceTag^{#2}_c\tonde*{\temp[1];#4}}{}%
        }
    }
  }
}
\NewDocumentCommand{\SchSpace}{s o o}{%
  \providecommand*{\SchSpaceTag}{\newlink{SchSpace_global}{\mathcal{S}}}%
  \IfBooleanT{#1}{\renewcommand*{\SchSpaceTag}{\mathcal{S}}}%
  \IfNoValueTF{#2}{%
    \SchSpaceTag%
  }{%
    \readlist\temp{#2}%
    \IfNoValueTF{#3}{
      \ifnumequal{\countlist{#2}}{1}{\SchSpaceTag\tonde*{\temp[1]}}{}%
    }{%
      \ifnumequal{\countlist{#2}}{1}{\SchSpaceTag\tonde*{\temp[1];#3}}{}%
    }
  }
}

\NewDocumentCommand{\PositiveMeasures}{s o}{%
\providecommand*{\PositiveMeasuresTag}{}%
  \IfBooleanTF{#1}%
    {\renewcommand*{\PositiveMeasuresTag}{\mathcal{M}^{+}}}%
    {\renewcommand*{\PositiveMeasuresTag}{\newlink{PositiveMeasures_global}{\mathcal{M}^{+}}}}
  \IfNoValueTF{#2}%
    {\PositiveMeasuresTag}%
    {\PositiveMeasuresTag\tonde*{#2}}%
}

\NewDocumentCommand{\PositiveBoundedMeasures}{s o}{%
\providecommand*{\PositiveBoundedMeasuresTag}{}%
  \IfBooleanTF{#1}%
    {\renewcommand*{\PositiveBoundedMeasuresTag}{\mathcal{M}^{+}_{b}}}%
    {\renewcommand*{\PositiveBoundedMeasuresTag}{\newlink{PositiveBoundedMeasures_global}{\mathcal{M}^{+}_{b}}}}
  \IfNoValueTF{#2}%
    {\PositiveBoundedMeasuresTag}%
    {\PositiveBoundedMeasuresTag\tonde*{#2}}%
}
\NewDocumentCommand{\SignedMeasures}{s o}{%
\providecommand*{\SignedMeasuresTag}{}%
  \IfBooleanTF{#1}%
    {\renewcommand*{\SignedMeasuresTag}{\mathcal{M}_{b}}}%
    {\renewcommand*{\SignedMeasuresTag}{\newlink{SignedMeasures_global}{\mathcal{M}_{b}}}}
  \IfNoValueTF{#2}%
    {\SignedMeasuresTag}%
    {\SignedMeasuresTag\tonde*{#2;\R}}%
}
\NewDocumentCommand{\VectorMeasures}{s o o}{%
\providecommand*{\VectorMeasuresTag}{}%
  \IfBooleanTF{#1}%
    {\renewcommand*{\VectorMeasuresTag}{\mathcal{M}_{b}}}%
    {\renewcommand*{\VectorMeasuresTag}{\newlink{VectorMeasures_global}{\mathcal{M}_{b}}}}
  \IfNoValueTF{#2}%
    {\VectorMeasuresTag}%
    {\VectorMeasuresTag\tonde*{#2;#3}}%
}

\NewDocumentCommand{\Borel}{s o}{%
\newcommand*{\BorelTag}{\newlink{BorelSets_global}{\mathcal{B}}}%
  \IfBooleanT{#1}%
    {\renewcommand*{\BorelTag}{\mathcal{B}}}%
  \IfNoValueTF{#2}%
    {\BorelTag}%
    {\BorelTag\tonde*{#2}}%
}

\NewDocumentCommand{\dirac}{s o o}{%
    \providecommand*{\diractag}{}%
    \IfBooleanTF{#1}
    {\renewcommand*{\diractag}{\delta}}%
    {\renewcommand*{\diractag}{\newlink{dirac_global}{\delta}}}%
    \IfNoValueTF{#2}%
    {\diractag}%
    {%
        \ifstrempty{#2}
        {%
            \IfNoValueTF{#3}%
            {%
                \diractag%
            }%
            {%
                \diractag\tonde*{#3}%
            }%
        }%
        {%
            \IfNoValueTF{#3}%
            {%
                {\diractag}_{#2}}%
            {%
                {\diractag}_{#2}\tonde*{#3}%
            }%
        }%
    }%
}
\NewDocumentCommand{\lebesgue}{s m o}{%
  \providecommand*{\lebesguetag}{}%
  \IfBooleanTF{#1}
  {\renewcommand*{\lebesguetag}{\mathcal{L}}}%
  {\renewcommand*{\lebesguetag}{\newlink{lebesgue_global}{\mathcal{L}}}}%
  \IfNoValueTF{#3}%
  {%
      {\lebesguetag}^{#2}}%
  {%
      {\lebesguetag}^{#2}\tonde*{#3}%
  }%
}
\newcommand*{\integral}[1]{\int_{#1}}
\NewDocumentCommand{\integralmean}{m o}{
  \IfNoValueTF{#2}
  {\newlink{integralmean1_global}{\dashint_{#1}}}
  {\newlink{integralmean2_global}{\dashint_{#1}^{#2}}}
}

\newcommand*{\de}{\mathrm{d}}
\newcommand*{\integralde}{\thinspace\mathrm{d}}

\newcommand*{\unmezzo}{\frac{1}{2}}

\newcommand*{\BallRadiusCenter}[2]{{\BallSymbol}_{#1}({#2})}
\newcommand*{\BallVolumeCenter}[2]{{\BallSymbol}^{#1}({#2})}
\newcommand*{\BallRadius}[1]{{\BallSymbol}_{#1}}
\newcommand*{\BallVolume}[1]{{\BallSymbol}^{#1}}
\newcommand*{\comp}[1]{{#1}^{\newlink{complement_global}{c}}} 

\newtheorem{innerclaim}{Claim}
\newtheorem*{innerclaimNotNumbered}{Claim}
\NewDocumentEnvironment{claim}{o}
{
  \IfNoValueTF{#1}
  {\innerclaimNotNumbered}
  {\renewcommand\theinnerclaim{#1}\innerclaim}
}
{
  \IfNoValueTF{#1}
  {\endinnerclaimNotNumbered}
  {\endinnerclaim}
}

\NewDocumentCommand{\FourierTransform}{s o o}{%
\providecommand*{\FourierTransformTag}{\newlink{FourierTransform_global}{\mathfrak{F}}}%
  \IfBooleanT{#1}%
    {\renewcommand*{\FourierTransformTag}{\mathfrak{F}}}%
  \IfNoValueTF{#2}%
  {%
    \IfNoValueTF{#3}
    {\FourierTransformTag}
    {{\FourierTransformTag}_{#3}}
  }%
  {%
    \IfNoValueTF{#3}
    {\FourierTransformTag\quadre*{#2}}
    {{\FourierTransformTag}_{#3}\quadre*{#2}}
  }%
}

\newcommand*{\FourierHat}[1]%
{%
  \newlink{FourierTransform_global}{\widehat{#1}}
}%

\NewDocumentCommand{\Distrib}{s m o}{
  \providecommand*{\Distribtag}{}%
    \IfBooleanTF{#1}
    {\renewcommand*{\Distribtag}{T}}%
    {\renewcommand*{\Distribtag}{\newlink{Distrib_global}{T}}}%
    \IfNoValueTF{#3}%
    {\Distribtag\tonde*{#2}}
    {\dualPairing{\Distribtag\tonde*{#2}}{#3}}
}

\ProvideDocumentCommand{\pdv}{o m m}
{\url{https://ctan.mirror.garr.it/mirrors/ctan/macros/latex/contrib/derivative/derivative.pdf}}
\ProvideDocumentCommand{\odv}{o m m}
{\url{https://ctan.mirror.garr.it/mirrors/ctan/macros/latex/contrib/derivative/derivative.pdf}}
\NewDocumentCommand{\cmpn}{s m m}
{
  \IfBooleanTF{#1}
  {{\tonde*{#2}}_{#3}}
  {{#2}_{#3}}
}
\NewDocumentCommand{\posPart}{s m}{
  \providecommand*{\posParttag}{}%
    \IfBooleanTF{#1}
    {\renewcommand*{\posParttag}{+}}%
    {\renewcommand*{\posParttag}{\newlink{posPart_global}{+}}}%
    {#2}^{\posParttag}
}
\NewDocumentCommand{\negPart}{s m}{
  \providecommand*{\negParttag}{}%
    \IfBooleanTF{#1}
    {\renewcommand*{\negParttag}{-}}%
    {\renewcommand*{\negParttag}{\newlink{negPart_global}{-}}}%
    {#2}^{\negParttag}
}
\NewDocumentCommand{\posPartMeasure}{s m}{
  \providecommand*{\posPartMeasuretag}{}%
    \IfBooleanTF{#1}
    {\renewcommand*{\posPartMeasuretag}{+}}%
    {\renewcommand*{\posPartMeasuretag}{\newlink{posPartMeasure_global}{+}}}%
    {#2}^{\posPartMeasuretag}
}
\NewDocumentCommand{\negPartMeasure}{s m}{
  \providecommand*{\negPartMeasuretag}{}%
    \IfBooleanTF{#1}
    {\renewcommand*{\negPartMeasuretag}{-}}%
    {\renewcommand*{\negPartMeasuretag}{\newlink{negPartMeasure_global}{-}}}%
    {#2}^{\negPartMeasuretag}
}
\NewDocumentCommand{\TVmeas}{m o}{
  \IfNoValueTF{#2}
  {\newlink{TVmeas_global}{\lvert}#1\newlink{TVmeas_global}{\rvert}}
  {\newlink{TVmeas_global}{\lvert}#1\newlink{TVmeas_global}{\rvert}\tonde*{#2}}
}
\NewDocumentCommand{\sbdif}{o o}{
  \IfNoValueTF{#1}{
    \SubdifferentialSymbol
  }{
    \IfNoValueTF{#2}
    {\SubdifferentialSymbol#1}
    {\SubdifferentialSymbol#1\tonde*{#2}}
  }
}
\NewDocumentCommand{\intervallo}{m m m}{
  \providecommand*{\leftdelimiter}{undef}
  \providecommand*{\rightdelimiter}{undef}
  \ifstrequal{#1}{()}{(#2,#3)}{}
  \ifstrequal{#1}{[]}{[#2,#3]}{}
  \ifstrequal{#1}{(]}{(#2,#3]}{}
  \ifstrequal{#1}{[)}{[#2,#3)}{}
}
\NewDocumentCommand{\J}{m o}{
    \IfNoValueTF{#2}
        {\ModulusOfDeterminantOFJacobianMatrixSymbol_{#1}}
        {\ModulusOfDeterminantOFJacobianMatrixSymbol_{#1}\tonde*{#2}}
}
\NewDocumentCommand{\CharFun}{m o}{
    \IfNoValueTF{#2}
        {\CharacteristicFunctionSymbol_{#1}}
        {\CharacteristicFunctionSymbol_{#1}\tonde*{#2}}
}
\newcommand*{\conjexp}[1]%
{%
  {#1}^{\ConjugateExponentSymbol}%
}
\newcommand*{\closure}[1]{\overline{#1}}
\newcommand\restr[2]{{
  \left.\kern-\nulldelimiterspace 
  #1 
  \vphantom{\big|} 
  \right|_{#2} 
  }}
\newcommand*{\NonlocalIsopPBKernelVolume}[2]{\tonde*{\hyperlink{NonlocalIsopPB}{{\NonlocalIsoPBSymbol}_{#1,#2}}}}
\newcommand*{\NonlocalMaxPBKernelVolume}[2]{\tonde*{\hyperlink{NonlocalMaxPB}{{\NonlocalMaxPBSymbol}_{#1,#2}}}}
\newcommand*{\weaklyto}{\rightharpoonup}
\newcommand*{\rearr}[1]{{#1}^{*}}
\NewDocumentCommand{\PKresOm}{o}
{ 
  \IfNoValueTF{#1}
    {\NonlocalPerimeterSymbol_{\K}\tonde*{.;\Om}}
    {\NonlocalPerimeterSymbol_{\K}\tonde*{#1;\Om}}
}
\NewDocumentCommand{\PKresRd}{o}
{ 
  \IfNoValueTF{#1}
    {\NonlocalPerimeterSymbol_{\K}\tonde*{.;\Rd}}
    {\NonlocalPerimeterSymbol_{\K}\tonde*{#1;\Rd}}
}
\newcommand*{\NtsOne}{\textup{(}\hyperref[H:Nts]{$\mathrm{Nts_1}$}\textup{)}}
\newcommand*{\Rd}{\R[\d]}
\newcommand*{\RdmenoOrigine}{\Rd\setminus\graffe{0}}

\newcommand*{\OmOm}{\Om\times\Om}
\newcommand*{\LpRd}{\Lspace[\p][\Rd]}
\newcommand*{\LpOm}{\Lspace[\p][\Om]}

\newcommand*{\LoneRd}{\Lspace[1][\Rd]}
\newcommand*{\LoneOm}{\Lspace[1][\Om]}
\newcommand*{\LonelocRd}{\Llocspace[1][\Rd]}

\newcommand*{\LinftyRd}{\Lspace[\infty][\Rd]}

\newcommand*{\normLoneRd}[1]{\norm*{#1}_{\LoneRd}}
\newcommand*{\normLpRd}[1]{\norm*{#1}_{\LpRd}}

\newcommand*{\intRd}{\integral{\Rd}}

\newcommand*{\intOmOm}{\integral{\OmOm}}
\newcommand*{\CcinftyRd}{\CcptSpace[\infty][\Rd]}
\newcommand*{\CinftyRd}{\Cspace[\infty][\Rd]}

\newcommand*{\BVK}{\nonlocalBVspace[\K][\Rd]}
\newcommand*{\WKp}{\nonlocalSobolevspace[\K,\p][\Rd]}
\newcommand*{\WKone}{\nonlocalSobolevspace[\K,1][\Rd]}
\newcommand*{\WKpOm}{\nonlocalSobolevspace[\K,\p][\Om]}

\newcommand*{\BVKk}{\nonlocalBVspace[\Kk][\Rd]}
\newcommand*{\BVKpiuKsymMezzi}{\nonlocalBVspace[\KpiuKsymMezzi][\Rd]}

\newcommand*{\Bdelta}{\BallRadius{\delta}}
\newcommand*{\Br}{\BallRadius{\r}}
\newcommand*{\Brp}{\BallRadius{\rp}}
\newcommand*{\BR}{\BallRadius{\Rpos}}
\newcommand*{\Beps}{\BallRadius{\eps}}
\NewDocumentCommand{\seminormBVK}{o}{
\IfNoValueTF{#1}
{\seminorm{.}_{\BVK}}
{\seminorm{#1}_{\BVK}}    
}
\NewDocumentCommand{\normBVK}{o}{
\IfNoValueTF{#1}
{\norm{.}_{\BVK}}
{\norm{#1}_{\BVK}}    
}
\NewDocumentCommand{\seminormBVKk}{o}{
\IfNoValueTF{#1}
{\seminorm{.}_{\BVKk}}
{\seminorm{#1}_{\BVKk}}    
}
\NewDocumentCommand{\seminormWKp}{o}{
\IfNoValueTF{#1}
{\seminorm{.}_{\WKp}}
{\seminorm{#1}_{\WKp}}    
}
\NewDocumentCommand{\seminormWKone}{o}{
\IfNoValueTF{#1}
{\seminorm{.}_{\WKone}}
{\seminorm{#1}_{\WKone}}    
}
\NewDocumentCommand{\seminormWKpOm}{o}{
\IfNoValueTF{#1}
{\seminorm{.}_{\WKpOm}}
{\seminorm{#1}_{\WKpOm}}    
}
\NewDocumentCommand{\normWKp}{o}{
\IfNoValueTF{#1}
{\norm{.}_{\WKp}}
{\norm{#1}_{\WKp}}    
}
\NewDocumentCommand{\normWKpOm}{o}{
\IfNoValueTF{#1}
{\norm{.}_{\WKpOm}}
{\norm{#1}_{\WKpOm}}    
}
\NewDocumentCommand{\seminormBVKpiuKsymMezzi}{o}{
\IfNoValueTF{#1}
{\seminorm{.}_{\BVKpiuKsymMezzi}}
{\seminorm{#1}_{\BVKpiuKsymMezzi}}    
}
\newcommand*{\Ir}{\intervallo{()}{-\r}{\r}}

\newcommand*{\bdry}[1]{\TopologicalBoundary[#1]}
\newcommand*{\Hdmenouno}{\HausdorffMeasureSymbol^{\d-1}}

\newcommand*{\intrinfty}{\int_{\r}^{+\infty}}

\newcommand*{\intmenoinftymenor}{\int_{-\infty}^{-\r}}
\newcommand*{\intmenorr}{\int_{-\r}^{\r}}
\newcommand*{\intmenoMmenor}{\int_{-\M}^{-\r}}

\newcommand*{\raggioGiusto}{\r/\sqrt[\d]{2}}
\newcommand*{\palluno}{\BallRadiusCenter{\eps}{\xzero/2}}
\newcommand*{\palldue}{\BallRadiusCenter{\eps}{-\xzero/2}}
\newcommand*{\limepstozeropiu}{\lim_{\eps\to0^+}}
\newcommand*{\unoinfinitoescluso}{\intervallo{[)}{1}{+\infty}}
\newcommand*{\Bm}{\BallVolume{\m}}
\NewDocumentCommand{\almostevSymbol}{s}{%
    \providecommand*{\almostevSymboltag}{}%
    \IfBooleanTF{#1}
    {\renewcommand*{\almostevSymboltag}{\operatorname{a.e.}}}%
    {\renewcommand*{\almostevSymboltag}{\newlink{almostevSymbol_global}{\operatorname{a.e.}}}}%
    \almostevSymboltag%
}
\NewDocumentCommand{\TARGETalmostevSymbol}{o}%
{%
    \IfNoValueTF{#1}%
        {\newtarget{almostevSymbol_global}{\almostevSymbol}}%
        {\newtarget{almostevSymbol_global}{#1}}%
}
\NewDocumentCommand{\BallSymbol}{s}{%
    \providecommand*{\BallSymboltag}{}%
    \IfBooleanTF{#1}
    {\renewcommand*{\BallSymboltag}{B}}%
    {\renewcommand*{\BallSymboltag}{\newlink{BallSymbol_global}{B}}}%
    \BallSymboltag%
}
\NewDocumentCommand{\TARGETBallSymbol}{o}%
{%
    \IfNoValueTF{#1}%
        {\newtarget{BallSymbol_global}{\BallSymbol}}%
        {\newtarget{BallSymbol_global}{#1}}%
}

\NewDocumentCommand{\bigsqcup}{s}{%
    \providecommand*{\bigsqcuptag}{}%
    \IfBooleanTF{#1}
    {\renewcommand*{\bigsqcuptag}{\OLDbigsqcup}}%
    {\renewcommand*{\bigsqcuptag}{\newlink{bigsqcup_global}{\OLDbigsqcup}}}%
    \bigsqcuptag%
}
\NewDocumentCommand{\TARGETbigsqcup}{o}%
{%
    \IfNoValueTF{#1}%
        {\newtarget{bigsqcup_global}{\bigsqcup}}%
        {\newtarget{bigsqcup_global}{#1}}%
}
\NewDocumentCommand{\BoundedVariationFunctionsSpacelocSymbol}{s}{%
    \providecommand*{\BoundedVariationFunctionsSpacelocSymboltag}{}%
    \IfBooleanTF{#1}
    {\renewcommand*{\BoundedVariationFunctionsSpacelocSymboltag}{{\BoundedVariationFunctionsSpaceSymbol}_{\loc}}}%
    {\renewcommand*{\BoundedVariationFunctionsSpacelocSymboltag}{\newlink{OFF_global}{{\BoundedVariationFunctionsSpaceSymbol}_{\loc}}}}%
    \BoundedVariationFunctionsSpacelocSymboltag%
}
\NewDocumentCommand{\TARGETBoundedVariationFunctionsSpacelocSymbol}{o}%
{%
    \IfNoValueTF{#1}%
        {\newtarget{OFF_global}{\BoundedVariationFunctionsSpacelocSymbol}}%
        {\newtarget{OFF_global}{#1}}%
}
\NewDocumentCommand{\BoundedVariationFunctionsSpaceSymbol}{s}{%
    \providecommand*{\BoundedVariationFunctionsSpaceSymboltag}{}%
    \IfBooleanTF{#1}
    {\renewcommand*{\BoundedVariationFunctionsSpaceSymboltag}{BV}}%
    {\renewcommand*{\BoundedVariationFunctionsSpaceSymboltag}{\newlink{BoundedVariationFunctionsSpaceSymbol_global}{BV}}}%
    \BoundedVariationFunctionsSpaceSymboltag%
}
\NewDocumentCommand{\TARGETBoundedVariationFunctionsSpaceSymbol}{o}%
{%
    \IfNoValueTF{#1}%
        {\newtarget{BoundedVariationFunctionsSpaceSymbol_global}{\BoundedVariationFunctionsSpaceSymbol}}%
        {\newtarget{BoundedVariationFunctionsSpaceSymbol_global}{#1}}%
}
\NewDocumentCommand{\CharacteristicFunctionSymbol}{s}{%
    \providecommand*{\CharacteristicFunctionSymboltag}{}%
    \IfBooleanTF{#1}
    {\renewcommand*{\CharacteristicFunctionSymboltag}{\mathds{1}}}%
    {\renewcommand*{\CharacteristicFunctionSymboltag}{\newlink{CharacteristicFunctionSymbol_global}{\mathds{1}}}}%
    \CharacteristicFunctionSymboltag%
}
\NewDocumentCommand{\TARGETCharacteristicFunctionSymbol}{o}%
{%
    \IfNoValueTF{#1}%
        {\newtarget{CharacteristicFunctionSymbol_global}{\CharacteristicFunctionSymbol}}%
        {\newtarget{CharacteristicFunctionSymbol_global}{#1}}%
}
\NewDocumentCommand{\ComplexNumbersSymbol}{s}{%
    \providecommand*{\ComplexNumbersSymboltag}{}%
    \IfBooleanTF{#1}
    {\renewcommand*{\ComplexNumbersSymboltag}{\mathbb{C}}}%
    {\renewcommand*{\ComplexNumbersSymboltag}{\newlink{ComplexNumbersSymbol_global}{\mathbb{C}}}}%
    \ComplexNumbersSymboltag%
}
\NewDocumentCommand{\TARGETComplexNumbersSymbol}{o}%
{%
    \IfNoValueTF{#1}%
        {\newtarget{ComplexNumbersSymbol_global}{\ComplexNumbersSymbol}}%
        {\newtarget{ComplexNumbersSymbol_global}{#1}}%
}
\NewDocumentCommand{\ConjugateExponentSymbol}{s}{%
    \providecommand*{\ConjugateExponentSymboltag}{}%
    \IfBooleanTF{#1}
    {\renewcommand*{\ConjugateExponentSymboltag}{\prime}}%
    {\renewcommand*{\ConjugateExponentSymboltag}{\newlink{ConjugateExponentSymbol_global}{\prime}}}%
    \ConjugateExponentSymboltag%
}
\NewDocumentCommand{\TARGETConjugateExponentSymbol}{o}%
{%
    \IfNoValueTF{#1}%
        {\newtarget{ConjugateExponentSymbol_global}{\ConjugateExponentSymbol}}%
        {\newtarget{ConjugateExponentSymbol_global}{#1}}%
}
\NewDocumentCommand{\crM}{s}{%
    \providecommand*{\crMtag}{}%
    \IfBooleanTF{#1}
    {\renewcommand*{\crMtag}{c_{\r,\M}}}%
    {\renewcommand*{\crMtag}{\newlink{OFF_global}{c_{\r,\M}}}}%
    \crMtag%
}
\NewDocumentCommand{\TARGETcrM}{o}%
{%
    \IfNoValueTF{#1}%
        {\newtarget{OFF_global}{\crM}}%
        {\newtarget{OFF_global}{#1}}%
}
\NewDocumentCommand{\czero}{s}{%
    \providecommand*{\czerotag}{}%
    \IfBooleanTF{#1}
    {\renewcommand*{\czerotag}{c_{0}}}%
    {\renewcommand*{\czerotag}{\newlink{OFF_global}{c_{0}}}}%
    \czerotag%
}
\NewDocumentCommand{\TARGETczero}{o}%
{%
    \IfNoValueTF{#1}%
        {\newtarget{OFF_global}{\czero}}%
        {\newtarget{OFF_global}{#1}}%
}
\NewDocumentCommand{\d}{s}{%
    \providecommand*{\dtag}{}%
    \IfBooleanTF{#1}
    {\renewcommand*{\dtag}{d}}%
    {\renewcommand*{\dtag}{\newlink{d_global}{d}}}%
    \dtag%
}
\NewDocumentCommand{\TARGETd}{o}%
{%
    \IfNoValueTF{#1}%
        {\newtarget{d_global}{\d}}%
        {\newtarget{d_global}{#1}}%
}
\NewDocumentCommand{\delta}{s}{%
    \providecommand*{\deltatag}{}%
    \IfBooleanTF{#1}
    {\renewcommand*{\deltatag}{\OLDdelta}}%
    {\renewcommand*{\deltatag}{\newlink{OFF_global}{\OLDdelta}}}%
    \deltatag%
}
\NewDocumentCommand{\TARGETdelta}{o}%
{%
    \IfNoValueTF{#1}%
        {\newtarget{OFF_global}{\delta}}%
        {\newtarget{OFF_global}{#1}}%
}
\NewDocumentCommand{\Deltah}{s}{%
    \providecommand*{\Deltahtag}{}%
    \IfBooleanTF{#1}
    {\renewcommand*{\Deltahtag}{\ForwardDifferenceOperatorSymbol_{\h}}}%
    {\renewcommand*{\Deltahtag}{\newlink{OFF_global}{\ForwardDifferenceOperatorSymbol_{\h}}}}%
    \Deltahtag%
}
\NewDocumentCommand{\TARGETDeltah}{o}%
{%
    \IfNoValueTF{#1}%
        {\newtarget{OFF_global}{\Deltah}}%
        {\newtarget{OFF_global}{#1}}%
}
\NewDocumentCommand{\Deltahu}{s o}{%
    \providecommand*{\Deltahutag}{}%
    \IfBooleanTF{#1}
    {\renewcommand*{\Deltahutag}{\Deltah\u}}%
    {\renewcommand*{\Deltahutag}{\newlink{OFF_global}{\Deltah\u}}}%
    \IfNoValueTF{#2}%
    {\Deltahutag}%
    {%
        {\Deltahutag}\tonde*{#2}%
    }%
}

\NewDocumentCommand{\TARGETDeltahu}{o}%
{%
    \IfNoValueTF{#1}%
        {\newtarget{OFF_global}{\Deltahu}}%
        {\newtarget{OFF_global}{#1}}%
}
\NewDocumentCommand{\Deltahueps}{s o}{%
    \providecommand*{\Deltahuepstag}{}%
    \IfBooleanTF{#1}
    {\renewcommand*{\Deltahuepstag}{\Deltah\ueps}}%
    {\renewcommand*{\Deltahuepstag}{\newlink{OFF_global}{\Deltah\ueps}}}%
    \IfNoValueTF{#2}%
    {\Deltahuepstag}%
    {%
        {\Deltahuepstag}\tonde*{#2}%
    }%
}

\NewDocumentCommand{\TARGETDeltahueps}{o}%
{%
    \IfNoValueTF{#1}%
        {\newtarget{OFF_global}{\Deltahueps}}%
        {\newtarget{OFF_global}{#1}}%
}
\NewDocumentCommand{\Deps}{s o}{%
    \providecommand*{\Depstag}{}%
    \IfBooleanTF{#1}
    {\renewcommand*{\Depstag}{D_{\eps}}}%
    {\renewcommand*{\Depstag}{\newlink{OFF_global}{D_{\eps}}}}%
    \IfNoValueTF{#2}%
    {\Depstag}%
    {%
        {\Depstag}\tonde*{#2}%
    }%
}

\NewDocumentCommand{\TARGETDeps}{o}%
{%
    \IfNoValueTF{#1}%
        {\newtarget{OFF_global}{\Deps}}%
        {\newtarget{OFF_global}{#1}}%
}

\NewDocumentCommand{\DFAMILY}{s}{%
    \providecommand*{\DFAMILYtag}{}%
    \IfBooleanTF{#1}
    {\renewcommand*{\DFAMILYtag}{\graffe*{{}D_{\eps}}_{\eps>0}}}%
    {\renewcommand*{\DFAMILYtag}{\newlink{OFF_global}{\graffe*{{}D_{\eps}}_{\eps>0}}}}%
    \DFAMILYtag%
}
\NewDocumentCommand{\TARGETDFAMILY}{o}%
{%
    \IfNoValueTF{#1}%
        {\newtarget{OFF_global}{\DFAMILY}}%
        {\newtarget{OFF_global}{#1}}%
}
\NewDocumentCommand{\DistributionalDerivativeSymbol}{s}{%
    \providecommand*{\DistributionalDerivativeSymboltag}{}%
    \IfBooleanTF{#1}
    {\renewcommand*{\DistributionalDerivativeSymboltag}{D}}%
    {\renewcommand*{\DistributionalDerivativeSymboltag}{\newlink{DistributionalDerivativeSymbol_global}{D}}}%
    \DistributionalDerivativeSymboltag%
}
\NewDocumentCommand{\TARGETDistributionalDerivativeSymbol}{o}%
{%
    \IfNoValueTF{#1}%
        {\newtarget{DistributionalDerivativeSymbol_global}{\DistributionalDerivativeSymbol}}%
        {\newtarget{DistributionalDerivativeSymbol_global}{#1}}%
}
\NewDocumentCommand{\E}{s}{%
    \providecommand*{\Etag}{}%
    \IfBooleanTF{#1}
    {\renewcommand*{\Etag}{E}}%
    {\renewcommand*{\Etag}{\newlink{OFF_global}{E}}}%
    \Etag%
}
\NewDocumentCommand{\TARGETE}{o}%
{%
    \IfNoValueTF{#1}%
        {\newtarget{OFF_global}{\E}}%
        {\newtarget{OFF_global}{#1}}%
}
\NewDocumentCommand{\epsr}{s}{%
    \providecommand*{\epsrtag}{}%
    \IfBooleanTF{#1}
    {\renewcommand*{\epsrtag}{\varepsilon_{\r}}}%
    {\renewcommand*{\epsrtag}{\newlink{OFF_global}{\varepsilon_{\r}}}}%
    \epsrtag%
}
\NewDocumentCommand{\TARGETepsr}{o}%
{%
    \IfNoValueTF{#1}%
        {\newtarget{OFF_global}{\epsr}}%
        {\newtarget{OFF_global}{#1}}%
}
\NewDocumentCommand{\epszero}{s}{%
    \providecommand*{\epszerotag}{}%
    \IfBooleanTF{#1}
    {\renewcommand*{\epszerotag}{\varepsilon_{0}}}%
    {\renewcommand*{\epszerotag}{\newlink{OFF_global}{\varepsilon_{0}}}}%
    \epszerotag%
}
\NewDocumentCommand{\TARGETepszero}{o}%
{%
    \IfNoValueTF{#1}%
        {\newtarget{OFF_global}{\epszero}}%
        {\newtarget{OFF_global}{#1}}%
}
\NewDocumentCommand{\f}{s o}{%
    \providecommand*{\ftag}{}%
    \IfBooleanTF{#1}
    {\renewcommand*{\ftag}{f}}%
    {\renewcommand*{\ftag}{\newlink{OFF_global}{f}}}%
    \IfNoValueTF{#2}%
    {\ftag}%
    {%
        {\ftag}\tonde*{#2}%
    }%
}

\NewDocumentCommand{\TARGETf}{o}%
{%
    \IfNoValueTF{#1}%
        {\newtarget{OFF_global}{\f}}%
        {\newtarget{OFF_global}{#1}}%
}
\NewDocumentCommand{\F}{s}{%
    \providecommand*{\Ftag}{}%
    \IfBooleanTF{#1}
    {\renewcommand*{\Ftag}{F}}%
    {\renewcommand*{\Ftag}{\newlink{OFF_global}{F}}}%
    \Ftag%
}
\NewDocumentCommand{\TARGETF}{o}%
{%
    \IfNoValueTF{#1}%
        {\newtarget{OFF_global}{\F}}%
        {\newtarget{OFF_global}{#1}}%
}
\NewDocumentCommand{\ForwardDifferenceOperatorSymbol}{s}{%
    \providecommand*{\ForwardDifferenceOperatorSymboltag}{}%
    \IfBooleanTF{#1}
    {\renewcommand*{\ForwardDifferenceOperatorSymboltag}{\OLDDelta}}%
    {\renewcommand*{\ForwardDifferenceOperatorSymboltag}{\newlink{ForwardDifferenceOperatorSymbol_global}{\OLDDelta}}}%
    \ForwardDifferenceOperatorSymboltag%
}
\NewDocumentCommand{\TARGETForwardDifferenceOperatorSymbol}{o}%
{%
    \IfNoValueTF{#1}%
        {\newtarget{ForwardDifferenceOperatorSymbol_global}{\ForwardDifferenceOperatorSymbol}}%
        {\newtarget{ForwardDifferenceOperatorSymbol_global}{#1}}%
}
\NewDocumentCommand{\fr}{s o}{%
    \providecommand*{\frtag}{}%
    \IfBooleanTF{#1}
    {\renewcommand*{\frtag}{f_{\r}}}%
    {\renewcommand*{\frtag}{\newlink{OFF_global}{f_{\r}}}}%
    \IfNoValueTF{#2}%
    {\frtag}%
    {%
        {\frtag}\tonde*{#2}%
    }%
}

\NewDocumentCommand{\TARGETfr}{o}%
{%
    \IfNoValueTF{#1}%
        {\newtarget{OFF_global}{\fr}}%
        {\newtarget{OFF_global}{#1}}%
}
\NewDocumentCommand{\G}{s o}{%
    \providecommand*{\Gtag}{}%
    \IfBooleanTF{#1}
    {\renewcommand*{\Gtag}{G}}%
    {\renewcommand*{\Gtag}{\newlink{G_global}{G}}}%
    \IfNoValueTF{#2}%
    {\Gtag}%
    {%
        {\Gtag}\tonde*{#2}%
    }%
}

\NewDocumentCommand{\TARGETG}{o}%
{%
    \IfNoValueTF{#1}%
        {\newtarget{G_global}{\G}}%
        {\newtarget{G_global}{#1}}%
}

\NewDocumentCommand{\gamma}{s}{%
    \providecommand*{\gammatag}{}%
    \IfBooleanTF{#1}
    {\renewcommand*{\gammatag}{\OLDgamma}}%
    {\renewcommand*{\gammatag}{\newlink{OFF_global}{\OLDgamma}}}%
    \gammatag%
}
\NewDocumentCommand{\TARGETgamma}{o}%
{%
    \IfNoValueTF{#1}%
        {\newtarget{OFF_global}{\gamma}}%
        {\newtarget{OFF_global}{#1}}%
}
\NewDocumentCommand{\GradientVectorSymbol}{s}{%
    \providecommand*{\GradientVectorSymboltag}{}%
    \IfBooleanTF{#1}
    {\renewcommand*{\GradientVectorSymboltag}{\nabla}}%
    {\renewcommand*{\GradientVectorSymboltag}{\newlink{GradientVectorSymbol_global}{\nabla}}}%
    \GradientVectorSymboltag%
}
\NewDocumentCommand{\TARGETGradientVectorSymbol}{o}%
{%
    \IfNoValueTF{#1}%
        {\newtarget{GradientVectorSymbol_global}{\GradientVectorSymbol}}%
        {\newtarget{GradientVectorSymbol_global}{#1}}%
}
\NewDocumentCommand{\h}{s o}{%
    \providecommand*{\htag}{}%
    \IfBooleanTF{#1}
    {\renewcommand*{\htag}{h}}%
    {\renewcommand*{\htag}{\newlink{OFF_global}{h}}}%
    \IfNoValueTF{#2}%
    {\htag}%
    {%
        {\htag}_{#2}%
    }%
}

\NewDocumentCommand{\TARGETh}{o}%
{%
    \IfNoValueTF{#1}%
        {\newtarget{OFF_global}{\h}}%
        {\newtarget{OFF_global}{#1}}%
}
\NewDocumentCommand{\H}{s o}{%
    \providecommand*{\Htag}{}%
    \IfBooleanTF{#1}
    {\renewcommand*{\Htag}{H}}%
    {\renewcommand*{\Htag}{\newlink{H_global}{H}}}%
    \IfNoValueTF{#2}%
    {\Htag}%
    {%
        {\Htag}\tonde*{#2}%
    }%
}

\NewDocumentCommand{\TARGETH}{o}%
{%
    \IfNoValueTF{#1}%
        {\newtarget{H_global}{\H}}%
        {\newtarget{H_global}{#1}}%
}
\NewDocumentCommand{\HausdorffMeasureSymbol}{s}{%
    \providecommand*{\HausdorffMeasureSymboltag}{}%
    \IfBooleanTF{#1}
    {\renewcommand*{\HausdorffMeasureSymboltag}{\mathcal{H}}}%
    {\renewcommand*{\HausdorffMeasureSymboltag}{\newlink{OFF_global}{\mathcal{H}}}}%
    \HausdorffMeasureSymboltag%
}
\NewDocumentCommand{\TARGETHausdorffMeasureSymbol}{o}%
{%
    \IfNoValueTF{#1}%
        {\newtarget{OFF_global}{\HausdorffMeasureSymbol}}%
        {\newtarget{OFF_global}{#1}}%
}
\NewDocumentCommand{\HKbdryBr}{s o}{%
    \providecommand*{\HKbdryBrtag}{}%
    \IfBooleanTF{#1}
    {\renewcommand*{\HKbdryBrtag}{\NonlocalCurvatureSymbol_{\K,\bdry\Br}}}%
    {\renewcommand*{\HKbdryBrtag}{\newlink{OFF_global}{\NonlocalCurvatureSymbol_{\K,\bdry\Br}}}}%
    \IfNoValueTF{#2}%
    {\HKbdryBrtag}%
    {%
        {\HKbdryBrtag}\tonde*{#2}%
    }%
}

\NewDocumentCommand{\TARGETHKbdryBr}{o}%
{%
    \IfNoValueTF{#1}%
        {\newtarget{OFF_global}{\HKbdryBr}}%
        {\newtarget{OFF_global}{#1}}%
}
\NewDocumentCommand{\HKbdryE}{s o}{%
    \providecommand*{\HKbdryEtag}{}%
    \IfBooleanTF{#1}
    {\renewcommand*{\HKbdryEtag}{\NonlocalCurvatureSymbol_{\K,\bdry\E}}}%
    {\renewcommand*{\HKbdryEtag}{\newlink{OFF_global}{\NonlocalCurvatureSymbol_{\K,\bdry\E}}}}%
    \IfNoValueTF{#2}%
    {\HKbdryEtag}%
    {%
        {\HKbdryEtag}\tonde*{#2}%
    }%
}

\NewDocumentCommand{\TARGETHKbdryE}{o}%
{%
    \IfNoValueTF{#1}%
        {\newtarget{OFF_global}{\HKbdryE}}%
        {\newtarget{OFF_global}{#1}}%
}
\NewDocumentCommand{\HKepsbdryBr}{s o}{%
    \providecommand*{\HKepsbdryBrtag}{}%
    \IfBooleanTF{#1}
    {\renewcommand*{\HKepsbdryBrtag}{\NonlocalCurvatureSymbol_{\Keps,\bdry\Br}}}%
    {\renewcommand*{\HKepsbdryBrtag}{\newlink{OFF_global}{\NonlocalCurvatureSymbol_{\Keps,\bdry\Br}}}}%
    \IfNoValueTF{#2}%
    {\HKepsbdryBrtag}%
    {%
        {\HKepsbdryBrtag}\tonde*{#2}%
    }%
}

\NewDocumentCommand{\TARGETHKepsbdryBr}{o}%
{%
    \IfNoValueTF{#1}%
        {\newtarget{OFF_global}{\HKepsbdryBr}}%
        {\newtarget{OFF_global}{#1}}%
}
\NewDocumentCommand{\HKepsbdryE}{s o}{%
    \providecommand*{\HKepsbdryEtag}{}%
    \IfBooleanTF{#1}
    {\renewcommand*{\HKepsbdryEtag}{\NonlocalCurvatureSymbol_{\Keps,\bdry\E}}}%
    {\renewcommand*{\HKepsbdryEtag}{\newlink{OFF_global}{\NonlocalCurvatureSymbol_{\Keps,\bdry\E}}}}%
    \IfNoValueTF{#2}%
    {\HKepsbdryEtag}%
    {%
        {\HKepsbdryEtag}\tonde*{#2}%
    }%
}

\NewDocumentCommand{\TARGETHKepsbdryE}{o}%
{%
    \IfNoValueTF{#1}%
        {\newtarget{OFF_global}{\HKepsbdryE}}%
        {\newtarget{OFF_global}{#1}}%
}

\NewDocumentCommand{\Hmeno}{s}{%
    \providecommand*{\Hmenotag}{}%
    \IfBooleanTF{#1}
    {\renewcommand*{\Hmenotag}{H^{-}}}%
    {\renewcommand*{\Hmenotag}{\newlink{OFF_global}{H^{-}}}}%
    \Hmenotag%
}
\NewDocumentCommand{\TARGETHmeno}{o}%
{%
    \IfNoValueTF{#1}%
        {\newtarget{OFF_global}{\Hmeno}}%
        {\newtarget{OFF_global}{#1}}%
}

\NewDocumentCommand{\Hpiu}{s}{%
    \providecommand*{\Hpiutag}{}%
    \IfBooleanTF{#1}
    {\renewcommand*{\Hpiutag}{H^{+}}}%
    {\renewcommand*{\Hpiutag}{\newlink{OFF_global}{H^{+}}}}%
    \Hpiutag%
}
\NewDocumentCommand{\TARGETHpiu}{o}%
{%
    \IfNoValueTF{#1}%
        {\newtarget{OFF_global}{\Hpiu}}%
        {\newtarget{OFF_global}{#1}}%
}
\NewDocumentCommand{\hzero}{s o}{%
    \providecommand*{\hzerotag}{}%
    \IfBooleanTF{#1}
    {\renewcommand*{\hzerotag}{h_{0}}}%
    {\renewcommand*{\hzerotag}{\newlink{OFF_global}{h_{0}}}}%
    \IfNoValueTF{#2}%
    {\hzerotag}%
    {%
        {\hzerotag}_{#2}%
    }%
}

\NewDocumentCommand{\TARGEThzero}{o}%
{%
    \IfNoValueTF{#1}%
        {\newtarget{OFF_global}{\hzero}}%
        {\newtarget{OFF_global}{#1}}%
}
\NewDocumentCommand{\IG}{s o}{%
    \providecommand*{\IGtag}{}%
    \IfBooleanTF{#1}
    {\renewcommand*{\IGtag}{\mathcal{I}_{\G}}}%
    {\renewcommand*{\IGtag}{\newlink{IG_global}{\mathcal{I}_{\G}}}}%
    \IfNoValueTF{#2}%
    {\IGtag}%
    {%
        {\IGtag}\tonde*{#2}%
    }%
}

\NewDocumentCommand{\TARGETIG}{o}%
{%
    \IfNoValueTF{#1}%
        {\newtarget{IG_global}{\IG}}%
        {\newtarget{IG_global}{#1}}%
}
\NewDocumentCommand{\IK}{s o}{%
    \providecommand*{\IKtag}{}%
    \IfBooleanTF{#1}
    {\renewcommand*{\IKtag}{\mathcal{I}_{\K}}}%
    {\renewcommand*{\IKtag}{\newlink{IK_global}{\mathcal{I}_{\K}}}}%
    \IfNoValueTF{#2}%
    {\IKtag}%
    {%
        {\IKtag}\tonde*{#2}%
    }%
}

\NewDocumentCommand{\TARGETIK}{o}%
{%
    \IfNoValueTF{#1}%
        {\newtarget{IK_global}{\IK}}%
        {\newtarget{IK_global}{#1}}%
}
\NewDocumentCommand{\IKsym}{s o}{%
    \providecommand*{\IKsymtag}{}%
    \IfBooleanTF{#1}
    {\renewcommand*{\IKsymtag}{\mathcal{I}_{\Ksym}}}%
    {\renewcommand*{\IKsymtag}{\newlink{IKsym_global}{\mathcal{I}_{\Ksym}}}}%
    \IfNoValueTF{#2}%
    {\IKsymtag}%
    {%
        {\IKsymtag}\tonde*{#2}%
    }%
}

\NewDocumentCommand{\TARGETIKsym}{o}%
{%
    \IfNoValueTF{#1}%
        {\newtarget{IKsym_global}{\IKsym}}%
        {\newtarget{IKsym_global}{#1}}%
}

\NewDocumentCommand{\IntegerNumbersSymbol}{s}{%
    \providecommand*{\IntegerNumbersSymboltag}{}%
    \IfBooleanTF{#1}
    {\renewcommand*{\IntegerNumbersSymboltag}{\mathbb{Z}}}%
    {\renewcommand*{\IntegerNumbersSymboltag}{\newlink{IntegerNumbersSymbol_global}{\mathbb{Z}}}}%
    \IntegerNumbersSymboltag%
}
\NewDocumentCommand{\TARGETIntegerNumbersSymbol}{o}%
{%
    \IfNoValueTF{#1}%
        {\newtarget{IntegerNumbersSymbol_global}{\IntegerNumbersSymbol}}%
        {\newtarget{IntegerNumbersSymbol_global}{#1}}%
}

\NewDocumentCommand{\JacobianMatrixSymbol}{s}{%
    \providecommand*{\JacobianMatrixSymboltag}{}%
    \IfBooleanTF{#1}
    {\renewcommand*{\JacobianMatrixSymboltag}{D}}%
    {\renewcommand*{\JacobianMatrixSymboltag}{\newlink{JacobianMatrixSymbol_global}{D}}}%
    \JacobianMatrixSymboltag%
}
\NewDocumentCommand{\TARGETJacobianMatrixSymbol}{o}%
{%
    \IfNoValueTF{#1}%
        {\newtarget{JacobianMatrixSymbol_global}{\JacobianMatrixSymbol}}%
        {\newtarget{JacobianMatrixSymbol_global}{#1}}%
}

\NewDocumentCommand{\k}{s}{%
    \providecommand*{\ktag}{}%
    \IfBooleanTF{#1}
    {\renewcommand*{\ktag}{k}}%
    {\renewcommand*{\ktag}{\newlink{OFF_global}{k}}}%
    \ktag%
}
\NewDocumentCommand{\TARGETk}{o}%
{%
    \IfNoValueTF{#1}%
        {\newtarget{OFF_global}{\k}}%
        {\newtarget{OFF_global}{#1}}%
}
\NewDocumentCommand{\K}{s o}{%
    \providecommand*{\Ktag}{}%
    \IfBooleanTF{#1}
    {\renewcommand*{\Ktag}{K}}%
    {\renewcommand*{\Ktag}{\newlink{K_global}{K}}}%
    \IfNoValueTF{#2}%
    {\Ktag}%
    {%
        {\Ktag}\tonde*{#2}%
    }%
}

\NewDocumentCommand{\TARGETK}{o}%
{%
    \IfNoValueTF{#1}%
        {\newtarget{K_global}{\K}}%
        {\newtarget{K_global}{#1}}%
}
\NewDocumentCommand{\Kdelta}{s o}{%
    \providecommand*{\Kdeltatag}{}%
    \IfBooleanTF{#1}
    {\renewcommand*{\Kdeltatag}{K_{\delta}}}%
    {\renewcommand*{\Kdeltatag}{\newlink{Kdelta_global}{K_{\delta}}}}%
    \IfNoValueTF{#2}%
    {\Kdeltatag}%
    {%
        {\Kdeltatag}\tonde*{#2}%
    }%
}

\NewDocumentCommand{\TARGETKdelta}{o}%
{%
    \IfNoValueTF{#1}%
        {\newtarget{Kdelta_global}{\Kdelta}}%
        {\newtarget{Kdelta_global}{#1}}%
}
\NewDocumentCommand{\Keps}{s o}{%
    \providecommand*{\Kepstag}{}%
    \IfBooleanTF{#1}
    {\renewcommand*{\Kepstag}{K_{\eps}}}%
    {\renewcommand*{\Kepstag}{\newlink{OFF_global}{K_{\eps}}}}%
    \IfNoValueTF{#2}%
    {\Kepstag}%
    {%
        {\Kepstag}\tonde*{#2}%
    }%
}

\NewDocumentCommand{\TARGETKeps}{o}%
{%
    \IfNoValueTF{#1}%
        {\newtarget{OFF_global}{\Keps}}%
        {\newtarget{OFF_global}{#1}}%
}

\NewDocumentCommand{\KFAMILY}{s}{%
    \providecommand*{\KFAMILYtag}{}%
    \IfBooleanTF{#1}
    {\renewcommand*{\KFAMILYtag}{\graffe*{{}K_{\eps}}_{\eps>0}}}%
    {\renewcommand*{\KFAMILYtag}{\newlink{OFF_global}{\graffe*{{}K_{\eps}}_{\eps>0}}}}%
    \KFAMILYtag%
}
\NewDocumentCommand{\TARGETKFAMILY}{o}%
{%
    \IfNoValueTF{#1}%
        {\newtarget{OFF_global}{\KFAMILY}}%
        {\newtarget{OFF_global}{#1}}%
}
\NewDocumentCommand{\Kk}{s o}{%
    \providecommand*{\Kktag}{}%
    \IfBooleanTF{#1}
    {\renewcommand*{\Kktag}{K_{\k}}}%
    {\renewcommand*{\Kktag}{\newlink{OFF_global}{K_{\k}}}}%
    \IfNoValueTF{#2}%
    {\Kktag}%
    {%
        {\Kktag}\tonde*{#2}%
    }%
}

\NewDocumentCommand{\TARGETKk}{o}%
{%
    \IfNoValueTF{#1}%
        {\newtarget{OFF_global}{\Kk}}%
        {\newtarget{OFF_global}{#1}}%
}
\NewDocumentCommand{\Kn}{s o}{%
    \providecommand*{\Kntag}{}%
    \IfBooleanTF{#1}
    {\renewcommand*{\Kntag}{K_{\n}}}%
    {\renewcommand*{\Kntag}{\newlink{OFF_global}{K_{\n}}}}%
    \IfNoValueTF{#2}%
    {\Kntag}%
    {%
        {\Kntag}\tonde*{#2}%
    }%
}

\NewDocumentCommand{\TARGETKn}{o}%
{%
    \IfNoValueTF{#1}%
        {\newtarget{OFF_global}{\Kn}}%
        {\newtarget{OFF_global}{#1}}%
}
\NewDocumentCommand{\KpiuKsymMezzi}{s o}{%
    \providecommand*{\KpiuKsymMezzitag}{}%
    \IfBooleanTF{#1}
    {\renewcommand*{\KpiuKsymMezzitag}{\tonde{\K+\Ksym}/2}}%
    {\renewcommand*{\KpiuKsymMezzitag}{\newlink{KpiuKsymMezzi_global}{\tonde{\K+\Ksym}/2}}}%
    \IfNoValueTF{#2}%
    {\KpiuKsymMezzitag}%
    {%
        {\KpiuKsymMezzitag}\tonde*{#2}%
    }%
}

\NewDocumentCommand{\TARGETKpiuKsymMezzi}{o}%
{%
    \IfNoValueTF{#1}%
        {\newtarget{KpiuKsymMezzi_global}{\KpiuKsymMezzi}}%
        {\newtarget{KpiuKsymMezzi_global}{#1}}%
}
\NewDocumentCommand{\Ks}{s o}{%
    \providecommand*{\Kstag}{}%
    \IfBooleanTF{#1}
    {\renewcommand*{\Kstag}{K_{\s}}}%
    {\renewcommand*{\Kstag}{\newlink{Ks_global}{K_{\s}}}}%
    \IfNoValueTF{#2}%
    {\Kstag}%
    {%
        {\Kstag}\tonde*{#2}%
    }%
}

\NewDocumentCommand{\TARGETKs}{o}%
{%
    \IfNoValueTF{#1}%
        {\newtarget{Ks_global}{\Ks}}%
        {\newtarget{Ks_global}{#1}}%
}

\NewDocumentCommand{\KSEQUENCE}{s}{%
    \providecommand*{\KSEQUENCEtag}{}%
    \IfBooleanTF{#1}
    {\renewcommand*{\KSEQUENCEtag}{\graffe*{{{}K}_{\k}}_{\k\in\N}}}%
    {\renewcommand*{\KSEQUENCEtag}{\newlink{OFF_global}{\graffe*{{{}K}_{\k}}_{\k\in\N}}}}%
    \KSEQUENCEtag%
}
\NewDocumentCommand{\TARGETKSEQUENCE}{o}%
{%
    \IfNoValueTF{#1}%
        {\newtarget{OFF_global}{\KSEQUENCE}}%
        {\newtarget{OFF_global}{#1}}%
}
\NewDocumentCommand{\Ksym}{s o}{%
    \providecommand*{\Ksymtag}{}%
    \IfBooleanTF{#1}
    {\renewcommand*{\Ksymtag}{K_{\text{sym}}}}%
    {\renewcommand*{\Ksymtag}{\newlink{Ksym_global}{K_{\text{sym}}}}}%
    \IfNoValueTF{#2}%
    {\Ksymtag}%
    {%
        {\Ksymtag}\tonde*{#2}%
    }%
}

\NewDocumentCommand{\TARGETKsym}{o}%
{%
    \IfNoValueTF{#1}%
        {\newtarget{Ksym_global}{\Ksym}}%
        {\newtarget{Ksym_global}{#1}}%
}

\NewDocumentCommand{\lambda}{s}{%
    \providecommand*{\lambdatag}{}%
    \IfBooleanTF{#1}
    {\renewcommand*{\lambdatag}{\OLDlambda}}%
    {\renewcommand*{\lambdatag}{\newlink{OFF_global}{\OLDlambda}}}%
    \lambdatag%
}
\NewDocumentCommand{\TARGETlambda}{o}%
{%
    \IfNoValueTF{#1}%
        {\newtarget{OFF_global}{\lambda}}%
        {\newtarget{OFF_global}{#1}}%
}

\NewDocumentCommand{\LebesgueSpacelocSymbol}{s}{%
    \providecommand*{\LebesgueSpacelocSymboltag}{}%
    \IfBooleanTF{#1}
    {\renewcommand*{\LebesgueSpacelocSymboltag}{{\LebesgueSpaceSymbol}_{\loc}}}%
    {\renewcommand*{\LebesgueSpacelocSymboltag}{\newlink{OFF_global}{{\LebesgueSpaceSymbol}_{\loc}}}}%
    \LebesgueSpacelocSymboltag%
}
\NewDocumentCommand{\TARGETLebesgueSpacelocSymbol}{o}%
{%
    \IfNoValueTF{#1}%
        {\newtarget{OFF_global}{\LebesgueSpacelocSymbol}}%
        {\newtarget{OFF_global}{#1}}%
}

\NewDocumentCommand{\LebesgueSpaceSymbol}{s}{%
    \providecommand*{\LebesgueSpaceSymboltag}{}%
    \IfBooleanTF{#1}
    {\renewcommand*{\LebesgueSpaceSymboltag}{L}}%
    {\renewcommand*{\LebesgueSpaceSymboltag}{\newlink{OFF_global}{L}}}%
    \LebesgueSpaceSymboltag%
}
\NewDocumentCommand{\TARGETLebesgueSpaceSymbol}{o}%
{%
    \IfNoValueTF{#1}%
        {\newtarget{OFF_global}{\LebesgueSpaceSymbol}}%
        {\newtarget{OFF_global}{#1}}%
}

\NewDocumentCommand{\LipschitzSymbol}{s}{%
    \providecommand*{\LipschitzSymboltag}{}%
    \IfBooleanTF{#1}
    {\renewcommand*{\LipschitzSymboltag}{\operatorname{Lip}}}%
    {\renewcommand*{\LipschitzSymboltag}{\newlink{LipschitzSymbol_global}{\operatorname{Lip}}}}%
    \LipschitzSymboltag%
}
\NewDocumentCommand{\TARGETLipschitzSymbol}{o}%
{%
    \IfNoValueTF{#1}%
        {\newtarget{LipschitzSymbol_global}{\LipschitzSymbol}}%
        {\newtarget{LipschitzSymbol_global}{#1}}%
}

\NewDocumentCommand{\m}{s}{%
    \providecommand*{\mtag}{}%
    \IfBooleanTF{#1}
    {\renewcommand*{\mtag}{m}}%
    {\renewcommand*{\mtag}{\newlink{OFF_global}{m}}}%
    \mtag%
}
\NewDocumentCommand{\TARGETm}{o}%
{%
    \IfNoValueTF{#1}%
        {\newtarget{OFF_global}{\m}}%
        {\newtarget{OFF_global}{#1}}%
}

\NewDocumentCommand{\M}{s}{%
    \providecommand*{\Mtag}{}%
    \IfBooleanTF{#1}
    {\renewcommand*{\Mtag}{M}}%
    {\renewcommand*{\Mtag}{\newlink{OFF_global}{M}}}%
    \Mtag%
}
\NewDocumentCommand{\TARGETM}{o}%
{%
    \IfNoValueTF{#1}%
        {\newtarget{OFF_global}{\M}}%
        {\newtarget{OFF_global}{#1}}%
}

\NewDocumentCommand{\ModulusOfDeterminantOFJacobianMatrixSymbol}{s}{%
    \providecommand*{\ModulusOfDeterminantOFJacobianMatrixSymboltag}{}%
    \IfBooleanTF{#1}
    {\renewcommand*{\ModulusOfDeterminantOFJacobianMatrixSymboltag}{J}}%
    {\renewcommand*{\ModulusOfDeterminantOFJacobianMatrixSymboltag}{\newlink{ModulusOfDeterminantOFJacobianMatrixSymbol_global}{J}}}%
    \ModulusOfDeterminantOFJacobianMatrixSymboltag%
}
\NewDocumentCommand{\TARGETModulusOfDeterminantOFJacobianMatrixSymbol}{o}%
{%
    \IfNoValueTF{#1}%
        {\newtarget{ModulusOfDeterminantOFJacobianMatrixSymbol_global}{\ModulusOfDeterminantOFJacobianMatrixSymbol}}%
        {\newtarget{ModulusOfDeterminantOFJacobianMatrixSymbol_global}{#1}}%
}

\NewDocumentCommand{\mr}{s}{%
    \providecommand*{\mrtag}{}%
    \IfBooleanTF{#1}
    {\renewcommand*{\mrtag}{\m_{\r}}}%
    {\renewcommand*{\mrtag}{\newlink{OFF_global}{\m_{\r}}}}%
    \mrtag%
}
\NewDocumentCommand{\TARGETmr}{o}%
{%
    \IfNoValueTF{#1}%
        {\newtarget{OFF_global}{\mr}}%
        {\newtarget{OFF_global}{#1}}%
}

\NewDocumentCommand{\n}{s}{%
    \providecommand*{\ntag}{}%
    \IfBooleanTF{#1}
    {\renewcommand*{\ntag}{n}}%
    {\renewcommand*{\ntag}{\newlink{OFF_global}{n}}}%
    \ntag%
}
\NewDocumentCommand{\TARGETn}{o}%
{%
    \IfNoValueTF{#1}%
        {\newtarget{OFF_global}{\n}}%
        {\newtarget{OFF_global}{#1}}%
}

\NewDocumentCommand{\NaturalNumbersSymbol}{s}{%
    \providecommand*{\NaturalNumbersSymboltag}{}%
    \IfBooleanTF{#1}
    {\renewcommand*{\NaturalNumbersSymboltag}{\mathbb{N}}}%
    {\renewcommand*{\NaturalNumbersSymboltag}{\newlink{NaturalNumbersSymbol_global}{\mathbb{N}}}}%
    \NaturalNumbersSymboltag%
}
\NewDocumentCommand{\TARGETNaturalNumbersSymbol}{o}%
{%
    \IfNoValueTF{#1}%
        {\newtarget{NaturalNumbersSymbol_global}{\NaturalNumbersSymbol}}%
        {\newtarget{NaturalNumbersSymbol_global}{#1}}%
}

\NewDocumentCommand{\NonlocalCurvatureSymbol}{s}{%
    \providecommand*{\NonlocalCurvatureSymboltag}{}%
    \IfBooleanTF{#1}
    {\renewcommand*{\NonlocalCurvatureSymboltag}{H}}%
    {\renewcommand*{\NonlocalCurvatureSymboltag}{\newlink{NonlocalCurvatureSymbol_global}{H}}}%
    \NonlocalCurvatureSymboltag%
}
\NewDocumentCommand{\TARGETNonlocalCurvatureSymbol}{o}%
{%
    \IfNoValueTF{#1}%
        {\newtarget{NonlocalCurvatureSymbol_global}{\NonlocalCurvatureSymbol}}%
        {\newtarget{NonlocalCurvatureSymbol_global}{#1}}%
}

\NewDocumentCommand{\NonlocalFunctionalSymbol}{s}{%
    \providecommand*{\NonlocalFunctionalSymboltag}{}%
    \IfBooleanTF{#1}
    {\renewcommand*{\NonlocalFunctionalSymboltag}{\mathcal{V}}}%
    {\renewcommand*{\NonlocalFunctionalSymboltag}{\newlink{NonlocalFunctionalSymbol_global}{\mathcal{V}}}}%
    \NonlocalFunctionalSymboltag%
}
\NewDocumentCommand{\TARGETNonlocalFunctionalSymbol}{o}%
{%
    \IfNoValueTF{#1}%
        {\newtarget{NonlocalFunctionalSymbol_global}{\NonlocalFunctionalSymbol}}%
        {\newtarget{NonlocalFunctionalSymbol_global}{#1}}%
}

\NewDocumentCommand{\NonlocalIsoPBSymbol}{s}{%
    \providecommand*{\NonlocalIsoPBSymboltag}{}%
    \IfBooleanTF{#1}
    {\renewcommand*{\NonlocalIsoPBSymboltag}{P}}%
    {\renewcommand*{\NonlocalIsoPBSymboltag}{\newlink{OFF_global}{P}}}%
    \NonlocalIsoPBSymboltag%
}
\NewDocumentCommand{\TARGETNonlocalIsoPBSymbol}{o}%
{%
    \IfNoValueTF{#1}%
        {\newtarget{OFF_global}{\NonlocalIsoPBSymbol}}%
        {\newtarget{OFF_global}{#1}}%
}

\NewDocumentCommand{\NonlocalMaxPBSymbol}{s}{%
    \providecommand*{\NonlocalMaxPBSymboltag}{}%
    \IfBooleanTF{#1}
    {\renewcommand*{\NonlocalMaxPBSymboltag}{P'}}%
    {\renewcommand*{\NonlocalMaxPBSymboltag}{\newlink{OFF_global}{P'}}}%
    \NonlocalMaxPBSymboltag%
}
\NewDocumentCommand{\TARGETNonlocalMaxPBSymbol}{o}%
{%
    \IfNoValueTF{#1}%
        {\newtarget{OFF_global}{\NonlocalMaxPBSymbol}}%
        {\newtarget{OFF_global}{#1}}%
}

\NewDocumentCommand{\NonlocalPerimeterSymbol}{s}{%
    \providecommand*{\NonlocalPerimeterSymboltag}{}%
    \IfBooleanTF{#1}
    {\renewcommand*{\NonlocalPerimeterSymboltag}{\mathcal{P}}}%
    {\renewcommand*{\NonlocalPerimeterSymboltag}{\newlink{NonlocalPerimeterSymbol_global}{\mathcal{P}}}}%
    \NonlocalPerimeterSymboltag%
}
\NewDocumentCommand{\TARGETNonlocalPerimeterSymbol}{o}%
{%
    \IfNoValueTF{#1}%
        {\newtarget{NonlocalPerimeterSymbol_global}{\NonlocalPerimeterSymbol}}%
        {\newtarget{NonlocalPerimeterSymbol_global}{#1}}%
}
\NewDocumentCommand{\nuBr}{s o}{%
    \providecommand*{\nuBrtag}{}%
    \IfBooleanTF{#1}
    {\renewcommand*{\nuBrtag}{\OLDnu_{\Br}}}%
    {\renewcommand*{\nuBrtag}{\newlink{OFF_global}{\OLDnu_{\Br}}}}%
    \IfNoValueTF{#2}%
    {\nuBrtag}%
    {%
        {\nuBrtag}\tonde*{#2}%
    }%
}

\NewDocumentCommand{\TARGETnuBr}{o}%
{%
    \IfNoValueTF{#1}%
        {\newtarget{OFF_global}{\nuBr}}%
        {\newtarget{OFF_global}{#1}}%
}

\NewDocumentCommand{\Om}{s}{%
    \providecommand*{\Omtag}{}%
    \IfBooleanTF{#1}
    {\renewcommand*{\Omtag}{\Omega}}%
    {\renewcommand*{\Omtag}{\newlink{OFF_global}{\Omega}}}%
    \Omtag%
}
\NewDocumentCommand{\TARGETOm}{o}%
{%
    \IfNoValueTF{#1}%
        {\newtarget{OFF_global}{\Om}}%
        {\newtarget{OFF_global}{#1}}%
}

\NewDocumentCommand{\OmDUE}{s}{%
    \providecommand*{\OmDUEtag}{}%
    \IfBooleanTF{#1}
    {\renewcommand*{\OmDUEtag}{\Om_{2}}}%
    {\renewcommand*{\OmDUEtag}{\newlink{OFF_global}{\Om_{2}}}}%
    \OmDUEtag%
}
\NewDocumentCommand{\TARGETOmDUE}{o}%
{%
    \IfNoValueTF{#1}%
        {\newtarget{OFF_global}{\OmDUE}}%
        {\newtarget{OFF_global}{#1}}%
}

\NewDocumentCommand{\OmUNO}{s}{%
    \providecommand*{\OmUNOtag}{}%
    \IfBooleanTF{#1}
    {\renewcommand*{\OmUNOtag}{\Om_{1}}}%
    {\renewcommand*{\OmUNOtag}{\newlink{OFF_global}{\Om_{1}}}}%
    \OmUNOtag%
}
\NewDocumentCommand{\TARGETOmUNO}{o}%
{%
    \IfNoValueTF{#1}%
        {\newtarget{OFF_global}{\OmUNO}}%
        {\newtarget{OFF_global}{#1}}%
}

\NewDocumentCommand{\p}{s}{%
    \providecommand*{\ptag}{}%
    \IfBooleanTF{#1}
    {\renewcommand*{\ptag}{p}}%
    {\renewcommand*{\ptag}{\newlink{OFF_global}{p}}}%
    \ptag%
}
\NewDocumentCommand{\TARGETp}{o}%
{%
    \IfNoValueTF{#1}%
        {\newtarget{OFF_global}{\p}}%
        {\newtarget{OFF_global}{#1}}%
}

\NewDocumentCommand{\PartialDerivativeSymbol}{s}{%
    \providecommand*{\PartialDerivativeSymboltag}{}%
    \IfBooleanTF{#1}
    {\renewcommand*{\PartialDerivativeSymboltag}{\partial}}%
    {\renewcommand*{\PartialDerivativeSymboltag}{\newlink{OFF_global}{\partial}}}%
    \PartialDerivativeSymboltag%
}
\NewDocumentCommand{\TARGETPartialDerivativeSymbol}{o}%
{%
    \IfNoValueTF{#1}%
        {\newtarget{OFF_global}{\PartialDerivativeSymbol}}%
        {\newtarget{OFF_global}{#1}}%
}
\NewDocumentCommand{\phi}{s o}{%
    \providecommand*{\phitag}{}%
    \IfBooleanTF{#1}
    {\renewcommand*{\phitag}{\OLDphi}}%
    {\renewcommand*{\phitag}{\newlink{OFF_global}{\OLDphi}}}%
    \IfNoValueTF{#2}%
    {\phitag}%
    {%
        {\phitag}\tonde*{#2}%
    }%
}

\NewDocumentCommand{\TARGETphi}{o}%
{%
    \IfNoValueTF{#1}%
        {\newtarget{OFF_global}{\phi}}%
        {\newtarget{OFF_global}{#1}}%
}

\NewDocumentCommand{\PhiFAMILY}{s}{%
    \providecommand*{\PhiFAMILYtag}{}%
    \IfBooleanTF{#1}
    {\renewcommand*{\PhiFAMILYtag}{\graffe*{{}\OLDPhi_{\t}}_{\abs{\t}<\epszero}}}%
    {\renewcommand*{\PhiFAMILYtag}{\newlink{PhiFAMILY_global}{\graffe*{{}\OLDPhi_{\t}}_{\abs{\t}<\epszero}}}}%
    \PhiFAMILYtag%
}
\NewDocumentCommand{\TARGETPhiFAMILY}{o}%
{%
    \IfNoValueTF{#1}%
        {\newtarget{PhiFAMILY_global}{\PhiFAMILY}}%
        {\newtarget{PhiFAMILY_global}{#1}}%
}

\NewDocumentCommand{\PhiinvFAMILY}{s}{%
    \providecommand*{\PhiinvFAMILYtag}{}%
    \IfBooleanTF{#1}
    {\renewcommand*{\PhiinvFAMILYtag}{\graffe*{{}{\OLDPhi}^{-1}_{\t}}_{\abs{\t}<\epszero}}}%
    {\renewcommand*{\PhiinvFAMILYtag}{\newlink{PhiFAMILY_global}{\graffe*{{}{\OLDPhi}^{-1}_{\t}}_{\abs{\t}<\epszero}}}}%
    \PhiinvFAMILYtag%
}
\NewDocumentCommand{\TARGETPhiinvFAMILY}{o}%
{%
    \IfNoValueTF{#1}%
        {\newtarget{PhiFAMILY_global}{\PhiinvFAMILY}}%
        {\newtarget{PhiFAMILY_global}{#1}}%
}
\NewDocumentCommand{\Phiinvt}{s o}{%
    \providecommand*{\Phiinvttag}{}%
    \IfBooleanTF{#1}
    {\renewcommand*{\Phiinvttag}{{\OLDPhi}^{-1}_{\t}}}%
    {\renewcommand*{\Phiinvttag}{\newlink{PhiFAMILY_global}{{\OLDPhi}^{-1}_{\t}}}}%
    \IfNoValueTF{#2}%
    {\Phiinvttag}%
    {%
        {\Phiinvttag}\tonde*{#2}%
    }%
}

\NewDocumentCommand{\TARGETPhiinvt}{o}%
{%
    \IfNoValueTF{#1}%
        {\newtarget{PhiFAMILY_global}{\Phiinvt}}%
        {\newtarget{PhiFAMILY_global}{#1}}%
}
\NewDocumentCommand{\Phit}{s o}{%
    \providecommand*{\Phittag}{}%
    \IfBooleanTF{#1}
    {\renewcommand*{\Phittag}{\OLDPhi_{\t}}}%
    {\renewcommand*{\Phittag}{\newlink{PhiFAMILY_global}{\OLDPhi_{\t}}}}%
    \IfNoValueTF{#2}%
    {\Phittag}%
    {%
        {\Phittag}\tonde*{#2}%
    }%
}

\NewDocumentCommand{\TARGETPhit}{o}%
{%
    \IfNoValueTF{#1}%
        {\newtarget{PhiFAMILY_global}{\Phit}}%
        {\newtarget{PhiFAMILY_global}{#1}}%
}
\NewDocumentCommand{\pK}{s o}{%
    \providecommand*{\pKtag}{}%
    \IfBooleanTF{#1}
    {\renewcommand*{\pKtag}{p_{\K}}}%
    {\renewcommand*{\pKtag}{\newlink{OFF_global}{p_{\K}}}}%
    \IfNoValueTF{#2}%
    {\pKtag}%
    {%
        {\pKtag}\tonde*{#2}%
    }%
}

\NewDocumentCommand{\TARGETpK}{o}%
{%
    \IfNoValueTF{#1}%
        {\newtarget{OFF_global}{\pK}}%
        {\newtarget{OFF_global}{#1}}%
}
\NewDocumentCommand{\PK}{s o}{%
    \providecommand*{\PKtag}{}%
    \IfBooleanTF{#1}
    {\renewcommand*{\PKtag}{\NonlocalPerimeterSymbol_{\K}}}%
    {\renewcommand*{\PKtag}{\newlink{OFF_global}{\NonlocalPerimeterSymbol_{\K}}}}%
    \IfNoValueTF{#2}%
    {\PKtag}%
    {%
        {\PKtag}\tonde*{#2}%
    }%
}

\NewDocumentCommand{\TARGETPK}{o}%
{%
    \IfNoValueTF{#1}%
        {\newtarget{OFF_global}{\PK}}%
        {\newtarget{OFF_global}{#1}}%
}
\NewDocumentCommand{\PKk}{s o}{%
    \providecommand*{\PKktag}{}%
    \IfBooleanTF{#1}
    {\renewcommand*{\PKktag}{\NonlocalPerimeterSymbol_{\Kk}}}%
    {\renewcommand*{\PKktag}{\newlink{OFF_global}{\NonlocalPerimeterSymbol_{\Kk}}}}%
    \IfNoValueTF{#2}%
    {\PKktag}%
    {%
        {\PKktag}\tonde*{#2}%
    }%
}

\NewDocumentCommand{\TARGETPKk}{o}%
{%
    \IfNoValueTF{#1}%
        {\newtarget{OFF_global}{\PKk}}%
        {\newtarget{OFF_global}{#1}}%
}
\NewDocumentCommand{\PKsym}{s o}{%
    \providecommand*{\PKsymtag}{}%
    \IfBooleanTF{#1}
    {\renewcommand*{\PKsymtag}{\NonlocalPerimeterSymbol_{\Ksym}}}%
    {\renewcommand*{\PKsymtag}{\newlink{OFF_global}{\NonlocalPerimeterSymbol_{\Ksym}}}}%
    \IfNoValueTF{#2}%
    {\PKsymtag}%
    {%
        {\PKsymtag}\tonde*{#2}%
    }%
}

\NewDocumentCommand{\TARGETPKsym}{o}%
{%
    \IfNoValueTF{#1}%
        {\newtarget{OFF_global}{\PKsym}}%
        {\newtarget{OFF_global}{#1}}%
}
\NewDocumentCommand{\PKs}{s o}{%
    \providecommand*{\PKstag}{}%
    \IfBooleanTF{#1}
    {\renewcommand*{\PKstag}{\NonlocalPerimeterSymbol_{\Ks}}}%
    {\renewcommand*{\PKstag}{\newlink{OFF_global}{\NonlocalPerimeterSymbol_{\Ks}}}}%
    \IfNoValueTF{#2}%
    {\PKstag}%
    {%
        {\PKstag}\tonde*{#2}%
    }%
}

\NewDocumentCommand{\TARGETPKs}{o}%
{%
    \IfNoValueTF{#1}%
        {\newtarget{OFF_global}{\PKs}}%
        {\newtarget{OFF_global}{#1}}%
}

\NewDocumentCommand{\r}{s}{%
    \providecommand*{\rtag}{}%
    \IfBooleanTF{#1}
    {\renewcommand*{\rtag}{r}}%
    {\renewcommand*{\rtag}{\newlink{OFF_global}{r}}}%
    \rtag%
}
\NewDocumentCommand{\TARGETr}{o}%
{%
    \IfNoValueTF{#1}%
        {\newtarget{OFF_global}{\r}}%
        {\newtarget{OFF_global}{#1}}%
}

\NewDocumentCommand{\RealNumbersSymbol}{s}{%
    \providecommand*{\RealNumbersSymboltag}{}%
    \IfBooleanTF{#1}
    {\renewcommand*{\RealNumbersSymboltag}{\mathbb{R}}}%
    {\renewcommand*{\RealNumbersSymboltag}{\newlink{OFF_global}{\mathbb{R}}}}%
    \RealNumbersSymboltag%
}
\NewDocumentCommand{\TARGETRealNumbersSymbol}{o}%
{%
    \IfNoValueTF{#1}%
        {\newtarget{OFF_global}{\RealNumbersSymbol}}%
        {\newtarget{OFF_global}{#1}}%
}
\NewDocumentCommand{\rhoeps}{s o}{%
    \providecommand*{\rhoepstag}{}%
    \IfBooleanTF{#1}
    {\renewcommand*{\rhoepstag}{\OLDrho_{\eps}}}%
    {\renewcommand*{\rhoepstag}{\newlink{OFF_global}{\OLDrho_{\eps}}}}%
    \IfNoValueTF{#2}%
    {\rhoepstag}%
    {%
        {\rhoepstag}\tonde*{#2}%
    }%
}

\NewDocumentCommand{\TARGETrhoeps}{o}%
{%
    \IfNoValueTF{#1}%
        {\newtarget{OFF_global}{\rhoeps}}%
        {\newtarget{OFF_global}{#1}}%
}

\NewDocumentCommand{\rhoFAMILY}{s}{%
    \providecommand*{\rhoFAMILYtag}{}%
    \IfBooleanTF{#1}
    {\renewcommand*{\rhoFAMILYtag}{\graffe*{{}\OLDrho_{\eps}}_{\eps>0}}}%
    {\renewcommand*{\rhoFAMILYtag}{\newlink{OFF_global}{\graffe*{{}\OLDrho_{\eps}}_{\eps>0}}}}%
    \rhoFAMILYtag%
}
\NewDocumentCommand{\TARGETrhoFAMILY}{o}%
{%
    \IfNoValueTF{#1}%
        {\newtarget{OFF_global}{\rhoFAMILY}}%
        {\newtarget{OFF_global}{#1}}%
}

\NewDocumentCommand{\rp}{s}{%
    \providecommand*{\rptag}{}%
    \IfBooleanTF{#1}
    {\renewcommand*{\rptag}{r'}}%
    {\renewcommand*{\rptag}{\newlink{OFF_global}{r'}}}%
    \rptag%
}
\NewDocumentCommand{\TARGETrp}{o}%
{%
    \IfNoValueTF{#1}%
        {\newtarget{OFF_global}{\rp}}%
        {\newtarget{OFF_global}{#1}}%
}

\NewDocumentCommand{\Rpos}{s}{%
    \providecommand*{\Rpostag}{}%
    \IfBooleanTF{#1}
    {\renewcommand*{\Rpostag}{R}}%
    {\renewcommand*{\Rpostag}{\newlink{OFF_global}{R}}}%
    \Rpostag%
}
\NewDocumentCommand{\TARGETRpos}{o}%
{%
    \IfNoValueTF{#1}%
        {\newtarget{OFF_global}{\Rpos}}%
        {\newtarget{OFF_global}{#1}}%
}

\NewDocumentCommand{\s}{s}{%
    \providecommand*{\stag}{}%
    \IfBooleanTF{#1}
    {\renewcommand*{\stag}{s}}%
    {\renewcommand*{\stag}{\newlink{s_global}{s}}}%
    \stag%
}
\NewDocumentCommand{\TARGETs}{o}%
{%
    \IfNoValueTF{#1}%
        {\newtarget{s_global}{\s}}%
        {\newtarget{s_global}{#1}}%
}

\NewDocumentCommand{\SobolevSpacelocSymbol}{s}{%
    \providecommand*{\SobolevSpacelocSymboltag}{}%
    \IfBooleanTF{#1}
    {\renewcommand*{\SobolevSpacelocSymboltag}{{\SobolevSpaceSymbol}_{\loc}}}%
    {\renewcommand*{\SobolevSpacelocSymboltag}{\newlink{OFF_global}{{\SobolevSpaceSymbol}_{\loc}}}}%
    \SobolevSpacelocSymboltag%
}
\NewDocumentCommand{\TARGETSobolevSpacelocSymbol}{o}%
{%
    \IfNoValueTF{#1}%
        {\newtarget{OFF_global}{\SobolevSpacelocSymbol}}%
        {\newtarget{OFF_global}{#1}}%
}

\NewDocumentCommand{\SobolevSpaceSymbol}{s}{%
    \providecommand*{\SobolevSpaceSymboltag}{}%
    \IfBooleanTF{#1}
    {\renewcommand*{\SobolevSpaceSymboltag}{W}}%
    {\renewcommand*{\SobolevSpaceSymboltag}{\newlink{OFF_global}{W}}}%
    \SobolevSpaceSymboltag%
}
\NewDocumentCommand{\TARGETSobolevSpaceSymbol}{o}%
{%
    \IfNoValueTF{#1}%
        {\newtarget{OFF_global}{\SobolevSpaceSymbol}}%
        {\newtarget{OFF_global}{#1}}%
}

\NewDocumentCommand{\sqcup}{s}{%
    \providecommand*{\sqcuptag}{}%
    \IfBooleanTF{#1}
    {\renewcommand*{\sqcuptag}{\OLDsqcup}}%
    {\renewcommand*{\sqcuptag}{\newlink{sqcup_global}{\OLDsqcup}}}%
    \sqcuptag%
}
\NewDocumentCommand{\TARGETsqcup}{o}%
{%
    \IfNoValueTF{#1}%
        {\newtarget{sqcup_global}{\sqcup}}%
        {\newtarget{sqcup_global}{#1}}%
}

\NewDocumentCommand{\SubdifferentialSymbol}{s}{%
    \providecommand*{\SubdifferentialSymboltag}{}%
    \IfBooleanTF{#1}
    {\renewcommand*{\SubdifferentialSymboltag}{\partial}}%
    {\renewcommand*{\SubdifferentialSymboltag}{\newlink{SubdifferentialSymbol_global}{\partial}}}%
    \SubdifferentialSymboltag%
}
\NewDocumentCommand{\TARGETSubdifferentialSymbol}{o}%
{%
    \IfNoValueTF{#1}%
        {\newtarget{SubdifferentialSymbol_global}{\SubdifferentialSymbol}}%
        {\newtarget{SubdifferentialSymbol_global}{#1}}%
}

\NewDocumentCommand{\SupportSymbol}{s}{%
    \providecommand*{\SupportSymboltag}{}%
    \IfBooleanTF{#1}
    {\renewcommand*{\SupportSymboltag}{\operatorname{supp}}}%
    {\renewcommand*{\SupportSymboltag}{\newlink{SupportSymbol_global}{\operatorname{supp}}}}%
    \SupportSymboltag%
}
\NewDocumentCommand{\TARGETSupportSymbol}{o}%
{%
    \IfNoValueTF{#1}%
        {\newtarget{SupportSymbol_global}{\SupportSymbol}}%
        {\newtarget{SupportSymbol_global}{#1}}%
}

\NewDocumentCommand{\t}{s}{%
    \providecommand*{\ttag}{}%
    \IfBooleanTF{#1}
    {\renewcommand*{\ttag}{t}}%
    {\renewcommand*{\ttag}{\newlink{OFF_global}{t}}}%
    \ttag%
}
\NewDocumentCommand{\TARGETt}{o}%
{%
    \IfNoValueTF{#1}%
        {\newtarget{OFF_global}{\t}}%
        {\newtarget{OFF_global}{#1}}%
}

\NewDocumentCommand{\TopologicalBoundarySymbol}{s}{%
    \providecommand*{\TopologicalBoundarySymboltag}{}%
    \IfBooleanTF{#1}
    {\renewcommand*{\TopologicalBoundarySymboltag}{\partial}}%
    {\renewcommand*{\TopologicalBoundarySymboltag}{\newlink{TopologicalBoundarySymbol_global}{\partial}}}%
    \TopologicalBoundarySymboltag%
}
\NewDocumentCommand{\TARGETTopologicalBoundarySymbol}{o}%
{%
    \IfNoValueTF{#1}%
        {\newtarget{TopologicalBoundarySymbol_global}{\TopologicalBoundarySymbol}}%
        {\newtarget{TopologicalBoundarySymbol_global}{#1}}%
}
\NewDocumentCommand{\u}{s o}{%
    \providecommand*{\utag}{}%
    \IfBooleanTF{#1}
    {\renewcommand*{\utag}{u}}%
    {\renewcommand*{\utag}{\newlink{OFF_global}{u}}}%
    \IfNoValueTF{#2}%
    {\utag}%
    {%
        {\utag}\tonde*{#2}%
    }%
}

\NewDocumentCommand{\TARGETu}{o}%
{%
    \IfNoValueTF{#1}%
        {\newtarget{OFF_global}{\u}}%
        {\newtarget{OFF_global}{#1}}%
}
\NewDocumentCommand{\U}{s o}{%
    \providecommand*{\Utag}{}%
    \IfBooleanTF{#1}
    {\renewcommand*{\Utag}{U}}%
    {\renewcommand*{\Utag}{\newlink{OFF_global}{U}}}%
    \IfNoValueTF{#2}%
    {\Utag}%
    {%
        {\Utag}\tonde*{#2}%
    }%
}

\NewDocumentCommand{\TARGETU}{o}%
{%
    \IfNoValueTF{#1}%
        {\newtarget{OFF_global}{\U}}%
        {\newtarget{OFF_global}{#1}}%
}
\NewDocumentCommand{\ueps}{s o}{%
    \providecommand*{\uepstag}{}%
    \IfBooleanTF{#1}
    {\renewcommand*{\uepstag}{u_{\eps}}}%
    {\renewcommand*{\uepstag}{\newlink{OFF_global}{u_{\eps}}}}%
    \IfNoValueTF{#2}%
    {\uepstag}%
    {%
        {\uepstag}\tonde*{#2}%
    }%
}

\NewDocumentCommand{\TARGETueps}{o}%
{%
    \IfNoValueTF{#1}%
        {\newtarget{OFF_global}{\ueps}}%
        {\newtarget{OFF_global}{#1}}%
}
\NewDocumentCommand{\Ueps}{s o}{%
    \providecommand*{\Uepstag}{}%
    \IfBooleanTF{#1}
    {\renewcommand*{\Uepstag}{U_{\eps}}}%
    {\renewcommand*{\Uepstag}{\newlink{OFF_global}{U_{\eps}}}}%
    \IfNoValueTF{#2}%
    {\Uepstag}%
    {%
        {\Uepstag}\tonde*{#2}%
    }%
}

\NewDocumentCommand{\TARGETUeps}{o}%
{%
    \IfNoValueTF{#1}%
        {\newtarget{OFF_global}{\Ueps}}%
        {\newtarget{OFF_global}{#1}}%
}

\NewDocumentCommand{\uFAMILY}{s}{%
    \providecommand*{\uFAMILYtag}{}%
    \IfBooleanTF{#1}
    {\renewcommand*{\uFAMILYtag}{\graffe*{{}u_{\eps}}_{\eps>0}}}%
    {\renewcommand*{\uFAMILYtag}{\newlink{OFF_global}{\graffe*{{}u_{\eps}}_{\eps>0}}}}%
    \uFAMILYtag%
}
\NewDocumentCommand{\TARGETuFAMILY}{o}%
{%
    \IfNoValueTF{#1}%
        {\newtarget{OFF_global}{\uFAMILY}}%
        {\newtarget{OFF_global}{#1}}%
}

\NewDocumentCommand{\UFAMILY}{s}{%
    \providecommand*{\UFAMILYtag}{}%
    \IfBooleanTF{#1}
    {\renewcommand*{\UFAMILYtag}{\graffe*{{}U_{\eps}}_{\eps>0}}}%
    {\renewcommand*{\UFAMILYtag}{\newlink{OFF_global}{\graffe*{{}U_{\eps}}_{\eps>0}}}}%
    \UFAMILYtag%
}
\NewDocumentCommand{\TARGETUFAMILY}{o}%
{%
    \IfNoValueTF{#1}%
        {\newtarget{OFF_global}{\UFAMILY}}%
        {\newtarget{OFF_global}{#1}}%
}
\NewDocumentCommand{\uk}{s o}{%
    \providecommand*{\uktag}{}%
    \IfBooleanTF{#1}
    {\renewcommand*{\uktag}{u_{\k}}}%
    {\renewcommand*{\uktag}{\newlink{OFF_global}{u_{\k}}}}%
    \IfNoValueTF{#2}%
    {\uktag}%
    {%
        {\uktag}\tonde*{#2}%
    }%
}

\NewDocumentCommand{\TARGETuk}{o}%
{%
    \IfNoValueTF{#1}%
        {\newtarget{OFF_global}{\uk}}%
        {\newtarget{OFF_global}{#1}}%
}
\NewDocumentCommand{\un}{s o}{%
    \providecommand*{\untag}{}%
    \IfBooleanTF{#1}
    {\renewcommand*{\untag}{u_{\n}}}%
    {\renewcommand*{\untag}{\newlink{OFF_global}{u_{\n}}}}%
    \IfNoValueTF{#2}%
    {\untag}%
    {%
        {\untag}\tonde*{#2}%
    }%
}

\NewDocumentCommand{\TARGETun}{o}%
{%
    \IfNoValueTF{#1}%
        {\newtarget{OFF_global}{\un}}%
        {\newtarget{OFF_global}{#1}}%
}

\NewDocumentCommand{\uOm}{s}{%
    \providecommand*{\uOmtag}{}%
    \IfBooleanTF{#1}
    {\renewcommand*{\uOmtag}{{\u}_{\Om}}}%
    {\renewcommand*{\uOmtag}{\newlink{OFF_global}{{\u}_{\Om}}}}%
    \uOmtag%
}
\NewDocumentCommand{\TARGETuOm}{o}%
{%
    \IfNoValueTF{#1}%
        {\newtarget{OFF_global}{\uOm}}%
        {\newtarget{OFF_global}{#1}}%
}

\NewDocumentCommand{\uSEQUENCE}{s}{%
    \providecommand*{\uSEQUENCEtag}{}%
    \IfBooleanTF{#1}
    {\renewcommand*{\uSEQUENCEtag}{\graffe*{{{}u}_{\k}}_{\k\in\N}}}%
    {\renewcommand*{\uSEQUENCEtag}{\newlink{OFF_global}{\graffe*{{{}u}_{\k}}_{\k\in\N}}}}%
    \uSEQUENCEtag%
}
\NewDocumentCommand{\TARGETuSEQUENCE}{o}%
{%
    \IfNoValueTF{#1}%
        {\newtarget{OFF_global}{\uSEQUENCE}}%
        {\newtarget{OFF_global}{#1}}%
}
\NewDocumentCommand{\varphi}{s o}{%
    \providecommand*{\varphitag}{}%
    \IfBooleanTF{#1}
    {\renewcommand*{\varphitag}{\OLDvarphi}}%
    {\renewcommand*{\varphitag}{\newlink{OFF_global}{\OLDvarphi}}}%
    \IfNoValueTF{#2}%
    {\varphitag}%
    {%
        {\varphitag}\tonde*{#2}%
    }%
}

\NewDocumentCommand{\TARGETvarphi}{o}%
{%
    \IfNoValueTF{#1}%
        {\newtarget{OFF_global}{\varphi}}%
        {\newtarget{OFF_global}{#1}}%
}
\NewDocumentCommand{\varphir}{s o}{%
    \providecommand*{\varphirtag}{}%
    \IfBooleanTF{#1}
    {\renewcommand*{\varphirtag}{\OLDvarphi_{\r}}}%
    {\renewcommand*{\varphirtag}{\newlink{OFF_global}{\OLDvarphi_{\r}}}}%
    \IfNoValueTF{#2}%
    {\varphirtag}%
    {%
        {\varphirtag}\tonde*{#2}%
    }%
}

\NewDocumentCommand{\TARGETvarphir}{o}%
{%
    \IfNoValueTF{#1}%
        {\newtarget{OFF_global}{\varphir}}%
        {\newtarget{OFF_global}{#1}}%
}
\NewDocumentCommand{\varphiR}{s o}{%
    \providecommand*{\varphiRtag}{}%
    \IfBooleanTF{#1}
    {\renewcommand*{\varphiRtag}{\OLDvarphi_{\Rpos}}}%
    {\renewcommand*{\varphiRtag}{\newlink{OFF_global}{\OLDvarphi_{\Rpos}}}}%
    \IfNoValueTF{#2}%
    {\varphiRtag}%
    {%
        {\varphiRtag}\tonde*{#2}%
    }%
}

\NewDocumentCommand{\TARGETvarphiR}{o}%
{%
    \IfNoValueTF{#1}%
        {\newtarget{OFF_global}{\varphiR}}%
        {\newtarget{OFF_global}{#1}}%
}
\NewDocumentCommand{\VK}{s o}{%
    \providecommand*{\VKtag}{}%
    \IfBooleanTF{#1}
    {\renewcommand*{\VKtag}{\NonlocalFunctionalSymbol_{\K}}}%
    {\renewcommand*{\VKtag}{\newlink{OFF_global}{\NonlocalFunctionalSymbol_{\K}}}}%
    \IfNoValueTF{#2}%
    {\VKtag}%
    {%
        {\VKtag}\tonde*{#2}%
    }%
}

\NewDocumentCommand{\TARGETVK}{o}%
{%
    \IfNoValueTF{#1}%
        {\newtarget{OFF_global}{\VK}}%
        {\newtarget{OFF_global}{#1}}%
}
\NewDocumentCommand{\VKdelta}{s o}{%
    \providecommand*{\VKdeltatag}{}%
    \IfBooleanTF{#1}
    {\renewcommand*{\VKdeltatag}{\NonlocalFunctionalSymbol_{\Kdelta}}}%
    {\renewcommand*{\VKdeltatag}{\newlink{OFF_global}{\NonlocalFunctionalSymbol_{\Kdelta}}}}%
    \IfNoValueTF{#2}%
    {\VKdeltatag}%
    {%
        {\VKdeltatag}\tonde*{#2}%
    }%
}

\NewDocumentCommand{\TARGETVKdelta}{o}%
{%
    \IfNoValueTF{#1}%
        {\newtarget{OFF_global}{\VKdelta}}%
        {\newtarget{OFF_global}{#1}}%
}

\NewDocumentCommand{\x}{s o}{%
    \providecommand*{\xtag}{}%
    \IfBooleanTF{#1}
    {\renewcommand*{\xtag}{x}}%
    {\renewcommand*{\xtag}{\newlink{OFF_global}{x}}}%
    \IfNoValueTF{#2}%
    {\xtag}%
    {%
        {\xtag}_{#2}%
    }%
}

\NewDocumentCommand{\TARGETx}{o}%
{%
    \IfNoValueTF{#1}%
        {\newtarget{OFF_global}{\x}}%
        {\newtarget{OFF_global}{#1}}%
}
\NewDocumentCommand{\X}{s o}{%
    \providecommand*{\Xtag}{}%
    \IfBooleanTF{#1}
    {\renewcommand*{\Xtag}{X}}%
    {\renewcommand*{\Xtag}{\newlink{X_global}{X}}}%
    \IfNoValueTF{#2}%
    {\Xtag}%
    {%
        {\Xtag}\tonde*{#2}%
    }%
}

\NewDocumentCommand{\TARGETX}{o}%
{%
    \IfNoValueTF{#1}%
        {\newtarget{X_global}{\X}}%
        {\newtarget{X_global}{#1}}%
}

\NewDocumentCommand{\xp}{s o}{%
    \providecommand*{\xptag}{}%
    \IfBooleanTF{#1}
    {\renewcommand*{\xptag}{x'}}%
    {\renewcommand*{\xptag}{\newlink{OFF_global}{x'}}}%
    \IfNoValueTF{#2}%
    {\xptag}%
    {%
        {\xptag}_{#2}%
    }%
}

\NewDocumentCommand{\TARGETxp}{o}%
{%
    \IfNoValueTF{#1}%
        {\newtarget{OFF_global}{\xp}}%
        {\newtarget{OFF_global}{#1}}%
}

\NewDocumentCommand{\xzero}{s o}{%
    \providecommand*{\xzerotag}{}%
    \IfBooleanTF{#1}
    {\renewcommand*{\xzerotag}{x_{0}}}%
    {\renewcommand*{\xzerotag}{\newlink{OFF_global}{x_{0}}}}%
    \IfNoValueTF{#2}%
    {\xzerotag}%
    {%
        {\xzerotag}_{#2}%
    }%
}

\NewDocumentCommand{\TARGETxzero}{o}%
{%
    \IfNoValueTF{#1}%
        {\newtarget{OFF_global}{\xzero}}%
        {\newtarget{OFF_global}{#1}}%
}

\NewDocumentCommand{\y}{s o}{%
    \providecommand*{\ytag}{}%
    \IfBooleanTF{#1}
    {\renewcommand*{\ytag}{y}}%
    {\renewcommand*{\ytag}{\newlink{OFF_global}{y}}}%
    \IfNoValueTF{#2}%
    {\ytag}%
    {%
        {\ytag}_{#2}%
    }%
}

\NewDocumentCommand{\TARGETy}{o}%
{%
    \IfNoValueTF{#1}%
        {\newtarget{OFF_global}{\y}}%
        {\newtarget{OFF_global}{#1}}%
}

\NewDocumentCommand{\z}{s o}{%
    \providecommand*{\ztag}{}%
    \IfBooleanTF{#1}
    {\renewcommand*{\ztag}{z}}%
    {\renewcommand*{\ztag}{\newlink{OFF_global}{z}}}%
    \IfNoValueTF{#2}%
    {\ztag}%
    {%
        {\ztag}_{#2}%
    }%
}

\NewDocumentCommand{\TARGETz}{o}%
{%
    \IfNoValueTF{#1}%
        {\newtarget{OFF_global}{\z}}%
        {\newtarget{OFF_global}{#1}}%
}

\begin{document}
\title{A note on non-local Sobolev spaces and non-local perimeters}
	
\author[K.~Bessas]{Konstantinos Bessas}
\address[K.~Bessas]{Dipartimento di Matematica,
Università di Pavia, Via Adolfo Ferrata 5, 27100 Pavia, Italy}
\email{konstantinos.bessas@unipv.it}

\author[G. C.~Brusca]{Giuseppe Cosma Brusca}
\address[G. C.~Brusca]{SISSA, Via Bonomea 265, 34136 Trieste, Italy}
\email{gbrusca@sissa.it}

\date{\today}



\thanks{\textit{Acknowledgements}. 
KB has been supported by Fondazione Cariplo, grant n° 2023-0873.\\
KB is member of the INdAM-GNAMPA and his work is partially supported
by the INdAM -GNAMPA 2024 Project "Proprietà geometriche e regolarità 
in problemi variazionali 
locali e non locali", codice CUP E53C23001670001. GCB is member of the INdAM-GNAMPA}

\begin{abstract}
    We investigate the space of non-local Sobolev functions
    associated with an integral kernel.
    We prove an extension result, Sobolev and Poincaré inequalities
    and an isoperimetric inequality for the non-local perimeter restricted to a set.
    Finally, we remark on non-local isoperimetric problems,
    even when the underlying kernel is not necessarily radially symmetric.
\end{abstract}

\maketitle

\tableofcontents    

\section{Introduction}\label{sec:intro}
The aim of this note is presenting a result of continuous extension for non-local Sobolev spaces in which interactions 
are weighted by a non-negative kernel $K$. 
Letting $\TARGETK$ be a kernel, i.e., a measurable function from $\Rd$ to $\intervallo{[]}{0}{+\infty}$
such that $K\not\equiv0$ up to negligible sets, and letting
$\Om$ denote an open subset of $\Rd$, we consider the seminorm
\begin{equation}\label{eq:seminormWKp}
    \seminormWKpOm[\u]
    \coloneqq
    \tonde*{
        \intOmOm
        \abs{\u[\x]-\u[\y]}^{\p}
        \K[\x-\y]
        \integralde\x\de\y
    }^{\frac{1}{\p}}\in\intervallo{[]}{0}{+\infty},
\end{equation}
for $\p\in\unoinfinitoescluso$ and $\u:\Rd\to\intervallo{[]}{-\infty}{+\infty}$ measurable and finite a.e.,
and introduce the Banach space
\begin{equation*}
    \WKpOm:=\graffe{\u\in\LpOm:\seminormWKpOm[\u]<+\infty}
\end{equation*}
equipped with the norm $\|u\|_{W^{K,p}(\Omega)}:=\|u\|_{L^p(\Omega)}+[u]_{W^{K,p}(\Omega)}$.

We prove that, under fairly mild assumptions on the kernel,
$W^{K,p}(\Omega)$ is continuously embedded in $W^{K,p}(\R^d)$ through an extension operator,
provided that $\Omega$ is a bounded, open set with Lipschitz boundary.
The proof of this result is obtained by adapting the arguments employed in \cite{DiNPalVal12-MR2944369} for fractional Sobolev spaces,
i.e., in the case $K(z)=|z|^{-d-sp}$ for $s\in(0,1)$.
In order to treat a rather general class of kernels,
we detect a set of assumptions that are needed to us, but that may possibly be weakened,
in order to carry out our analysis. Some of these hypotheses have been discussed by the first author and Stefani in \cite{BesSte25-MR4845992}
for studying the properties of the above function space when $\Omega=\R^d$. In that work, it is remarked that certain assumptions are necessary
in order to obtain a non-trivial function space: as an example, a simple computation shows that if $K\in L^1(\R^d)$, then $W^{K,p}(\R^d)=L^p(\R^d)$,
and, for this reason, we may always assume that the kernel is \emph{non-integrable},
\begin{equation}\label{H:Nint}
\tag{Nint}    K \notin L^1(\R^d).
\end{equation}
In this work, we prove that another natural assumption is the
integrability of the kernel \emph{far from the origin}, i.e.,
\begin{equation}
    \label{H:Far}
    \tag{Far}
    \text{
    $\K\in\Lspace[1][\Rd\setminus\Br]$ for all $r>0$,
    }
\end{equation}
where $B_r$ denotes the euclidean ball centred at the origin with radius $r$. Indeed, we show that, if the above is not in force, then $\WKp=\graffe{0}$.

The main difficulties in mimicking the arguments presented in \cite{DiNPalVal12-MR2944369} are due to the general structure of the kernel $K$, as radial symmetry and homogeneity are usually lacking. Here, we present the main assumptions that we need, additionally to the above mentioned $(\text{Nint})$ and $(\text{Far})$, in order to obtain our result. 

We let $|\cdot|_*$ denote a generic norm on $\R^d$, and, for the sake of clarity, we let $|\cdot|_2$ denote the standard euclidean norm on $\R^d$. Given a measurable function $K\colon \R^d\to [0,+\infty]$, we will consider the following conditions:
\begin{equation}\label{H:Dec}
    \tag{$\mathrm{Dec}(c_0,|\cdot|_*)$}
    \begin{split}
      &  \text{ there exists } c_0\in(0,1] \text{ such that } \\
      & \quad  |x|_*\leq |y|_* \implies K(x)\geq c_0 K(y);
    \end{split}  
\end{equation}
\begin{equation}
\tag{$\mathrm{Dou}(D,|\cdot|_*)$}
\begin{split}
    & \text{ there exist positive constants } D \text{ and } C_D \text{ such that } \\ \label{H:Dou}
& \quad |y|_*=2|x|_*, |x|_*\leq D \implies K(x)\leq C_DK(y);  
\end{split}
\end{equation}
\begin{equation}\label{H:Nts}
\tag{$\mathrm{Nts}_p$}    \int_{\R^d}K(z)\min\{1, \EucNorm{\z}^{\p}\}\,\de\z <+\infty.
\end{equation}
We comment on the above hypotheses.
As for the assumption \eqref{H:Dec}, we observe that if $c_0=1$, then the kernel is of the form $K(z)=\kappa(|z|_*)$ for some $\kappa\colon [0,+\infty) \to [0,+\infty]$ 
non-increasing. Our analysis takes into account also kernels which are not necessarily of this form, indeed if $c_0\in(0,1)$, such assumption allows to deal with kernels that are not described by a one-dimensional, non-increasing profile. For instance, we may consider 
\begin{equation*}
    \kappa \colon [0,+\infty)\to [0,+\infty]
\end{equation*}
a non-increasing function, and
\begin{equation*}
    \phi \colon \R^d\to[\alpha,\beta]
\end{equation*}
a measurable function with $0<\alpha<\beta$.
Then, the kernel 
\begin{equation*}
K(z)=\kappa(|z|_*)\phi(z),
\end{equation*}
satisfies \eqref{H:Dec}
 with $c_0=\alpha/\beta$.

The doubling assumption \eqref{H:Dou} gives quantitative information on the growth of the kernel close to the origin. We will take advantage of this assumption, combined with \eqref{H:Dec}, in order to overcome the lack of homogeneity of $K$.

Finally, the assumption \eqref{H:Nts} is stronger than \eqref{H:Far} and it implies that the usual local Sobolev space $W^{1,p}(\R^d)$ is contained in $W^{K,p}(\R^d)$, with continuous injection. We point out that, as a consequence of the equivalence of the norms on $\R^d$, \eqref{H:Nts} could be equivalently stated in terms of the generic norm $|\cdot|_*$.

\smallskip

A prototypical kernel that satisfies all the above assumptions is the one piecewise defined by fractional kernels: 
\begin{equation*}
K(z)=
\begin{cases}
\alpha_1|z|_*^{-d-s_1p} & \text{ if } |z|_*\leq R_1, \\
\alpha_k|z|_*^{-d-s_kp} & \text{ if } R_{k-1}<|z|_*\leq R_{k},\quad k\in\{2,...,M-1\},
\\
\alpha_M|z|^{-d-s_Mp}_* & \text{ if } |z|_* > R_{M-1},
\end{cases}
\end{equation*}
where $M\in \mathbb{N}^+$ is fixed and it holds that $\alpha_1\geq ...\geq \alpha_M\geq0$, $1\leq R_1<...<R_M$, and $0<s_1\leq...\leq s_M<1$.

 \noindent Also mixed fractional-logarithmic kernels can be taken into account in our analysis. For instance, the kernel
 \begin{equation*}
     K(z)= 
     \begin{cases}
         (1-\log(|z|_*))^{-\alpha}|z|_*^{-d-s} & \text{ if } |z|_*\leq 1, \\
         0 & \text{ if } |z|_*>1
     \end{cases}
 \end{equation*}
always satisfies \eqref{H:Far}, \eqref{H:Dou} and \NtsOne{}; while \eqref{H:Nint} holds for 
\begin{equation*}
    (s,\alpha)\in (0,1)\times \R \cup \{0\}\times(-\infty, 1],
\end{equation*}
and \eqref{H:Dec} holds for every $c_0\in\intervallo{(]}{0}{1}$ whenever $\alpha\in(-\infty, 0]$.

\noindent Finally, we observe that assumption \eqref{H:Dec} allows to consider oscillating kernels when $c_0\in(0,1)$. Indeed, an example of admissible kernel is obtained starting from the one-dimensional profile
\begin{equation*}
    \kappa(t)=
    \begin{cases}
\beta t^{-d-s} & \text{ if } t\in(0,1], \\
\alpha\sin(2\pi t)+\beta & \text{ if } t\in (1, M], \\
\beta (t-M+1)^{-d-s} & \text{ if } t\in(M,+\infty),
    \end{cases}
\end{equation*}
where $M\in \mathbb{N}^+, 0<\alpha<\beta$, and it is given by
\begin{equation*}
    K(z)=\kappa(|z|_*),
\end{equation*}
which satisfies \eqref{H:Dec} with $c_0=\frac{\beta-\alpha}{\alpha+\beta}$.

We point out that an extension result is also stated (but not proved) in \cite[Theorem 3.78]{Fog20-phd} under the further assumption that the kernel is radial, without requiring a doubling property. In our proof, radial symmetry of the kernel is not required, but, in order to employ the arguments of \cite{DiNPalVal12-MR2944369}, we also rely on \eqref{H:Dou}. 

\smallskip

As an application of the extension result, we prove continuous and compact embeddings of $W^{K,p}(\Omega)$ in $L^q(\Omega)$ 
for some $q$ possibly strictly larger than $p$. At that end, we resort to some results that are already present in the literature, 
see again \cite{BesSte25-MR4845992} or \cite{CesNov18-MR3732175} by Cesaroni and Novaga.
To some extent, the reason that leads us to study the extension property of non-local Sobolev spaces
is that it allows us to prove Poincaré-Wirtinger and Sobolev inequalities, which, in turn, can be used to obtain
a non-local isoperimetric inequality relative to Lipschitz domains that may be employed for the study of (non-local) isoperimetric problems on the whole $\R^d$.

\subsection{Non-local geometric variational problems}
Several non-local geometric variational problems can be
studied applying the theory of non-local Sobolev functions. 

A first example is contained in \cite{BesSte25-MR4845992},
where the space of functions of non-local bounded variation (which coincides with $\WKone$)
was investigated and its theory was applied to
generalize results for a denoising model with $\Lspace[1]$ fidelity
and a non-local total variation associated to a singular kernel.
These results are related to the fidelity
of the model and were first proved in \cite{Bes22-MR4419012} for
the fractional kernel. 
In recent years there has been interest for non-local image denoising problems,
see for instance 
\cite{AntDiaJinSch24-MR4712400,MazSolTol22-MR4451904,NovOno23-MR4645235}. 

Another class of geometric variational problems 
involving the fractional kernel can be found in
\cite{BesNovOno23-MR4674821,CesNov17-MR3640534}
where the main object is a model to describe the
shape of a liquid under the action of a bulk energy.
In these models the structure of the fractional kernel
seems to play a significant role, due to the strong
(although partial) regularity theory in this context, which is
not clear for more general kernels.\
We also recall recent non-local models dealing with the study of liquids 
subjected to non-local interactions
\cite{MagVal17-MR3717439,DipMagVal17-MR3707346,DipMagVal22-MR4404780,DeLDipVal24-MR4721780,indrei2024nonlocalalmgrenproblem}.

Finally, another application is related to the study of the non-local isoperimetric problem,
which will be discussed in more detail in \cref{sec:final}.
We anticipate that under strong symmetry and monotonicity assumptions on the kernel,
the theory is well-known. Nevertheless, as soon as one drops these hypotheses, the situation
becomes significantly more complex. Some partial results are contained in \cite{CesNov18-MR3732175}
for integrable kernels and in
\cite{Lud14-MR3161386,Kre21-MR4214043}
for the fractional anisotropic perimeters.
\smallskip

\subsection{Plan of the paper}

The plan of the paper is the following.
In \cref{sec:notation} we first fix some notations and recall some preliminary results.

In \cref{sec:first_properties} we prove some properties of the space of non-local Sobolev functions.
First, the density of smooth functions in non-local Sobolev spaces is proved in \cref{res:BVKDense},
then the triviality of such spaces if \eqref{H:Far}
is not in force is proved in \cref{res:NonFarBVKtrivial}.

\cref{sec:extension} is devoted to the proof of the extension result (\cref{thm:extension}) and also contains, as an immediate Corollary, a density result of functions that are smooth up to the boundary of the regular domain $\Omega$.

In \cref{sec:SobolevPoincare} we recall some known facts about continuous and compact embeddings for $W^{K,p}(\R^d)$ and apply our extension theorem to obtain analogous results for $W^{K,p}(\Omega)$, see \cref{cor:compactembedding} and
\cref{cor:continuousembedding}, together with a Poincaré-Wirtinger inequality (\cref{prop:poincaré}).

In \cref{sec:final}, the monotonicity with respect to the radius of the non-local perimeter of euclidean balls is addressed in \cref{res:PKBR}.
An explicit formula to compute, in dimension one, the non-local perimeter of balls (i.e., intervals) is the object of \cref{res:formulaPKinter}.
An explicit example showing that in general euclidean balls are not solutions of the non-local isoperimetric problem with volume constraint
(even if the kernel is radially symmetric) is contained in \cref{res:BallsNotMinimizers}.
As a consequence of the previously obtained Poincaré-Wirtinger inequality, in \cref{cor:localisoperimetric} we furnish a non-local isoperimetric inequality relative to a regular domain $\Omega$.

\section{Notation and preliminary results}\label{sec:notation}

We fix some basic notations and recall some well-known facts, whose proof can be found, for instance, in \cite{BesSte25-MR4845992}.

Let $\TARGETd\geq1$. 
We put $\newtarget{complement_global}{\comp{\E}\coloneqq\Rd\setminus\E}$ for every $\E\subset\Rd$. For $\E\subset\Rd$ we denote by $\TARGETCharacteristicFunctionSymbol[\CharFun{\E}]$ the characteristic
function of $\E$, which is defined by $\CharFun{\E}[\x]\coloneqq1$ if $\x\in\E$ and
$\CharFun{\E}[\x]\coloneqq0$ if $\x\in\comp{\E}$. For any $\E,\F\subset\Rd$, we put $\distance\tonde{\E,\F}\coloneqq\inf\graffe*{\EucNorm{\x-\y}:\x\in\E,\y\in\F}$.

We let the symbol $\newtarget{lebesgue_global}{\lebesgue{\d}}$ denote the Lebesgue measure of $\Rd$ and, unless otherwise specified,
the term \emph{measurable} is used as a shorthand of $\lebesgue{\d}$-measurable
and \emph{a.e.} as a shorthand of $\lebesgue{\d}$-almost everywhere. If $\E\subset\Rd$ is a measurable set
we denote $\lebd{\E}\coloneqq\lebesgue{\d}[\E]$. $\Hdmenouno$ denotes the Hausdorff measure of dimension $\d-1$.

We write $\TARGETbigsqcup[\E\TARGETsqcup\F]$ to denote $\E\cup\F$ whenever $\E,\F\subset\Rd$
are measurable sets such that $\lebd{\E\cap\F}=0$.

We denote by $\TARGETBallSymbol[\BallRadiusCenter{\r}{\x}]\coloneqq\graffe*{\y\in\Rd:\EucNorm{\y-\x}<\r}$ the open euclidean ball
of radius $\r>0$ centred at $\x\in\Rd$ and $\Br\coloneq\BallRadiusCenter{\r}{0}$.
Moreover, we let $\BallVolumeCenter{\m}{\x}\coloneqq\BallRadiusCenter{\r\tonde{\m}}{\x}$,
where $\r\tonde{\m}=\tonde*{{\m}/{\lebd{\BallRadius{1}}}}^{1/\d}$, be the open euclidean ball
of volume $\m>0$ centred at $\x\in\Rd$ and $\Bm\coloneq\BallVolumeCenter{\m}{0}$.

We let the symbol $\TARGETTopologicalBoundarySymbol[\bdry{\E}]$ denote the topological boundary of
a set $\E\subset\Rd$.

If $\Om\subset\Rd$ is a measurable set of finite measure
we put $\rearr{\Om}=\BallVolume{\lebd{\Om}}$.
If $\u:\Rd\to\intervallo{[]}{-\infty}{+\infty}$ is a measurable function such that
$\lebd{\graffe*{\abs{\u}>\t}}<+\infty$ for all $\t>0$, 
we define $\rearr{\u}$, the \emph{symmetric-decreasing rearrangement} of $\u$, 
by
\begin{equation*}
    \rearr{\u}\tonde{\x}
    \coloneqq
    \int_0^{+\infty}
        \CharFun{\rearr{\graffe*{\abs{\u}>\t}}}(x)
    \integralde\t,
\end{equation*}
for every $\x\in\Rd$.

We define the forward difference operator $\Deltah:\LonelocRd\to\LonelocRd$ as,
\begin{equation*}
    \TARGETForwardDifferenceOperatorSymbol[\Deltahu[\x]]\coloneqq\u[\x+\h]-\u[\x],
\end{equation*}
for every $\x,\h\in\Rd$.

We then observe that, if $\K$ is a kernel,
\begin{equation}\label{eq:seminormWKpalt}
    \seminormWKp[\u]
    =
    \tonde*{
    \intRd
    \normLpRd{\Deltahu}^{\p}
    \K[\h]
    \integralde\h
    }^{\frac{1}{\p}},
\end{equation}
for all $\p\in\intervallo{[)}{1}{+\infty}$. 

We say that $\K$ is \emph{symmetric}  if
\begin{equation}
    \label{H:Sym}
    \tag{Sym}
    \text{
    $\K[\x]=\K[-\x]$ for all $\x\in\Rd$,
    }
\end{equation}
$\K$ is \emph{strictly positive} if
\begin{equation}
    \label{H:Pos}
    \tag{Pos}
    \text{
    $\K[\x]>0$ for all $\x\in\Rd$,
    }
\end{equation}
$\K$ has \emph{positive infimum around the origin} if
\begin{equation}
    \label{H:Inf}
    \tag{Inf}
    \text{
    there exist $\r,\mu>0$ such that $\K[\x]\geq\mu$ for all $\x\in\Br$.
    }
\end{equation}

\begin{remark}[][res:KsymSubstK]
    Let $\K$ be a kernel, $\p\in\unoinfinitoescluso$ and $\Om\subset\Rd$ open.
    We define $\TARGETKsym[\Ksym[\x]]\coloneqq\tonde*{\K[\x]+\K[-\x]}/{2}$ for every $\x\in\Rd$
    and we observe that $\Ksym$ is a kernel satisfying \eqref{H:Sym}, and
    \begin{equation*}
    \seminormWKpOm[\u]
    =
    \seminorm{\u}_{\nonlocalSobolevspace[\Ksym,\p][\Om]},
    \end{equation*}
for every $\u:\Om\to\intervallo{[]}{-\infty}{+\infty}$ measurable and finite a.e.. 
Furthermore, passing from $\K$ to $\Ksym$, the assumptions \eqref{H:Far},
\eqref{H:Nint}, \eqref{H:Nts}, \eqref{H:Pos} and \eqref{H:Inf} are preserved.
\end{remark}
In view of \cref{res:KsymSubstK},
we observe that the hypothesis \eqref{H:Sym}
is not restrictive if one is interested in computing \eqref{eq:seminormWKp},
provided that $\Ksym$ preserves all the initial assumptions satisfied by $\K$,
up to replace $\K$ with the kernel $\Ksym$ defined in \cref{res:KsymSubstK}.

\section{Properties of non-local Sobolev functions}\label{sec:first_properties}

In this section we first prove a density result for non-local Sobolev functions
and then show that \eqref{H:Far} is essential to make the spaces $\WKp$ non-trivial.

\begin{proposition}[][res:BVKDense]
    $\CinftyRd\cap\WKp$ is dense in $\WKp$
    for every $\p\in\unoinfinitoescluso$.
\end{proposition}
\begin{proof}

    We follow the same strategy as in \cite[Theorem 6.62]{Leo23-MR4567945}.

    Let $\K$ be a kernel and
    let $\p\in\unoinfinitoescluso$.
    Let $\TARGETrhoeps[\rhoFAMILY]$ be a family of standard mollifiers and for every
    $\eps>0$ define $\ueps\coloneqq\rhoeps\convolution\u$.
    We claim that for every $\eps>0$,
    \begin{equation}\label{eq:claimUNO}
        \seminormWKp[\ueps]\leq\seminormWKp[\u],
    \end{equation}
    and that
    \begin{equation}\label{eq:claimDUE}
        \limepstozeropiu\seminormWKp[\ueps-\u]=0.
    \end{equation}
    We define $\Deps[\h]\coloneqq\normLpRd{\Deltahueps-\Deltahu}^{\p}$, $\U[\h]\coloneqq\normLpRd{\Deltahu}^{\p}$ and
    $\Ueps[\h]\coloneqq\normLpRd{\Deltahueps}^{\p}$
    for every $\h\in\Rd$ and $\eps>0$.
    Thanks to \eqref{eq:seminormWKpalt}, for every $\eps>0$ we get,
    \begin{equation}\label{eq:intermedio}
        \seminormWKp[\ueps-\u]^{\p}
        =
        \intRd
        \Deps[\h]
        \K[\h]
        \integralde\h.
    \end{equation}
    We note that $\Deltahueps=\rhoeps\convolution\Deltahu$ for every $\h\in\Rd$ and $\eps>0$,
    which implies that $\Ueps[\h]\leq\U[\h]$ for every $\eps>0$.
    Thus, invoking \eqref{eq:seminormWKpalt} again, \eqref{eq:claimUNO} immediately follows.
    Moreover, 
    $\limepstozeropiu\Deps[\h]=0$ 
    and $\Deps[\h]\leq2^{\p-1}\tonde*{\Ueps[\h]+\U[\h]}\leq2^{\p}\U[\h]$ for every $\h\in\Rd$.
    Since $2^{\p}\K\U\in L^1(\R^d)$, because of \eqref{eq:seminormWKpalt}, it is a domination
    for the family $\graffe*{\K\Deps}_{\eps>0}$ and so we can apply
    Lebesgue dominated convergence theorem to pass to the limit in \eqref{eq:intermedio}
    to get \eqref{eq:claimDUE}.
\end{proof}

If a kernel does not satisfy \eqref{H:Far}, it necessarily develops
singularities away from the origin
(at finite distances or at infinity), as explained in the following:

\begin{lemma}[][res:NonFar]
Let $\K$ be a kernel. 
$\K$ does not satisfy assumption \eqref{H:Far} if and only if
either 
\begin{equation}\label{eq:KintInfty2}
    \integral{\comp{\Br}}\K[\h]\integralde\h
    =
    +\infty,
\end{equation}
for every $\r>0$, or
there exists $\hzero\in\RdmenoOrigine$ such that
\begin{equation}\label{eq:KintInfty}
    \integral{\Om}\K[\h]\integralde\h
    =
    +\infty,
\end{equation}
for every $\Om\subset\Rd$ neighborhood of $\hzero$.
   
In particular, if $\K$ satisfies \eqref{H:Far}, then $\K\in\Llocspace[1][\RdmenoOrigine]$.  
\end{lemma}
\begin{proof}
    
Let $\K$ be a kernel which does not satisfy assumption \eqref{H:Far}.

If $\K\notin\Lspace[1][\Rd\setminus\Br]$ for every $\r>0$,
then \eqref{eq:KintInfty2} holds. 
Therefore, we can assume that there exist $0<\r<\rp$
s.t. $\K\notin\Lspace[1][\Brp\setminus\Br]$. 
By contradiction, assume that for every $\hzero\in\closure{\Brp\setminus\Br}$
there exists $\Om$ neighborhood of $\hzero$
such that \eqref{eq:KintInfty} does not hold. 
By compactness of $\closure{\Brp\setminus\Br}$ we can find a finite open
cover of $\Brp\setminus\Br$ such that the integral of $\K$ on each 
element of the cover is finite. This contradicts the fact that $\K$ is not
integrable on $\Brp\setminus\Br$. 
Therefore, the thesis holds for some
$\hzero\in\closure{\Brp\setminus\Br}\subset\RdmenoOrigine$.

Conversely, if there exists $\hzero\in\RdmenoOrigine$ such that
\eqref{eq:KintInfty} holds for every $\Om\subset\Rd$ neighborhood of $\hzero$,
we put $\r\coloneqq\abs{\hzero}>0$ and observe 
that $\K\not\in\Lspace[1][\Rd\setminus\Br]$.
To conclude we simply recall that if we assume \eqref{eq:KintInfty2},
then \eqref{H:Far} cannot hold by definition. 
\end{proof}

\begin{proposition}[][res:NonFarBVKtrivial]
Let $\K$ be a kernel which does not satisfy assumption \eqref{H:Far}.
Then, $\WKp=\graffe{0}$ for every $\p\in\unoinfinitoescluso$.     
\end{proposition}
\begin{proof}
    
Let $\K$ be a kernel which does not satisfy assumption \eqref{H:Far} and
let $\p\in\unoinfinitoescluso$.

\textit{Case 1.}
Let $\u\in\Cspace[0][\Rd]\cap\WKp$.
Having in mind 
\eqref{eq:seminormWKpalt},
we define $\U[\h]\coloneqq\normLpRd{\Deltahu}^{\p}$ for every $\h\in\Rd$
and 
$\U[\infty]\coloneqq\intRd\liminf_{\h\to\infty}\abs{\Deltahu[\x]}^{\p}\integralde\x$
, so that

\begin{equation*}
    \seminormWKp[\u]^{\p}
    =
    \intRd
    \U[\h]
    \K[\h]
    \integralde\h.
\end{equation*}

Observe that $0\leq U(h) \leq 2^p\|u\|^p_{L^p(\R^d)}$ for all $h\in\R^d\cup \{\infty\}$ and thanks to Fatou's Lemma it is lower semicontinuous.

By \cref{res:NonFar} either there exists $\hzero\in\RdmenoOrigine$ 
such that \eqref{eq:KintInfty} holds for every $\Om$ neighborhood of $\hzero$,
or \eqref{eq:KintInfty2} holds for every $\r>0$.
In the first case, we claim that $\U[\hzero]=0$.
Indeed, if $\U[\hzero]>0$, by lower semicontinuity, we can find a neighborhood $\Om$ of $\hzero$ 
such that $\U>\U[\hzero]/2$ on $\Om$ and consequently obtain
\begin{equation*}
    +\infty
    >
    2\seminormWKp[\u]^{\p}
    \geq
    \integral{\Om}\U[\h]\K[\h]\integralde\h
    \geq
    \frac{\U[\hzero]}{2}\integral{\Om}\K[\h]\integralde\h,
\end{equation*}
which contradicts \eqref{eq:KintInfty}. This implies that $\U[\hzero]=0$.
In the second case, we claim that $\U[\infty]=0$.
Indeed, if $\U[\infty]>0$, by lower semicontinuity, we can find $\r>0$
such that $\U>\U[\infty]/2$ on $\comp{\Br}$ and consequently obtain
\begin{equation*}
    +\infty
    >
    2\seminormWKp[\u]^{\p}
    \geq
    \integral{\comp{\Br}}\U[\h]\K[\h]\integralde\h
    \geq
    \frac{\U[\infty]}{2}\integral{\comp{\Br}}\K[\h]\integralde\h,
\end{equation*}
which contradicts \eqref{eq:KintInfty2}

In the first case, the condition
$\U[\hzero]=0$ implies that $\u[\x+\hzero]=\u[\x]$ for every $\x\in\Rd$,
in other words $\u$ is periodic with respect to translations of vectors in $\hzero\Z$.
Since $\u\in\LpRd\cap\Cspace[0][\Rd]$, $\u\equiv0$.

In the second case, the condition
$\U[\infty]=0$ implies that $\liminf_{\h\to\infty}\abs{\Deltahu[\x]}=0$
for $\almostev$ $\x\in\Rd$. If we assume that $\u\not\equiv0$ then, thanks to the
continuity of $\u$ we can find $\xzero\in\Rd$ such that
$\u[\xzero]\neq0$
and $\liminf_{\h\to\infty}\abs{\u[\xzero+\h]-\u[\xzero]}=0$.
In particular, $\liminf_{\h\to\infty}\abs{\u[\xzero+\h]}\geq\abs{\u[\xzero]}>0$, which contradicts 
the fact that $\u[.+\xzero]\in\LpRd$. 

\textit{Case 2.}
Let $\u\in\WKp$.
Thanks to \cref{res:BVKDense} we can find a sequence $\uSEQUENCE\subset\Cspace[\infty][\Rd]\cap\WKp$
such that $\uk\to\u$ in $\WKp$.
Since by the previous case we know that $\uk\equiv0$ for all $\k\in\N$, we deduce that $\u\equiv0$ $\almostev$ in $\Rd$.
\end{proof}

\section{An extension result}\label{sec:extension}

In this section we prove the extension result; that is, the continuous embedding of $W^{K,p}(\Omega)$ in $W^{K,p}(\R^d)$, provided that the boundary of $\Omega$ is regular enough and that the kernel $K$ satisfies suitable assumptions. Our proof is obtained by adapting the arguments contained in Section $5$ of \cite{DiNPalVal12-MR2944369} and follows by three preliminary lemmas. The first one furnishes the natural extension of a function in $W^{K,p}(\Omega)$ having compact support in $\Omega$. In the following, we will let $C$ denote a positive constant that may change from line to line.

\begin{lemma}\label{lemma:vanishing} Let $K$ be a kernel satisfying \eqref{H:Far}, $\Omega \subset \R^d$ an open set, and $u\in W^{K,p}(\Omega)$ with $p\in[1,+\infty)$. Assume that there exists a compact set $V\subset \Omega$ such that $u=0$ a.e. in $\Omega\setminus V$, then the function $\widetilde{u}$ defined as
\begin{equation}\label{extensionvanishing}
\widetilde{u}(x):=
\begin{cases}
u(x) & \text{ if } x\in\Omega, \\[3pt]
0 & \text{ if } x\in \Omega^c
\end{cases}    
\end{equation}
belongs to $W^{K,p}(\R^d)$ and there exists a positive constant $C=C(K, V, \Omega)$ such that 
\begin{equation*}
    \|\widetilde{u}\|_{W^{K,p}(\R^d)} \leq C\|u\|_{W^{K,p}(\Omega)}. 
\end{equation*}
\end{lemma}

\begin{proof}
It holds that $\|\widetilde{u}\|_{L^p(\R^d)}=\|u\|_{L^p(\Omega)}$. As for the seminorm, we have
\begin{align*}
    \int_{\R^d}\int_{\R^d} |\widetilde{u}(x)-\widetilde{u}(y)|^p K(x-y)\,\de\x\de\y & = \int_{\Omega}\int_{\Omega} |u(x)-u(y)|^p K(x-y)\,\de\x\de\y \\
    & \,\,\quad  + \int_{\Omega}|u(x)|^p\Bigl\{\int_{\Omega^c}  K(x-y)\,\de\y\Bigr\}\,\de\x \\
     & \,\,\quad  + \int_{\Omega}|u(y)|^p\Bigl\{\int_{\Omega^c}  K(x-y)\,\de\x\Bigr\}\,\de\y \\
     & = [u]^p_{W^{K,p}(\Omega)} + \int_V |u(x)|^p  \Bigl\{\int_{x-\Omega^c} K(z)\,\de\z\Bigr\}\,\de\x \\
     & \,\,\quad  + \int_V |u(y)|^p  \Bigl\{\int_{\Omega^c-y}  K(z)\,\de\z\Bigr\}\,\de\y \\
     & \leq [u]^p_{W^{K,p}(\Omega)} + 2 \|u\|^p_{L^p(\Omega)} \int_{B^c_r(0)} K(z)\,\de\z,
\end{align*}
where we performed two times the change of variables $z:=x-y$, and the last inequality follows by the fact that there exists $r>0$ such that the sets $x-\Omega^c$ and $\Omega^c-y$ are contained in $B^c_r(0)$ for all $x,y\in V$. Finally, the thesis follows by \eqref{H:Far}.
\end{proof}

The following lemma allows us to extend functions to a larger set $\Omega$ provided that this can be obtained by reflecting the initial domain with respect to a hyperplane.

\begin{lemma}\label{lemma:reflection} Let $\Omega \subset \R^d$ be an open set symmetric with respect to the coordinate $x_d$ and let $K$ be a kernel satisfying \eqref{H:Far}, \eqref{H:Dec}, and \eqref{H:Dou}. Consider the sets $\Omega_+:=\{x\in\Omega : x_d >0\}$ and $\Omega_-:=\Omega\setminus \Omega_+$, and let $u\in W^{K,p}(\Omega_+)$ with $p\in[1,+\infty)$. Then the function $\overline{u}$ defined as
\begin{equation}\label{extensionreflection}
\overline{u}(x):=
\begin{cases}
u(x', x_d) & \text{ if } x\in\Omega_+ \\[3pt]
u(x',-x_d) & \text{ if } x\in \Omega_-
\end{cases},    
\end{equation}
belongs to $W^{K,p}(\Omega)$ and there exists a positive constant $C=C(c_0, C_D,|\cdot|_*, K)$ such that 
\begin{equation*}
    \|\widetilde{u}\|_{W^{K,p}(\Omega)} \leq C\|u\|_{W^{K,p}(\Omega_+)}. 
\end{equation*}
\end{lemma}

\begin{proof}
    It is immediate to observe that $\|u\|^p_{L^p(\Omega)}= 2 \|u\|^p_{L^p(\Omega_+)}$. As for the seminorms, it holds
    \begin{align*}
    & \int_{\Omega}\int_{\Omega} |\overline{u}(x)-\overline{u}(y)|^p K(x-y)\,\de\x\de\y \\
    & = [u]^p_{W^{K,p}(\Omega_+)} \\
    & \,\, \quad + \int_{\Omega_+}\int_{\Omega_-} |u(x)-u(y',-y_d)|^pK(x-y)\,\de\x\de\y  \\
    & \,\, \quad + \int_{\Omega_-}\int_{\Omega_+} |u(x',-x_d)-u(y)|^pK(x-y)\,\de\x\de\y \\
    & \,\, \quad  + \int_{\Omega_-}\int_{\Omega_-} |u(x',-x_d)-u(y',-y_d)|^p K(x-y)\,\de\x\de\y.
    \end{align*}
We estimate separately the last three addends.

By the change of variables $z:=(y', -y_d)$, we have
\begin{align*}
    & \int_{\Omega_+}\int_{\Omega_-} |u(x)-u(y',-y_d)|^pK(x-y)\,\de\x\de\y \\
    & = \int_{\Omega_+}\int_{\Omega_+} |u(x)-u(z)|^pK((x_1-z_1,...,x_{d-1}-z_{d-1}, x_d+z_d))\,\de\x\de\z.
\end{align*}
Note that
\begin{equation*}
    \EucNorm{\x-\y}=\EucNorm{(x_1-z_1,...,x_{d-1}-z_{d-1}, x_d+z_d)}\geq \EucNorm{\x-\z},
\end{equation*}
hence, by the equivalence of the norms, there exists a positive constant $c_*$ such that
\begin{equation*}
    c_*|(x_1-z_1,...,x_{d-1}-z_{d-1}, x_d+z_d)|_*\geq |x-z|_*,
\end{equation*}
and then, by \eqref{H:Dec}, we obtain
\begin{align*}
    & \int_{\Omega_+}\int_{\Omega_+} |u(x)-u(z)|^pK((x_1-z_1,...,x_{d-1}-z_{d-1}, x_d+z_d))\,\de\x\de\z \\
    & \leq \frac{1}{c_0}\int_{\Omega_+}\int_{\Omega_+} |u(x)-u(z)|^pK\Bigl(\frac{x-z}{c_*}\Bigr)\,\de\x\de\z. 
\end{align*}
Now we argue differently for short-range and long-range interactions. Let $m\in \mathbb{N}^+$ be such that 
\begin{equation*}
\frac{2^m}{c_*}\geq1.
\end{equation*}
Applying $m$ times \eqref{H:Dou}, and then applying \eqref{H:Dec}, we obtain
\begin{align}
    & \notag \iint_{\Omega_+^2\cap \{|x-z|_*\leq \frac{c_*D}{2^{m-1}}\}} |u(x)-u(z)|^pK\Bigl(\frac{x-z}{c_*}\Bigr)\,\de\x\de\z \\ \notag
    & \leq C_D^m \int_{\Omega_+}\int_{\Omega_+} |u(x)-u(z)|^pK\Bigl(\frac{2^m(x-z)}{c_*}\Bigr)\,\de\x\de\z \\ \label{reflection0.1}
    & \leq \frac{C_D^m}{c_0} \int_{\Omega_+}\int_{\Omega_+} |u(x)-u(z)|^pK(x-z)\,\de\x\de\z, 
\end{align}
where the iterated application of \eqref{H:Dou} is made possible by the fact that
\begin{equation*}
    \frac{2^k}{c_*}|x-z|_*\leq D
\end{equation*}
for all $k\in\{0,...,m-1\}$. As for long-range interactions, by the change of variables $w:=\frac{x-z}{c_*}$, we have
\begin{align}
    & \notag \iint_{\Omega_+^2\cap \{|x-z|_*>\frac{c_*D}{2^{m-1}}\}} |u(x)-u(z)|^pK\Bigl(\frac{x-z}{c_*}\Bigr)\,\de\x\de\z \\ \label{reflection0.2}
    & \leq 2^{p-1}\|u\|^p_{L^p(\Omega_+)}c_*^d\int_{\{w\in\R^d : |w|_* > \frac{D}{2^{m-1}}\}} K(w)\,\de w.
\end{align}
Gathering \eqref{reflection0.1} and \eqref{reflection0.2}, and using \eqref{H:Far} together with the equivalence of the norms, we obtain
\begin{equation}\label{reflection1}
        \int_{\Omega_+}\int_{\Omega_-} |u(x)-u(y',-y_d)|^pK(x-y)\,\de\x\de\y \leq C \|u\|_{W^{K,p}(\Omega_+)},
    \end{equation}
$C$ being a positive constant depending on $K, c_*, C_D$, and $c_0$.

With the same argument, we get
    \begin{equation}\label{reflection2}
        \int_{\Omega_-}\int_{\Omega_+} |u(x',-x_d)-u(y)|^pK(x-y)\,\de\x\de\y \leq C\|u\|_{W^{K,p}(\Omega_+)}.
    \end{equation}

Finally, we treat the last summand. By the changes of variables $z:=(x',-x_d)$ and $w:=(y', -y_d)$, we have
\begin{align*}
    & \int_{\Omega_-}\int_{\Omega_-} |u(x',-x_d)-u(y',-y_d)|^pK(x-y)\,\de\x\de\y \\
    & = \int_{\Omega_+}\int_{\Omega_+} |u(z)-u(w)|^pK((z_1-w_1,...,z_{d-1}-w_{d-1}, -z_d+w_d))\,\de\z\de w.
\end{align*}
As before, we obtain that
\begin{equation*}
    c_*|(z_1-w_1,...,z_{d-1}-w_{d-1}, -z_d+w_d)|_*\geq |z-w|_*,
\end{equation*}
and then, by \eqref{H:Dec}, we get
\begin{align*}
    & \int_{\Omega_+}\int_{\Omega_+} |u(z)-u(w)|^pK((z_1-w_1,...,z_{d-1}-w_{d-1}, -z_d+w_d))\,\de\z\de w \\
    & \leq \frac{1}{c_0}\int_{\Omega_+}\int_{\Omega_+} |u(z)-u(w)|^pK\Bigl(\frac{z-w}{c_*}\Bigl)\,\de\z\de w. 
\end{align*}
Reasoning as for the first two addends, we obtain
\begin{equation}\label{reflection3}
        \int_{\Omega_-}\int_{\Omega_-} |u(x',-x_d)-u(y',-y_d)|^pK(x-y)\,\de\x\de\y \leq C \|u\|_{W^{K,p}(\Omega_+)},
    \end{equation}
$C$ being a positive constant depending on $K, c_*, C_D$, and $c_0$. Putting together \eqref{reflection1}, \eqref{reflection2}, and \eqref{reflection3} we obtain the thesis.
    \end{proof}

\begin{remark}
    If the norm $|\cdot|_*$ is symmetric with respect to the hyperplane $\{x_d=0\}$ in the sense that 
    \begin{equation}\label{symmetry}
        |(x_1,...,x_d)|_*=|(x_1,...,-x_d)|_*
    \end{equation}
    for all $x\in\R^d$, we do not need the doubling assumption in Lemma \ref{lemma:reflection}. Indeed, let $x\in \Omega_+, y\in \Omega_-$, and $z:=(y', -y_d)\in \Omega_+$. There exists $\lambda\in [0,1]$ such that 
    \begin{align*}
        z-x & =(y_1-x_1,...,y_{d-1}-x_{d-1},-y_d-x_d) \\
        & = (y_1-x_1,...,y_{d-1}-x_{d-1},\lambda(y_d-x_d)+(1-\lambda)(x_d-y_d));
    \end{align*}
    then, by the convexity of the norm, we obtain
    \begin{align*}
        |z-x|_* & \leq \lambda|(y_1-x_1,...,y_{d-1}-x_{d-1},y_d-x_d)|_*\\
        & \quad \, +(1-\lambda)|(y_1-x_1,...,y_{d-1}-x_{d-1},x_d-y_d)|_* \\
        & = |y-x|_*,
    \end{align*}
    where the last equality follows by \eqref{symmetry}. The above inequality allows to carry out the proof of the previous lemma without making use of \eqref{H:Dou}. Moreover, this observation simplifies the argument as it allows to not reasoning differently for short-range and long-range interactions. Note that, if the norm $|\cdot|_*$ is symmetric with respect to a different hyperplane $\Pi$, then the above lemma is still valid without requiring \eqref{H:Dou} provided that the set $\Omega$ can be obtained by reflecting the set $\Omega_+$ with respect to $\Pi$.
\end{remark}

The last lemma is needed to resort to an argument involving partition of unity that we will employ in the proof of the extension result.

\begin{lemma}\label{lemma:cutoff} Let $K$ be a kernel satisfying \eqref{H:Nts}, $\Omega \subset \R^d$ an open set, $u\in W^{K,p}(\Omega)$ with $p\in[1,+\infty)$, and $\psi$ a Lipschitz continuous function compactly supported in $\Omega$ such that $0\leq \psi\leq 1$. Then the function $\psi u$ belongs to $W^{K,p}(\Omega)$ and, letting $\text{Lip}\psi$ denote the Lipschitz constant of $\psi$, there exists a positive constant $C=C(\text{Lip}\psi,K)$ such that  
\begin{equation*}
    \|\psi u\|_{W^{K,p}(\Omega)} \leq  C\|u\|_{W^{K,p}(\Omega)}.
\end{equation*}
\end{lemma}

\begin{proof}
    Since $0\leq \psi \leq 1$, it holds that $\|\psi u\|_{L^p(\Omega)} \leq \| u\|_{L^p(\Omega)}$. By adding and subtracting $\psi(x)u(y)$, we evaluate the seminorm as follows,
    \begin{align*}
        & \int_{\Omega}\int_{\Omega} |\psi u(x)-\psi u(y)|^p K(x-y)\,\de\x\de\y \\
        & \leq 2^{p-1} \Bigl\{ \int_{\Omega}\int_{\Omega} |\psi(x)|^p |u(x)-u(y)|^p K(x-y)\,\de\x\de\y \\
        & \,\quad + \int_{\Omega}\int_{\Omega} |u(y)|^p |\psi(x)-\psi(y)|^p K(x-y)\,\de\x\de\y \Bigr\} \\
        & \leq 2^{p-1} [u]^p_{W^{K,p}(\Omega)} \\
        & \,\quad  + 2^{p-1} \int_{\Omega}\int_{\Omega} |u(y)|^p|\psi(x)-\psi(y)|^p K(x-y)\,\de\x\de\y. 
     \end{align*}
By treating differently short-range and long-range interactions, we obtain
\begin{align*}
    & \iint_{\Omega^2\cap\{\EucNorm{\x-\y}\leq 1\} } |u(y)|^p|\psi(x)-\psi(y)|^p K(x-y)\,\de\x\de\y \\
    & \leq (\text{Lip}\psi)^p\|u\|^p_{L^p(\Omega)} \int_{B_1(0)}K(z)
    \EucNorm{\z}^{\p}\,\de\z, 
\end{align*}
and
\begin{align*}
    \iint_{\Omega^2\cap\{\EucNorm{\x-\y}> 1\} } |u(y)|^p|\psi(x)-\psi(y)|^p K(x-y)\,\de\x\de\y  \leq 2^p\|u\|^p_{L^p(\Omega)}\int_{B^c_1(0)} K(z)\,\de\z.
\end{align*}
Putting together these estimates, and making use of  \eqref{H:Nts}, which implies \eqref{H:Far}, the proof is concluded.
\end{proof}

Finally, we prove our main result.

\begin{theorem}[][thm:extension] 
    Let $\K$ be a kernel satisfying \eqref{H:Dec}, \eqref{H:Dou}, and \eqref{H:Nts}, 
    $\Omega \subset \R^d$ a bounded, open set with Lipschitz boundary and $p\in[1,+\infty)$. 
    Then $W^{K,p}(\Omega)$ is continuously embedded in $W^{K,p}(\R^d)$; that is, for every $u\in W^{K,p}(\Omega)$ there exists $\widetilde{u}\in W^{K,p}(\R^d)$ such that the restriction of $\widetilde{u}$ in $\Omega$ equals $u$ and there exists a positive constant $C=C(c_0,C_D,K, \Omega, |\cdot|_*)$ such that  
\begin{equation*}
    \|\widetilde{u}\|_{W^{K,p}(\R^d)} \leq C  \|u\|_{W^{K,p}(\Omega)}. 
\end{equation*}
\end{theorem}    

\begin{proof}
  Consider $\{B_j\}_{j=1}^M$ a finite cover of $\partial \Omega$ made up of open euclidean balls, $\partial \Omega \subset \bigcup_{j=1}^M B_j$, and let $\{\psi_j\}_{j=0}^M$ be a smooth partition of unity such that $\text{supp }\psi_0\subset \R^d\setminus \partial \Omega$ and $\text{supp }\psi_j\subset B_j$ for every $j\in\{1,...,M\}$. 

  By Lemma \ref{lemma:cutoff}, $\psi_0 u\in W^{K,p}(\Omega)$ and, by Lemma \ref{lemma:vanishing}, the extension $\widetilde{\psi_0u}$ belongs to $W^{K,p}(\R^d)$ with 
  \begin{equation}\label{extension:0}
      \|\widetilde{\psi_0u}\|_{W^{K,p}(\R^d)} \leq C \|u\|_{W^{K,p}(\Omega)},
  \end{equation}
  where $C$ depends on $\psi_0, \Omega, K$.

Now, we consider the restriction of $u$ to $\Omega \cap B_j$. We consider a bi-Lipschitz change of variables $T_j : Q \to B_j$ mapping $Q_+$ into $\Omega\cap B_j$, where we set
  \begin{equation*}
      Q:=\Bigl(-\frac{1}{2}, \frac{1}{2}\Bigr)^d \quad \text{ and } \quad  Q_+:=Q \cap \{x_d>0\},
  \end{equation*}
  and we define $v_j(y):=u(T_j(y)), y\in Q_+$. We aim at proving that $v_j\in W^{K,p}(Q_+)$. By the change of variables $x:=T_j^{-1}(y)$, we have that $\|v_j\|_{L^p(Q_+)} \leq C \|u\|_{L^p(\Omega\cap B_j)}$, with $C$ depending on $T_j$. As for the seminorms, we observe that there exists $L_j>0$ such that
\begin{equation*}
    \frac{1}{L_j}\EucNorm{T_j^{-1}(x)-T_j^{-1}(y)}\geq \EucNorm{\x-\y},
\end{equation*}
that, by the equivalence of the norms, implies 
\begin{equation}\label{lipestimate}
  \frac{c_*}{L_j}|T_j^{-1}(x)-T_j^{-1}(y)|_*\geq |x-y|_*.
\end{equation}
The changes of variables $x:=T_j(\hat{x})$ and $y:=T_j(\hat{y})$, \eqref{lipestimate}, and \eqref{H:Dec}, lead to
  \begin{align*}
      & \int_{Q_+}\int_{Q_+} |v_j(\hat{x})-v_j(\hat{y})|^pK(\hat{x}-\hat{y})\,\de \hat{x}\de \hat{y} \\
      & = \int_{Q_+}\int_{Q_+} |u(T_j(\hat{x}))-u(T_j(\hat{y}))|^pK(\hat{x}-\hat{y})\,\de \hat{x}\de \hat{y} \\
      & = \int_{\Omega\cap B_j}\int_{\Omega\cap B_j} |u(x)-u(y)|^pK(T_j^{-1}(x)-T_j^{-1}(y))|\text{det} T_j^{-1}|^2\,\de\x\de\y \\
      & \leq   \frac{C}{c_0} \int_{\Omega \cap B_j}\int_{\Omega \cap B_j} |u(x)-u(y)|^pK\Bigl(\frac{L_j}{c_*}(x-y)\Bigr)\,\de\x\de\y,
\end{align*}
        for some positive constant $C$ depending on $T_j$. Now we argue similarly to the proof of Lemma \ref{lemma:reflection}. We fix $m_j\in\mathbb{N}$ such that 
\begin{equation*}
\frac{2^{m_j}L_j}{c_*}\geq1,
\end{equation*}
and we apply $m_j$ times assumption \eqref{H:Dou} to estimate the contributions of the short-range interactions by
\begin{align} \notag
    & \iint_{(\Omega \cap B_j)^2\cap \{|x-y|_*\leq \frac{c_*D}{L_j2^{m_j-1}}\}} |u(x)-u(y)|^pK\Bigl(\frac{L_j}{c_*}(x-y)\Bigr)\,\de\x\de\y \\ \notag
    & \leq C_D^{m_j} \iint_{(\Omega \cap B_j)^2\cap \{|x-y|_*\leq\frac{c_*D}{L_j2^{m_j-1}}\}} |u(x)-u(y)|^pK\Bigl(2^{m_j}\frac{L_j}{c_*}(x-y)\Bigr)\,\de\x\de\y \\ \notag
    & \leq \frac{C_D^{m_j}}{c_0} \iint_{(\Omega \cap B_j)^2\cap \{|x-y|_*\leq\frac{c_*D}{L_j2^{m_j-1}}\}} |u(x)-u(y)|^pK(x-y)\,\de\x\de\y \\ \label{extension:1}
    & \leq \frac{C_D^{m_j}}{c_0}[u]^p_{W^{K,p}(\Omega)}.
\end{align}
Analogously, long-range interactions can be estimated as follows:    
  \begin{align} \notag
      & \iint_{(\Omega \cap B_j)^2\cap \{|x-y|_*>\frac{c_*D}{L_j2^{m_j-1}}\}} |u(x)-u(y)|^pK\Bigl(\frac{L_j}{c_*}(x-y)\Bigr)\,\de\x\de\y \\ \notag
      & \leq 2^{p-1}\|u\|^p_{L^p(\Omega)}\int_{\{z\in \R^d : |z|_*> \frac{c_*D}{L_j2^{m_j-1}}\}} K\Bigl(\frac{L_j}{c_*}z\Bigr)\,\de\z \\ \label{extension:2}
      & =  2^{p-1}\Bigl(\frac{c_*}{L_j}\Bigr)^d\|u\|^p_{L^p(\Omega)}\int_{ \{z\in \R^d : |z|_*> \frac{D}{2^{m_j-1}}\} } K(z)\,\de\z.
  \end{align}
Gathering \eqref{extension:1} and \eqref{extension:2}, and using \eqref{H:Far}, we infer that $v_j\in W^{K,p}(Q_+)$ with 
\begin{equation}\label{extension:3}
   \|v_j\|_{W^{K,p}(Q_+)} \leq C \|u\|_{W^{K,p}(\Omega)},  
\end{equation}
where $C$ depends on $T_j, c_0, c_*$, $C_D$ and $K$. 

We are in position to apply Lemma \ref{lemma:reflection} in order to obtain a function $\overline{v}_j\in W^{K,p}(Q)$ with
\begin{equation}\label{extension:4}
\|\overline{v}_j\|_{W^{K,p}(Q)}\leq C\|v_j\|_{W^{K,p}(Q_+)}.    
\end{equation}
We define $w_j(x):=\overline{v}_j(T_j^{-1}(x))$ for $x\in B_j$ and, arguing as for the function $v_j$, we obtain that $w_j\in W^{K,p}(B_j)$, with
\begin{equation}\label{extension:5}
    \|w_j\|_{W^{K,p}(B_j)}\leq C\|\overline{v}_j\|_{W^{K,p}(Q)},    
\end{equation}
for some $C$ depending on $T_j, c_0, c_*$, $C_D$, and $K$.

\noindent Therefore, combining Lemma \ref{lemma:cutoff} and Lemma \ref{lemma:vanishing}, we have that $\widetilde{\psi_j w_j} \in W^{K,p}(\R^d)$ with
\begin{equation}\label{extension:6}
    \|\widetilde{\psi_j w_j}\|_{W^{K,p}(\R^d)} \leq C \|u\|_{W^{K,p}(\Omega)},
\end{equation}
with the constant $C$ that, also by \eqref{extension:3}, \eqref{extension:4} and \eqref{extension:5}, depends on $T_j, c_0, c_*, C_D$, $\text{Lip}\psi_j, K$, and $\Omega$. In conclusion, the function
\begin{equation*}
    \widetilde{u}:= \widetilde{\psi_0u}+ \sum_{j=1}^M \widetilde{\psi_jw_j}
\end{equation*}
  is the desired extension as it coincides with $u$ on $\Omega$ and, gathering \eqref{extension:0} and \eqref{extension:6}, it satisfies
  \begin{equation*}
       \|\widetilde{u}\|_{W^{K,p}(\R^d)} \leq C \|u\|_{W^{K,p}(\Omega)},
  \end{equation*}
  where the positive constant $C$ depends on $C_D, c_0, c_*$, the kernel $K$, and the set $\Omega$ through the Lipschitz continuous changes of variables $\{T_j\}_{j=1}^M$ and the smooth functions $\{\psi_j\}_{j=0}^M$.
  \end{proof}

As a first application of \cref{thm:extension}, we combine it with \cref{res:BVKDense} 
and immediately obtain the following result.

\begin{corollary}
    Let $\K$ be a kernel satisfying \eqref{H:Dec}, \eqref{H:Dou}, and \eqref{H:Nts}, 
    $\Om\subset \Rd$ a bounded, open set with Lipschitz boundary and $\p\in\unoinfinitoescluso$.
    Then,
    $\Cspace[\infty][\overline{\Om}]\cap\WKpOm$ is dense in $\WKpOm$
    for every $\p\in\unoinfinitoescluso$.
\end{corollary}

\section{Sobolev and Poincaré inequalities}\label{sec:SobolevPoincare}

In this section, we present a compactness result in $W^{K,p}(\Omega)$ together with some estimates in the spirit of Poincaré-Wirtinger and Sobolev inequalities on bounded, open, regular subsets of $\R^d$. At that end, we recall some results that are already present in the literature starting with a result of local compactness contained in \cite{BesSte25-MR4845992}. In the following, we will let $C$ denote a positive constant that may change from line to line.

\begin{theorem}\label{thm:compactness}  Let $K$ be a kernel satisfying \eqref{H:Nint} and \eqref{H:Far} and let $\{u_k\}_k\subset W^{K,p}(\R^d)$ with $p\in[1,+\infty)$ be a sequence such that
\begin{equation*}
    \sup_k\|u_k\|_{W^{k,p}(\R^d)}<+\infty.
\end{equation*}
Then there exist a subsequence $\{k_j\}_j$ and $u\in W^{K,p}(\R^d)$ such that 
\begin{equation*}
    u_{k_j}\to u \,\,\text{ in } L^p_{\text{loc}}(\R^d)\,\, \text{ as } j\to+\infty.
\end{equation*}
\end{theorem}
Having at disposal our extension theorem, we prove the following.

\begin{corollary}[][cor:compactembedding]
     Let $K$ be a kernel satisfying \eqref{H:Nint}, \eqref{H:Dec}, \eqref{H:Dou}, and \eqref{H:Nts}, $\Omega \subset \R^d$ a bounded, open set with Lipschitz boundary, and $p\in[1,+\infty)$. Then $W^{K,p}(\Omega)$ is compactly embedded in $L^p(\Omega)$; that is, if $\{u_k\}_k\subset W^{K,p}(\Omega)$ is such that
    \begin{equation*}
        \sup_k\|u_k\|_{W^{K,p}(\Omega)}<+\infty,
    \end{equation*}
    there exist a subsequence $\{k_j\}_j$ and $u\in W^{K,p}(\Omega)$ such that 
\begin{equation*}
    u_{k_j}\to u \,\,\text{ in } L^p(\Omega)\,\, \text{ as } j\to+\infty.
\end{equation*}
\end{corollary}
\begin{proof}
 Consider a sequence $\{u_k\}_k$ which is bounded in $W^{K,p}(\Omega)$. First, we apply Theorem \ref{thm:extension} to obtain a sequence $\{\widetilde{u}_k\}_k\subset W^{K,p}(\R^d)$ such that
    \begin{equation*}
    \sup_k\|\widetilde{u}_k\|_{W^{K,p}(\R^d)}<+\infty,
\end{equation*}
then, by Theorem \ref{thm:compactness} we find a subsequence $\{k_j\}_j$ and $u\in W^{K,p}(\R^d)$ such that 
\begin{equation*}
    \widetilde{u}_{k_j}\to u \,\,\text{ in } L^p_{\text{loc}}(\R^d)\,\, \text{ as } j\to+\infty.
\end{equation*} 
This, in turn, implies $u_{k_j}\to u$ in $L^p(\Omega)$, which concludes the proof.
\end{proof}

We illustrate a Sobolev-type inequality that is presented in \cite{CesNov18-MR3732175} and we immediately deduce an analogous result on bounded, regular domains through the extension result. 

\begin{proposition}[][prop:sobolev]
Assume that there exist $q\in[1,+\infty)$ and a positive constant $C$ such that (recalling the notation in \cref{sec:notation}),
\begin{equation}\label{isoperimetricassumption}
    {\NonlocalPerimeterSymbol}_{\rearr{\K}}(\Bm) \geq \frac{m^{\frac{1}{q}}}{C}    \qquad
\text{ for every } m>0.
\end{equation}
Then for every $u\in W^{K,1}(\R^d)$ it holds that
\begin{equation*}
    \|u\|_{L^q(\R^d)} \leq C [u]_{W^{K,1}(\R^d)}.
\end{equation*}
\end{proposition}

\begin{corollary}[][cor:continuousembedding]
Let $K$ be a kernel satisfying \eqref{H:Dec}, \eqref{H:Dou}, and \NtsOne{}, $\Omega \subset \R^d$ a bounded, open set with Lipschitz boundary, and assume that \eqref{isoperimetricassumption} holds for some $q\in[1,+\infty)$. Then $W^{K,1}(\Omega)$ is continuously embedded in $L^q(\Omega)$; that is, there exists a positive constant $C$ such that
\begin{equation*}
    \|u\|_{L^q(\Omega)}\leq C\|u\|_{W^{K,1}(\Omega)}
\end{equation*}
for every $u\in W^{K,1}(\Omega)$.
\end{corollary}
\begin{proof}
Given $u\in W^{K,1}(\Omega)$, by Theorem \ref{thm:extension} there exists an extension $\widetilde{u}\in W^{K,1}(\R^d)$ such that
\begin{equation*}
       \|\widetilde{u}\|_{W^{K,1}(\R^d)} \leq C  \|u\|_{W^{K,1}(\Omega)},
\end{equation*}
then, by Proposition \ref{prop:sobolev} we obtain
\begin{equation*}
    \|\widetilde{u}\|_{L^q(\R^d)} \leq C\|u\|_{W^{K,1}(\Omega)}.
\end{equation*}
Since $\widetilde{u}$ coincides with $u$ on $\Omega$, the thesis follows.
\end{proof}

In view of the study of non-local isoperimetric problems related to the kernel $K$, a useful tool is a relative isoperimetric inequality. Such estimate is a particular case of a Poincaré-Wirtinger inequality, that in turn, can be obtained through the embedding results that we already proved. We premise a useful remark on constant functions.

\begin{remark}\label{rmk:constantfunction}
    We observe that, in general, $[u]_{W^{K,p}(\Omega)}=0$ does not imply that $u$ is constant a.e. on $\Omega$. For instance, if $\Omega$ is made up of two disjoint balls whose distance is larger than the diameter of the support of $K$, then $[u]_{W^{K,p}(\Omega)}=0$ for every locally constant function $u$. However, if $\Omega$ is connected and $K$ satisfies \eqref{H:Nint}, \eqref{H:Far}, and \eqref{H:Dec}, such property holds true. To see this, it suffices to show that $[u]_{W^{K,p}(\Omega)}=0$ implies $u$ is locally constant. Let $x_0\in\Omega$ and let $r>0$ such that $B_{2r}(x_0)\subset\subset \Omega$, then
\begin{equation*}
    0 = [u]_{W^{K,p}(\Omega)} \geq \int_{B_r(0)} K(h)\Bigl\{\int_{B_r(x_0)}  |u(x+h)-u(x)|^p\,\de\x\Bigr\}\,\de h.
\end{equation*}
As a consequence of \eqref{H:Nint}, \eqref{H:Far}, and \eqref{H:Dec}, 
we deduce that \eqref{H:Inf} holds. This fact, combined with the above inequality, implies (if $\r>0$ is sufficiently small)
\begin{equation*}
    0 = \int_{B_r(0)}\Bigl\{\int_{B_r(x_0)} |u(x+h)-u(x)|^p\,\de\x\Bigr\}\,\de h,
\end{equation*}
which in turn yields $u(x+h)=u(x)$ for a.e. $h\in B_r(0)$ and $x\in B_r(x_0)$; i.e., $u(x)=u(y)$ for a.e. $x,y\in B_{2r}(x_0)$. Therefore, $u$ is a.e. equal to a constant in a neighborhood of $x_0$.
\end{remark} 
In order to simplify the notations in the following, we set 
\begin{equation*}
    u_{\Omega}:=\frac{1}{|\Omega|}\int_{\Omega}u\,\de\x,
\end{equation*}
for every $\u\in\LoneOm$.

\begin{proposition}[][prop:poincaré] 
Let $\Omega \subset \R^d$ a bounded, connected, open set with Lipschitz boundary, and $p\in[1,+\infty)$. Let $K$ be a kernel satisfying \eqref{H:Nint}, \eqref{H:Dec}, \eqref{H:Dou}, and \eqref{H:Nts}. Then there exists a positive constant $C$ such that 
\begin{equation}\label{PW}
    \Bigl(\int_{\Omega} |u-u_\Omega|^p\,\de\x\Bigr)^{\frac{1}{p}} \leq C [u]_{W^{K,p}(\Omega)}
\end{equation}
    for every $u\in W^{K,p}(\Omega)$. Moreover, if \eqref{isoperimetricassumption} holds for some $q\in[1,+\infty)$, then there exists a positive constant $C$ such that 
\begin{equation}\label{PWimproved}
    \Bigl(\int_{\Omega} |u-u_\Omega|^q\,\de\x\Bigr)^{\frac{1}{q}} \leq C [u]_{W^{K,1}(\Omega)}
\end{equation}
    for every $u\in W^{K,1}(\Omega)$.
\end{proposition}
\begin{proof}
The proof of \eqref{PW} is obtained by the usual contradiction-compactness argument. If \eqref{PW} fails, there exists a sequence $\{u_k\}_k \subset W^{K,p}(\Omega)$ such that for all $k\in \mathbb{N}$ it holds $(u_k)_\Omega=0$, $\|u_k\|_{L^p(\Omega)}=1$, and moreover
    \begin{equation*}
        [u_k]_{W^{K,p}(\Omega)}\to 0 \quad  \text{ as } k\to+\infty.
    \end{equation*}
Applying Corollary \ref{cor:compactembedding} we deduce there exists $u\in W^{K,p}(\Omega)$ such that, up to a not relabelled subsequence, $u_k\to u$ in $L^p(\Omega)$, and $u_k\to u$ a.e. in $\Omega$. We claim that $u$ is constant; indeed, by Fatou's Lemma,
\begin{equation*}
0=\liminf_{k\to+\infty}\int_{\Omega}\int_{\Omega} |u_k(x)-u_k(y)|^p K(x-y)\,\de\x\de\y \geq \int_{\Omega}\int_{\Omega} |u(x)-u(y)|^p K(x-y)\,\de\x\de\y
\end{equation*}
so that Remark \ref{rmk:constantfunction} and the connectedness of $\Omega$ imply the claim. Since $u_\Omega=0$, we infer $u=0$, which contradicts $\|u\|_{L^p(\Omega)}=1.$    
    
The proof of \eqref{PWimproved} is similar. Arguing again by contradiction, there exists a sequence $\{u_k\}_k \subset W^{K,1}(\Omega)$ such that for all $k\in \mathbb{N}$ it holds $(u_k)_\Omega=0$ and $\|u_k\|_{L^q(\Omega)}=1$, and moreover
    \begin{equation*}
        [u_k]_{W^{K,1}(\Omega)}\to 0 \quad  \text{ as } k\to+\infty.
    \end{equation*}
    Note that, by \eqref{PW}, this implies $u_k\to 0$ in $L^1(\Omega)$, which in turn yields $u_k\to 0$ in $W^{K,1}(\Omega)$ as $k\to+\infty$. Applying Corollary \ref{cor:continuousembedding}, we infer that $u_k\to 0$ in $L^q(\Omega)$ as $k\to+\infty$, which contradicts the fact that $\|u_k\|_{L^q(\Omega)}=1$ for all $k\in\mathbb{N}$, concluding the proof.
\end{proof}
\begin{remark}
    If $\K$ satisfies \eqref{H:Pos} and \eqref{H:Dec}, the proof of \eqref{PW} follows by a straightforward computation without requiring any additional assumption either on $K$ or on $\partial \Omega$. Indeed, we put $d=\text{diam }\Omega$ and obtain
    \begin{align*}
        \int_{\Omega} |u-u_\Omega|^p\,\de\x & \leq \frac{1}{|\Omega|}\int_\Omega \int_\Omega |u(x)-u(y)|^p\,\de\x\de\y \\
        & \leq \frac{1}{|\Omega|\inf\{K(z): z\in B_{d}(0)\}}\int_\Omega \int_\Omega |u(x)-u(y)|^pK(x-y)\,\de\x\de\y,
    \end{align*}
    which is the claimed inequality since, by the assumptions, it holds $\inf\{K(z): z\in B_{d}(0)\}>0$.  
\end{remark}

\section{Remarks on the non-local isoperimetric problem}\label{sec:final}
Let $\K$ be a kernel and let $\Om,\E\subset\Rd$ be measurable sets.
We define the \emph{$\K$-perimeter} of $\E$ relative to $\Om$ as
\begin{equation*}
  \PKresOm[\E]\coloneqq
      \tonde*{
        \unmezzo\integral{
          \Om\times\Om
        }
        +
        \integral{\Om\times\comp{\Om}}
      }
      \abs*{\CharFun{\E}[\x]-\CharFun{\E}[\y]}\K[\x-\y]
      \integralde\x\de\y,   
\end{equation*} 
the \emph{$\K$-perimeter} of $\E$
by $\TARGETNonlocalPerimeterSymbol[\PK[\E]]\coloneqq\PKresRd[\E]=\unmezzo\seminormWKone[\CharFun{\E}]$
and finally we put
\begin{equation*}
    \TARGETNonlocalFunctionalSymbol[\VK[\E]]
    \coloneqq\integral{\E\times\E}\K[\x-\y]\integralde\x\de\y.
\end{equation*}
For every $\m>0$ we consider the following variational problems,
 \begin{equation}\tag{$\newtarget{NonlocalIsopPB}{{\NonlocalIsoPBSymbol}_{\K,\m}}$}
    \pK[\m]\coloneqq\inf_{\E\subset\Rd,\lebd{\E}=\m}\PK[\E],
 \end{equation}

 \begin{equation}\tag{$\newtarget{NonlocalMaxPB}{{\NonlocalMaxPBSymbol}_{\K,\m}}$}
    \sup_{\E\subset\Rd,\lebd{\E}=\m}\VK[\E],
 \end{equation}
 where the function
 $\pK$ is usually referred to as the \emph{isoperimetric profile} or the \emph{isovolumetric function}. We recall that, by \cref{res:KsymSubstK}, it would not be restrictive to replace the kernel $K$ by its symmetrized kernel $K_{\text{sym}}(x)=(K(x)+K(-x))/2$. Therefore, we will often assume that \eqref{H:Sym} is satisfied.

 \begin{remark}[][res:equiv_PB_PK_VK]
    We observe that if $\K\in\LoneRd$ and \eqref{H:Sym} holds, then
    \begin{equation*}\label{eq:linkPKandVK}
    \PK[\E]\coloneqq\normLoneRd{\K}\lebd{\E}-\VK[\E].
    \end{equation*} 
    Therefore, under this assumptions,
    $\NonlocalIsopPBKernelVolume{\K*}{\m}$
    is equivalent to
    $\NonlocalMaxPBKernelVolume{\K*}{\m}$
    for every $\m>0$.
 \end{remark}

 \subsection{Non-local perimeter of balls}

 In this section, given a kernel $\K$, we are interested in studying the function
 \begin{equation*}
    \intervallo{()}{0}{+\infty}\ni\r\mapsto\PK[\Br],
 \end{equation*}
 which in general is a bounded from above by the function $\intervallo{()}{0}{+\infty}\ni\r\mapsto\pK[\lebd{\Br}]$.
 We start introducing the non-local curvature.

 \begin{definition}[Non-local curvature][def:NonlocalCurvature]
     Let $\K$ be a kernel satisfying \eqref{H:Sym} and \eqref{H:Far} and let $E\subset\Rd$ be a measurable set.
     We define
     \begin{equation}\label{eq:NonlocalCurvature}
         \TARGETNonlocalCurvatureSymbol[\HKbdryE[\x]]
         \coloneqq
         \limsup_{\eps\to0^+}    
         \integral{\Rd\setminus\BallRadiusCenter{\eps}{\x}}
         \tonde*{\CharFun{\comp{\E}}[\y]-\CharFun{\E}[\y]}\K[\x-\y]\integralde\y\in\intervallo{[]}{-\infty}{+\infty}
     \end{equation}
     for every $\x\in\Rd$.
 \end{definition}
 \begin{remark}
     Let $E\subset\Rd$ be a measurable set and $\x\in\Rd$.
     We observe that \eqref{eq:NonlocalCurvature} is well-defined, being,
     \begin{equation*}
             \integral{\Rd\setminus\BallRadiusCenter{\eps}{\x}}
             \abs*{\CharFun{\comp{\E}}[\y]-\CharFun{\E}[\y]}\K[\x-\y]\integralde\y
         \leq
         2
         \integral{\Rd\setminus\BallRadius{\eps}}
             \K[\z]\integralde\z<+\infty,
     \end{equation*}
     for every $\eps>0$.
     We also observe that, \begin{equation*}
         \HKbdryE[\x]
         =
         \limsup_{\eps\to0^+}
         \HKepsbdryE[\x]\in\intervallo{[]}{-\infty}{+\infty},
     \end{equation*}
     where $\KFAMILY$ is a family of kernels defined by
     $\Keps\coloneqq\CharFun{\comp{\Beps}}\K$ for every $\eps>0$.
 \end{remark}
 
 \begin{lemma}[][res:HKBalls]
     Let $\K$ be a kernel satisfying \eqref{H:Sym} and \eqref{H:Far}.
     Then, $\HKbdryBr[\x]\in\intervallo{[]}{0}{+\infty}$ for every $\r>0$ and $\x\in\bdry{\Br}$.
 \end{lemma}
 \begin{proof}
     Let $\r>0$ and $\x\in\bdry{\Br}$. For simplicity, we assume that $x=(0,...,0,r)$, the other cases being analogous. We put $\Hpiu\coloneqq\graffe*{\x\in\Rd:\x[\d]>0}$ and $\Hmeno\coloneqq\graffe*{\x\in\Rd:\x[\d]<0}$ and for every $\eps>0$
     we define the kernel $\Keps\coloneqq\CharFun{\comp{\Beps}}\K$.
     Then, being $\comp{\Br}=(\Hpiu+\x)\sqcup((\Hmeno+\x)\setminus\Br)$
     and $(\Hmeno+\x)=\Br\sqcup((\Hmeno+\x)\setminus\Br)$, we can write
     \begin{align*}
         \HKepsbdryBr[\x]
         &=
         \integral{\comp{\Br}}\Keps[\x-\y]\integralde\y-\integral{\Br}\Keps[\x-\y]\integralde\y\\
         &=
         \integral{\Hpiu+\x}\Keps[\x-\y]\integralde\y
         + \integral{(\Hmeno+\x)\setminus\Br}\Keps[\x-\y]\integralde\y\\
         &\phantom{=}-
         \integral{\Hmeno+\x}\Keps[\x-\y]\integralde\y
         +
         \integral{(\Hmeno+\x)\setminus\Br}\Keps[\x-\y]\integralde\y.
     \end{align*}
     Since $\Keps$ satisfies \eqref{H:Sym}, by a change of variables we get,
     \begin{equation*}
         \integral{\Hpiu+\x}\Keps[\x-\y]\integralde\y
         =
         \integral{\Hmeno+\x}\Keps[\x-\y]\integralde\y.
     \end{equation*}
     This implies that, 
     \begin{equation*}
         \HKepsbdryBr[\x]
        =
         2\integral{(\Hmeno+\x)\setminus\Br}\Keps[\x-\y]\integralde\y\geq0,
     \end{equation*}
     which allows us to deduce that,
     \begin{equation*}
         \HKbdryE[\x]
         =
         \limsup_{\eps\to0^+}
         \HKepsbdryBr[\x]\geq0.
     \end{equation*}
     This concludes the proof.
 \end{proof}
 \begin{remark}
 With the same argument used in the proof of \cref{res:HKBalls}, we can show that the non-local curvature of a convex set $C$ is non-negative for any point on $\partial C$. Indeed, for any $x\in \partial C$, we can consider as $H^+$ the half-space whose boundary coincides with the tangent space of $\partial C$ at $x$ and which does not contain $C$, and as $H^-$ its complement. 
 \end{remark}
 
 \begin{lemma}[][res:PKBRifKsmooth]
     Let $\K$ be a kernel satisfying \eqref{H:Sym} and s.t. 
     $\K\in\Cspace[1][\RdmenoOrigine]$ and 
     $\sup_{\z\in\RdmenoOrigine}\K[\z]\EucNorm{\z}^{\d+\s}<+\infty$
     for some $\s\in\intervallo{()}{0}{1}$.
     Then,
     the function $\intervallo{()}{0}{+\infty}\ni\r\mapsto\PK[\Br]$
     is monotone non-decreasing.
 \end{lemma}
 
 \begin{proof} 
     We observe that the assumptions imply that $\K$ satisfies \eqref{H:Far}.
     Let $\f[\r]\coloneqq\PK[\Br]$ for every $\r>0$.
     Let $\nuBr[\x]\coloneqq\x/\r$ for every $\x\in\bdry{\Br}$ be the outer unit normal to $\bdry{\Br}$.
     Fix $\r>0$ and let $\X\in\CcptSpace[\infty][\Rd][\Rd]$ s.t. $\X[\x]=\nuBr[\x]$ for every $\x\in\bdry{\Br}$.
     Let $\Phit[\x]\coloneqq\x+\t\X[\x]$ for every $\t\in\R$ and $\x\in\Rd$.
   
     Thanks to \cite[Theorem 6.1]{FigFusMagMilMor15-MR3322379}, $\f$ is differentiable and
     \begin{align*}
         \odv{\f}{\r}\tonde{\r}
         =
         \odv{}{\t}_{\t=0}\PK[\Phit[\Br]]
         =
         \integral{\bdry{\Br}}\HKbdryBr\X\scalar\nuBr\Hdmenouno
         =
         \integral{\bdry{\Br}}\HKbdryBr\Hdmenouno.
     \end{align*}
     We conclude that $\odv{\f}{\r}\tonde{\r}\geq0$ by \cref{res:HKBalls}.
 \end{proof}
 
 \begin{lemma}[][res:approximK]
     Let $\KSEQUENCE\subset\LoneRd$ be a sequence of kernels and $\K\in\LoneRd$ be a kernel.
     If $\Kk\weaklyto\K$ in $\LoneRd$, then $\lim_{\k} [u]_{W^{K_k,1}(\R^d)} =[u]_{W^{K,1}(\R^d)}$ for every $\u\in\LoneRd$.
 \end{lemma}
 \begin{proof}
     We define $\U[\h]\coloneqq\normLoneRd{\Deltahu}$ for every $\h\in\Rd$,
     so that
     \begin{equation}\label{eq:seminormBVKkalt}
         [u]_{W^{K_k,1}(\R^d)}
         =
         \intRd
         \U[\h]
         \Kk[\h]
         \integralde\h,
     \end{equation}
     for every $\k\in\N$.
 
     Since $\U\in\LinftyRd$, with $\norm{\U}_{\LinftyRd}\leq2\norm{\u}_{\LoneRd}$,
     we can pass to the limit in \eqref{eq:seminormBVKkalt} and conclude the proof.
 \end{proof}

 \begin{proposition}[][res:PKBR]
     Let $\K$ be a kernel.
     Then,
     the function $\intervallo{()}{0}{+\infty}\ni\r\mapsto\PK[\Br]$
     is monotone non-decreasing.
 \end{proposition}
 \begin{proof}
     Without loss of generality, recalling  \cref{res:KsymSubstK}, we may assume that \eqref{H:Sym} holds.
     If $\K$ does not satisfy \eqref{H:Far}, then by \cref{res:NonFarBVKtrivial}
     $\PK[\Br]=+\infty$ for every $\r>0$, so that there is nothing to prove.
    Therefore, we assume that \eqref{H:Far} holds and observe that
    \cref{res:NonFar} guarantees then that $\K\in\Llocspace[1][\RdmenoOrigine]$.
    Let $0<\r<\Rpos$.
 
     \textit{Step 1.}
     Assume that $\K\in\LoneRd$.
     Then we can find a sequence $\KSEQUENCE\subset\CcinftyRd$
     s.t $\Kk\to\K$ in $\LoneRd$.
     For every $\k\in\N$ we can also assume that $\Kk$ satisfies \eqref{H:Sym} 
     up to replacing it with the kernel defined by $\tonde*{\Kk[\x]+\Kk[-\x]}/2$ 
     for every $\x\in\Rd$ (see \cref{res:KsymSubstK}).
     In particular, $\Kk$ satisfies the hypotheses of \cref{res:PKBRifKsmooth}
     for every $\k\in\N$, which implies that $\PKk[\Br]\leq\PKk[\BR]$ for every $\k\in\N$.
     Thanks to \cref{res:approximK} we conclude that $\PK[\Br]\leq\PK[\BR]$. 
 
     \textit{Step 2.}
     Let $\K\in\Llocspace[1][\RdmenoOrigine]$ 
     satisfying \eqref{H:Sym} 
     and define $\Kk\coloneqq\min\graffe{\k,\K}$ for every $\k\in\N$.
     Then, $\Kk\in\LoneRd$ and it satisfies \eqref{H:Sym} for every $\k\in\N$.
     By the previous step and since $\Kk\leq\K$ for every $\k\in\N$, we obtain
     \begin{equation}\label{eq:stima}
         \PKk[\Br]\leq\PKk[\BR]\leq\PK[\BR].
     \end{equation}   
     By the monotone convergence theorem (or by Fatou's lemma) from \eqref{eq:stima} we conclude that $\PK[\Br]\leq\PK[\BR]$.
 \end{proof}

 \begin{remark}
     If $\K$ is a kernel satisfying
     \NtsOne{},
     recalling \cite[eq. (2.8)]{BesSte25-MR4845992},
     we can obtain the following bound
     \begin{equation*}
         \PK[\Br]\leq\max\graffe*{
             \lebd{\BallRadius{1}}\r^{\d},\frac{1}{2}\Hdmenouno\tonde*{\bdry{\BallRadius{1}}}\r^{\d-1}
         }
         \integral{\Rd}\min\graffe*{1,\EucNorm{\x}}\K[\x]\integralde\x,
     \end{equation*}
     for every $\r>0$.
 \end{remark}

 We now provide an explicit formula
 to compute the non-local perimeter of one-dimensional balls.

 \begin{proposition}[][res:formulaPKinter]
     Let $\d=1$ and 
     let $\K$ be a kernel satisfying \eqref{H:Sym} and \eqref{H:Far}.
     Let $\xzero\in\intervallo{()}{0}{+\infty}$, $\czero\in\R$ and
      $\G\in\Wlocspace[1,1][\intervallo{()}{0}{+\infty}]$ defined by
     $\G[\x]\coloneqq\int_{\xzero}^{\x}\K[\y]\integralde\y+\czero$ for every $\x>0$.
     Let $\H\in\Cspace[1][\intervallo{()}{0}{+\infty}]$ s.t.
     $\odv{\H}{\x}=\G$.
     Then, for every $\r>0$,
     \begin{equation}\label{eq:formulaPKinter}
         \frac{1}{2}\PK[\Ir]=2\G[+\infty]\r+\H[0^+]-\H[2\r],
     \end{equation}
     where $\H[0^+]\coloneqq\lim_{\x\to0^+}\H[\x]\in\intervallo{(]}{-\infty}{+\infty}$
     and $\G[+\infty]\coloneqq\lim_{\x\to+\infty}\G[\x]\in\intervallo{()}{-\infty}{+\infty}$. 
 \end{proposition}
 \begin{proof}
     Recalling \cref{res:NonFar}, we observe that 
     $\K\in\Llocspace[1][\R\setminus\graffe{0}]$,
     so that $\G\in\Wlocspace[1,1][\intervallo{()}{0}{+\infty}]$ is well-defined.
 
     We observe that $\G$ is monotone non-decreasing
     and that $\H$ is convex.
     Furthermore, either $\H$ is monotone (either non-decreasing, or non-increasing) in $\intervallo{()}{0}{+\infty}$,
     or $\H$ is monotone non-increasing in $\intervallo{(]}{0}{\xp}$
     and monotone non-decreasing in $\intervallo{[)}{\xp}{+\infty}$ for some $\xp>0$.
     In any case, $\lim_{\x\to0^+}\H[\x]\in\intervallo{(]}{-\infty}{+\infty}$
     and $\lim_{\x\to+\infty}\G[\x]\in\intervallo{()}{-\infty}{+\infty}$ ($\G[\x]<+\infty$ thanks to \eqref{H:Far}). 
     We can then fix $\r>0$ and compute,
 
     \begin{align*}
        \frac{1}{2}\PK[\Ir]
         &=
         \unmezzo\tonde*{\intmenoinftymenor\de\y+\intrinfty\de\y}\intmenorr\K[\x-\y]\integralde\x\\
         &=
         \intmenoinftymenor\intmenorr\K[\x-\y]\integralde\x\de\y \\
         &=
         \intmenoinftymenor\tonde*{\G[\r-\y]-\G[-\r-\y]}\de\y\\
         &\overset{\hypertarget{circledast}{\circledast}}{=}
         \lim_{\M\to+\infty}\intmenoMmenor
         \odv{}{\y}\quadre*{-\H[\r-\y]+\H[-\r-\y]}
         \integralde\y\\
         &=
         -\H[2\r]+\H[0^+]
         +
         \lim_{\M\to+\infty}
         \tonde*{\H[\r+\M]-\H[-\r+\M]},
     \end{align*}
     where the second equality is a consequence of \eqref{H:Sym}, while the equality marked by $\newlink{circledast}{\circledast}$ is justified by the monotone convergence theorem,
     being $\G[\r-\y]-\G[-\r-\y]=\intmenorr\K[\x-\y]\geq0$ for any $\y\in\intervallo{()}{-\infty}{-\r}$.

     By Lagrange's mean value theorem for every $\M>0$ we can find $\crM\in\intervallo{()}{-\r+\M}{\r+\M}$ s.t.
     $\H[\r+\M]-\H[-\r+\M]=\G[\crM]2\r$.
     Therefore, 
     \begin{equation*}
         \lim_{\M\to+\infty}
         \tonde*{\H[\r+\M]-\H[-\r+\M]}=\G[+\infty]2\r,
     \end{equation*}
     which completes the proof of \eqref{eq:formulaPKinter}.
 \end{proof}
 \begin{remark}
     Under the assumptions of \cref{res:formulaPKinter},
     if $\K$ is such that $\PK[\Ir]<+\infty$ for some $\r>0$,
     then the function
     $\intervallo{()}{0}{+\infty}\ni\r\mapsto\PK[\Ir]$
     is everywhere finite and,
     taking the first derivative in \eqref{eq:formulaPKinter},
     we observe that it
     is monotone non-decreasing, concave and at least of class $\Cspace[1][\intervallo{()}{0}{+\infty}]$.
 \end{remark}
    \begin{example}
        As an application of \cref{res:formulaPKinter}, we can obtain explicit formulas
        for the the non-local perimeter of balls of some one-dimensional kernels. This is useful in order to ascertain the validity of \eqref{isoperimetricassumption}, which is 
        needed to obtain \cref{prop:poincaré,cor:continuousembedding,prop:sobolev} (see also \cref{cor:localisoperimetric}). 

        If we consider the fractional kernel $\Ks[\x]=\abs{\x}^{-1-\s}$ with $\s\in\intervallo{()}{0}{1}$ in dimension 1, 
        we can choose $\G[\x]=-\x^{-\s}/\s$ and $\H[\x]=-\x^{1 - \s}/((1 - \s) \s)$ for $\x>0$ 
        to find $\PKs[\intervallo{()}{-\r}{+\r}]=-2\H[2\r]$ for every $\r>0$. 
        This tells that \eqref{isoperimetricassumption} holds if the exponent $1/q$ equals $1-\s$; i.e., for $q=\frac{1}{1-s}$.   

        If we consider the kernel
        $\K[\x]=\abs{\x}^{-1}\tonde*{-\log\abs{\x}}^{\gamma-1}\CharFun{\intervallo{()}{-1/3}{1/3}}[\x]$,
        where $\gamma>0$, cfr. \cite[Remark 2.20]{BesSte25-MR4845992} and \cite{BruNgu20-MR3995732},
        we can choose 
            \begin{align*}
                \G[\x]&=
                \begin{cases} 
                    - \frac{\tonde{-\log{\left(\x \right)}}^{\gamma}}{\gamma} & 
                    \text{for}\: 0 < \x < \frac{1}{3} 
                    \\[4pt]-\frac{\left(\log{\left(3 \right)}\right)^{\gamma}}{\gamma}
                    & \text{for}\: \x \geq \frac{1}{3} 
                \end{cases},
                \\
                \H[\x]
                &=
                \begin{cases} 
                    - \frac{ \int_0^{\x} \tonde{-\log{\left(\t \right)}}^{\gamma}\, \de\t}{\gamma}
                    & \text{for}\: 0 < \x < \frac{1}{3} \\
                    -\frac{\left(\log{\left(3 \right)}\right)^{\gamma}}{\gamma}\tonde{\x-\frac{1}{3}} 
                    - \frac{\int_0^{1/3} \tonde{-\log{\left(\t \right)}}^{\gamma}\, \de\t}{\gamma}
                    & \text{for}\: \x \geq \frac{1}{3} 
                \end{cases},   
            \end{align*}
            so that
            $\phi[\r]\coloneqq\unmezzo\PK[\intervallo{()}{-\r/2}{+\r/2}]=\G[+\infty]\r-\H[\r]
            =
            -\frac{\left(\log{\left(3 \right)}\right)^{\gamma}}{\gamma}\r-\H[\r]$,
            for every $\r>0$.

            We first observe that $\phi$ is constant on $\intervallo{[)}{1/3}{+\infty}$,
            so that \eqref{isoperimetricassumption} cannot hold.
            This is not surprising, since $\K$ has compact support.

            Nevertheless, we can consider a variant of $\K$ whose behavior in $\intervallo{[)}{1/3}{+\infty}$
            is, for instance, of fractional type, so that $\phi$ restricted to $\intervallo{[)}{1/3}{+\infty}$ 
            is bounded from below by ${\r}^{\alpha}$ for some $\alpha\in\intervallo{(]}{0}{1}$. 
            We then deduce that $\lim_{\r\to0^+}\phi[\r]/\r=+\infty$
            and $\lim_{\r\to0^+}\phi[\r]/{\r}^{\alpha}=0$ for every $\alpha\in\intervallo{()}{0}{1}$.
            This, tells that if \eqref{isoperimetricassumption} holds,
            then necessarily the power of the volume in \eqref{isoperimetricassumption} must be $1$.
    \end{example}

\subsection{Existence of non-local isoperimetric sets}

We now discuss some results in the literature related to the
existence of minimizers of the problem $\NonlocalIsopPBKernelVolume{\K*}{\m}$.

Let $\m>0$. \cite[Theorem 5.6 and Proposition 5.4]{CesNov18-MR3732175} guarantee that if $\K\in\LoneRd$ is a 
generic kernel, then a suitable relaxation of $\NonlocalIsopPBKernelVolume{\K*}{\m}$
admits at least a solution and every solution of this relaxed problem has compact support. 
Moreover, in \cite[Theorem 5.7 and Remark 5.8]{CesNov18-MR3732175} it is shown that 
if $\K\in\LoneRd$ satisfies a further technical condition 
(which is satisfied for instance if the kernel is positive definite)
then $\NonlocalIsopPBKernelVolume{\K*}{\m}$ admits at least a solution, which is compact.
A key ingredient for these results is \cref{res:equiv_PB_PK_VK} together with
suitable first and second variations formulas.

On the other hand, if $\K$ is a sufficiently regular kernel (satisfying the assumptions of \cref{res:formulaPKinter}) and if 
$\E$ is a solution of $\NonlocalIsopPBKernelVolume{\K*}{\m}$, then it is 
a volume-constrained stationary set for the perimeter $\PK$ and as a consequence of
\cite[Theorem 6.1]{FigFusMagMilMor15-MR3322379}
the non-local curvature $\HKbdryE$ (recall \cref{def:NonlocalCurvature}) is constant on $\bdry{\E}$.

Under symmetry assumptions, \cite[Theorem 2.19]{BesSte25-MR4845992} (cfr. also \cite[Proposition 3.1]{CesNov18-MR3732175})
tells us that $\NonlocalIsopPBKernelVolume{\K*}{\m}$ is solved 
by euclidean balls of volume $\m$ (which are the unique solutions) if $\K$
is a radial and decreasing kernel with respect to the euclidean norm, i.e.
if $\K$ satisfies \textup{(}\hyperref[H:Dec]{$\mathrm{Dec(}1,\EucNorm{\cdot}\mathrm{)}$}\textup{)}.
This proof is based on the the \emph{Riesz's rearrangement inequality},
which is a crucial tool no more available once the kernel presents a different symmetry. 
In fact, in \cite{Van06-MR2245755} it is proved that Riesz's inequality fails if the rearrangement 
is performed by means of a general convex body rather than the euclidean ball.

We finally stress that this is not simply a technical issue: \textup{(}\hyperref[H:Dec]{$\mathrm{Dec(}1,\EucNorm{\cdot}\mathrm{)}$}\textup{)}
is essential to obtain the optimality of euclidean balls as \cref{res:BallsNotMinimizers} shows.
 \begin{proposition}[][res:BallsNotMinimizers]
    Let $\K$ be a kernel satisfying \eqref{H:Sym} and \eqref{H:Far}. 
    Assume that there exists $\delta>0$
    such that $\graffe*{\Kdelta>0}$ has non-empty interior, where $\Kdelta\coloneqq\CharFun{\comp{\Bdelta}}\K$.
    Then, euclidean balls of volume $\mr\coloneqq\lebd{\Br}$
    do not solve $\NonlocalIsopPBKernelVolume{\Kdelta*}{\mr}$
    for every $\r\in\intervallo{()}{0}{\delta/2}$.
 \end{proposition}
 \begin{proof}
    
    Let $\delta>0$, $\r\in\intervallo{()}{0}{\delta/2}$ and $\mr\coloneqq\lebd{\Br}$.
    Thanks to \cref{res:equiv_PB_PK_VK}, we observe that 
    $\NonlocalIsopPBKernelVolume{\Kdelta*}{\mr}$
    is equivalent to
    $\NonlocalMaxPBKernelVolume{\Kdelta*}{\mr}$,
    where $\Kdelta\coloneqq\CharFun{\comp{\Bdelta}}\K\in\LoneRd$ thanks to \eqref{H:Far}.
    
    Let $\xzero$ be a point in the interior of $\graffe*{\Kdelta>0}$
    and let $\epszero>0$ s.t. $\BallRadiusCenter{\epszero}{\xzero}\subset\graffe*{\Kdelta>0}\subset\comp{\Bdelta}$.
    Let $\OmUNO=\BallRadiusCenter{\raggioGiusto}{\xzero/2}$,
    $\OmDUE=\BallRadiusCenter{\raggioGiusto}{-\xzero/2}$
    and observe that
    $\distance\tonde*{\OmUNO,\OmDUE}=\EucNorm{\xzero}-2\raggioGiusto>\delta-2\raggioGiusto>0$.

    We fix $0<\eps<\min\graffe{\epszero/2,\raggioGiusto}$ and put $\Om\coloneqq\OmUNO\sqcup\OmDUE$, noting that $\lebd{\Om}=\mr$.
    Then,
    \begin{equation*}
        \VKdelta[\Om]=\VKdelta[\OmUNO]+\VKdelta[\OmDUE]
        +2\integral{\OmUNO}\integral{\OmDUE}\Kdelta[\x-\y]\integralde\y\de\x.
    \end{equation*}
    In particular, using the fact that for every $\x\in\palluno\subset\OmUNO$ 
    and every $\y\in\palldue\subset\OmDUE$ we have
    $\x-\y\in\BallRadiusCenter{\epszero}{\xzero}\subset\graffe*{\Kdelta>0}$,
    \begin{align*}
        \VKdelta[\Om]
        &\geq
        2\integral{\OmUNO}\integral{\OmDUE}\Kdelta[\x-\y]\integralde\y\de\x
        \\
        &\geq2\integral{\palluno}\integral{\palldue}\Kdelta[\x-\y]\integralde\y\de\x
        >
        0
        =
        \VKdelta[\Br],
    \end{align*}
    where the last equality follows from the restriction $\r\in\intervallo{()}{0}{\delta/2}$
    combined with the inclusion $\Bdelta\subset\graffe*{\Kdelta=0}$.
    This shows that $\Br$ is not a solution of $\NonlocalMaxPBKernelVolume{\Kdelta*}{\mr}$, which 
    concludes the proof.
 \end{proof}
We conclude this note addressing to a possible application of the extension result in \cref{sec:extension} and of the inequalities
previously obtained in \cref{sec:SobolevPoincare}. More precisely, inequality \eqref{PWimproved} can be employed to infer the following
non-local isoperimetric inequality relative to a regular set $\Omega$.

\begin{corollary}[][cor:localisoperimetric] 
Let $\Omega \subset \R^d$ a bounded, open set with Lipschitz boundary, and assume that \eqref{isoperimetricassumption} holds for some $q\in[1,+\infty)$. Let $K$ be a kernel satisfying \eqref{H:Nint}, \eqref{H:Dec}, \eqref{H:Dou}, and \NtsOne{}.  Then there exists a positive constant $C$ such that 
\begin{equation}\label{localisoperimetric}
    \min\{|E\cap \Omega|, |\Omega\setminus E|\}^{\frac{1}{q}} \leq C \mathcal{P}_K(E; \Omega)
\end{equation}
    for every $E\subseteq \R^d$ measurable.    
\end{corollary}
\begin{proof}
    Without loss of generality, assume $\mathcal{P}_K(E;\Omega)<+\infty$. We prove the statement under the further assumption that $\min\{|E\cap \Omega|, |\Omega\setminus E|\}=|E\cap \Omega|$, the other case being analogous. Applying Proposition \ref{prop:poincaré} with $u=\chi_E$ and noting that we are supposing $|\Omega \setminus E| \geq |\Omega|/2$, we obtain
\begin{align*}
        \mathcal{P}_K(E;\Omega) & \geq \frac{1}{2}[\chi_E]_{W^{K,1}(\Omega)} \\
        & \geq \frac{1}{2C}\Bigl(\int_{\Omega} \Bigl|\chi_E-\frac{|E\cap \Omega|}{|\Omega|}\Bigr|^q\,\de\x\Bigr)^{\frac{1}{q}} \\
        & = \frac{1}{2C} \Bigl[|E \cap \Omega|\Bigl(\frac{|\Omega \setminus E|}{|\Omega|}\Bigr)^q + |\Omega \setminus E|\Bigl(\frac{|E \cap \Omega|}{|\Omega|}\Bigr)^q \Bigr]^{\frac{1}{q}} \\
        & \geq \frac{1}{2C} |E\cap\Omega|^{\frac{1}{q}}\frac{|\Omega \setminus E|}{|\Omega|} \geq \frac{1}{4C} |E\cap\Omega|^{\frac{1}{q}},
        \end{align*}
        and this concludes the proof.
\end{proof}

\nocite{*}
\bibliographystyle{amsalpha}
\bibliography{biblio}

\end{document}